\documentclass{article}
\title{Gluing of Fourier-Mukai Partners in a Triangular Spectrum and Birational Geometry}
\author{Daigo Ito}
\renewcommand\footnotemark{}
\date{}

\usepackage{stix} 
\usepackage[english]{babel} 
\usepackage{amsmath,amsfonts,amsthm,amscd}
\usepackage{ascmac}
\usepackage{appendix}
\usepackage{bm}
\usepackage{cases}
\usepackage{changepage}
\usepackage{enumitem}
\usepackage{etoolbox}
\usepackage{float}
\usepackage{geometry}
\usepackage{graphicx}
\usepackage[utf8]{inputenc}
\usepackage [autostyle, english = american]{csquotes}
\usepackage{scalerel}
\usepackage{ytableau}
\usepackage{tikz-cd} 
\usepackage{tocloft}
\usepackage{setspace}
\usepackage[color,all]{xy}
\usepackage[cal=euler]{mathalfa}
\usepackage{url}
\usepackage{CJKutf8}
\makeatletter
\newcommand{\address}[1]{\gdef\@address{#1}}
\newcommand{\email}[1]{\gdef\@email{\url{#1}}}
\newcommand{\website}[1]{\gdef\@website{\url{#1}}}
\newcommand{\@endstuff}{\par\vspace{\baselineskip}\noindent\small
\begin{tabular}{@{}l}\scshape{Daigo Ito} \\ \scshape\@address\\\textrm{E-mail address:} \@email \\\textrm{Website:} \@website\end{tabular}}
\AtEndDocument{\@endstuff}
\makeatother
\address{Department of Mathematics, University of California, Berkeley, Evans Hall 935, CA 94720-3840}
\email{daigoi@berkeley.edu}
\website{https://daigoi.github.io/}

\usepackage[alphabetic,backrefs]{amsrefs}
\usepackage{hyperref}
\hypersetup{colorlinks,linkcolor=blue,citecolor=red}

\DeclareMathOperator{\agtt}{{agtt}}
\DeclareMathOperator{\aut}{Aut}

\DeclareMathOperator{\auteq}{Auteq}

\DeclareMathOperator{\End}{End}

\DeclareMathOperator{\FM}{FM}
\DeclareMathOperator{\fmspec}{Spec^{\mathsf{FM}}}

\DeclareMathOperator{\gtt}{{gtt}}

\let \hom \relax
\DeclareMathOperator{\hom}{Hom}

\let \im \relax
\DeclareMathOperator{\im}{Im}

\let \ker \relax
\DeclareMathOperator{\ker}{Ker}

\DeclareMathOperator{\ob}{Ob}

\DeclareMathOperator{\perf}{\mathsf{Perf}}
\DeclareMathOperator{\Perf}{\ecal{P}erf}
\DeclareMathOperator{\pfm}{pFM}

\DeclareMathOperator{\pic}{Pic}
\DeclareMathOperator{\pgtt}{pgtt}

\DeclareMathOperator{\rank}{rank}

\DeclareMathOperator{\Ser}{Ser}

\DeclareMathOperator{\sgtt}{{Sgtt}}

\DeclareMathOperator{\spc}{Spc}
\DeclareMathOperator{\spec}{Spec}

\DeclareMathOperator{\stab}{Stab}
\let \sl \relax 
\DeclareMathOperator{\sl}{SL}

\let \sp \relax
\DeclareMathOperator{\sp}{Sp}

\DeclareMathOperator{\supp}{Supp}

\DeclareMathOperator{\Th}{Th}

\newcommand {\bb}{\mathbb}
\renewcommand {\cal}{\mathcal}
\newcommand{\ecal}{\mathscr}
\renewcommand {\epsilon}{\varepsilon}

\newcommand {\bra}[1]{\langle{#1}\rangle}
\renewcommand {\l}{\left}

\renewcommand {\r}{\right}

\newcommand*{\DashedArrow}[1][]{\mathbin{\tikz [baseline=-0.25ex,-latex, dashed,#1] \draw [#1] (0pt,0.5ex) -- (1.3em,0.5ex);}}
\newcommand {\ratmap}{\DashedArrow[->,densely dashed    ]} 

\newcommand {\emp}{\emptyset}

\newcommand {\inj}{\hookrightarrow}
\newcommand {\inv}{^{-1}}
\newcommand {\iso}{\cong}

\newcommand {\surj}{\twoheadrightarrow}
\newcommand {\tens}{\otimes}

\newsavebox{\pullbacks}
\sbox\pullbacks{%
\begin{tikzpicture}%
\draw (0,0) -- (1ex,0ex);%
\draw (1ex,0ex) -- (1ex,1ex);%
\end{tikzpicture}}

\newcommand{\ag}{\mathsf{AG}}

\newcommand{\cat}{\mathsf{Cat}}

\newcommand{\coh}{\mathsf{coh}}

\newcommand{\D}{\mathsf{D}}

\newcommand{\dg}{\mathsf{dg}}

\newcommand {\fg}{\mathsf{fg}}

\newcommand{\fix}{\mathsf{fix}}
\newcommand{\fm}{\mathsf{FM}}

\newcommand{\sg}{\mathsf{SG}}
\newcommand{\h}{{\mathrm{H}}}

\newcommand{\ho}{\mathsf{Ho}}

\newcommand {\id}{{\rm id}}

\let \mod \relax
\newcommand {\mod}{\mathsf{Mod}}

\newcommand{\pg}{\mathsf{PG}}

\newcommand{\ser}{\mathsf{Ser}}
 
\let \sf \relax 
\newcommand{\sf}{\mathsf}

\let \top \relax
\newcommand {\top}{\mathsf{Top}}

\newcommand{\vect}{\mathsf{Vect}}

\newcommand{\injar}{\ar@{^(->}} 
\newcommand{\prarrow}[2]{\ar@<0.5ex>[r]^-{#1} \ar@<-0.5ex>[r]_-{#2}}
\newcommand{\plarrow}[2]{\ar@<0.5ex>[l]^-{#1} \ar@<-0.5ex>[l]_-{#2}}
\newcommand{\pdarrow}[2]{\ar@<0.5ex>[d]^-{#1} \ar@<-0.5ex>[d]_-{#2}}
\newcommand{\puarrow}[2]{\ar@<0.5ex>[u]^-{#1} \ar@<-0.5ex>[u]_-{#2}}

\theoremstyle{plain}
\newtheorem{theorems}{Theorem}[section] 
\newtheorem{claims}{Claim}[theorems]
\newtheorem{conjectures}[theorems]{Conjecture}
\newtheorem{corollaries}[theorems]{Corollary}
\newtheorem{lemmas}[theorems]{Lemma}

\newtheorem{penmdef}[claims]{Definition}

\theoremstyle{definition}
\newtheorem{constructions}[theorems]{Construction}
\newtheorem{definitions}[theorems]{Definition} 
\newtheorem{axioms}[theorems]{Axiom} 
\newtheorem{examples}[theorems]{Example}     
\newtheorem{notations}[theorems]{Notation}         

\newtheorem{question}[theorems]{Question}
\newtheorem{obss}[theorems]{Observation}
\newtheorem{penmlem}[claims]{Lemma}
\newtheorem{penmthm}[claims]{Theorem}
\newtheorem{penmcor}[claims]{Corollary}
\newtheorem{penmeg}[claims]{Example} 
\newtheorem{inclaims}[claims]{Claim}

\theoremstyle{remark}
\newtheorem{remarks}[theorems]{Remark}

\newtheorem{penmrem}[claims]{Remark}

\makeatletter
\newtheoremstyle{indented}
  {1pt}
  {1pt}
  {\addtolength{\@totalleftmargin}{1.5em}
   \addtolength{\linewidth}{-1.5em}
   \parshape 1 1.5em \linewidth}
  {}
  {\bfseries}
  {.}
  {.5em}
  {}
\makeatother

\theoremstyle{indented}
\newtheorem{pinddef}[claims]{Definition} 
\newtheorem{pindlem}[claims]{Lemma}
\newtheorem{pindthm}[claims]{Theorem}
\newtheorem{pindcor}[claims]{Corollary}
\newtheorem{pindeg}[claims]{Example} 
\newtheorem{pindrem}[claims]{Remark} 
\newtheorem{pindq}[claims]{Question} 
\newtheorem{pindc}[claims]{Conjecture} 
\newtheorem{pindclaim}[claims]{Claim}

\newenvironment{theorem}
{
	\pushQED{\qed}\begin{theorems}}
	{\popQED\end{theorems}}

\newenvironment{notation}
{
	\pushQED{\qed}\begin{notations}}
	{\popQED\end{notations}}

\newenvironment{conjecture}
{
	\pushQED{\qed}\begin{conjectures}}
	{\popQED\end{conjectures}}

\newenvironment{corollary}
{
	\pushQED{\qed}\begin{corollaries}}
	{\popQED\end{corollaries}}

\newenvironment{definition}
{
	\pushQED{\qed}\begin{definitions}}
	{\popQED\end{definitions}}

\newenvironment{lemma}
{
	\pushQED{\qed}\begin{lemmas}}
	{\popQED\end{lemmas}}

\newenvironment{remark}
{
	\pushQED{\qed}\begin{remarks}}
	{\popQED\end{remarks}}
	
\newenvironment{example}
{
	\pushQED{\qed}\begin{examples}}
	{\popQED\end{examples}}

\newenvironment{construction}
{
	\pushQED{\qed}\begin{constructions}}
	{\popQED\end{constructions}} 

\newenvironment{obs}
{
	\pushQED{\qed}\begin{obss}}
	{\popQED\end{obss}}

\newenvironment{indq}
{
	\pushQED{\qed}\begin{pindq}}
	{\popQED\end{pindq}} 
 \newenvironment{indc}
{
	\pushQED{\qed}\begin{pindc}}
	{\popQED\end{pindc}}

\newenvironment{enmlem}
{
	\pushQED{\qed} \begin{penmlem}}
	{\popQED\end{penmlem}}
	
\newenvironment{enmthm}
{
	\pushQED{\qed}\begin{penmthm}}
	{\popQED\end{penmthm}}

\newenvironment{inproof}
{\begin{proof}  }
	{\end{proof}}

\MakeOuterQuote{"}

\def\showanswers{1}

\newcommand{\hide}[1]{
\ifnum\showanswers=1

\fi

\ifnum\showanswers=0

\fi
}
\begin{document}
\maketitle
\begin{abstract}
    Balmer defined the tensor triangulated spectrum $\operatorname{Spec}_\otimes \mathcal T$ of a tensor triangulated category $(\mathcal T,\otimes)$ and showed that for a variety $X$, we have the reconstruction $X \cong \operatorname{Spec}_{\otimes_{\mathscr O_X}^\mathbb L}\operatorname{Perf} X$. In the absence of the tensor structure, Matsui recently introduced the triangular spectrum $\operatorname{Spec}_\vartriangle \mathcal T$ of a triangulated category $\mathcal T$ and showed that there exists a topological embedding $X \cong \operatorname{Spec}_{\otimes_{\mathscr O_X}^\mathbb L}\operatorname{Perf} X \subset \operatorname{Spec}_\vartriangle \operatorname{Perf} X$. In this paper, we construct a scheme $\operatorname{Spec}^\mathsf{FM} \mathcal T \subset \operatorname{Spec}_\vartriangle \mathcal T$, called the Fourier-Mukai (FM) locus, by gathering all varieties $X$ satisfying $\operatorname{Perf} X \simeq \mathcal T$. Those varieties are called FM partners of $\mathcal T$ and embedded into $\operatorname{Spec}_\vartriangle \mathcal T$ as tensor triangulated spectra. We present a variety of examples illustrating how geometric and birational properties of FM partners are reflected in the way their tensor triangulated spectra are glued in the FM locus. Finally, we compare the FM locus with other loci within the triangular spectrum admitting categorical characterizations, and in particular, make a precise conjecture about the relation of the FM locus with the Serre invariant locus.
\end{abstract}


\setcounter{tocdepth}{2}
\tableofcontents 

\section{Introduction}
\subsection*{Background}
When we think of the derived category $\perf X$ of perfect complexes on a smooth variety $X$, viewed as a triangulated category, as an algebraic invariant of the variety, it is natural to ask how much information $\perf X$ knows about $X$. In \cite{bondal_orlov_2001}, Bondal-Orlov showed that if $X$ is a smooth projective variety with ample (anti-) canonical bundle, then the triangulated category $\perf X$ knows everything about $X$, i.e., we can reconstruct $X$ solely from the triangulated category structure of $\perf X$. On the other hand, we cannot expect this type of complete reconstruction can be generalized too much since there are non-isomorphic varieties whose derived categories are triangulated equivalent. Nevertheless, in \cite{Balmer_2005} and \cite{Balmer_2010}, Balmer showed that if we in addition remember a monoidal structure on $\perf X$ given by the usual derived tensor product $\tens_{\ecal O_X}^\bb L$, then the pair $(\perf X, \tens_{\ecal O_X}^\bb L)$ can reconstruct $X$ when $X$ is a quasi-compact and quasi-separated scheme. In the proof, Balmer defined a ringed space $\spec_\tens \cal T$, called the tensor-triangular spectrum (\textbf{tt-spectrum} for short), for a triangulated category $\cal T$ with a compatible tensor product $\tens$, called a \textbf{tt-structure}. Indeed, when $X$ is a quasi-compact quasi-separated scheme, he showed that we have $X = \spec_{\tens_{\ecal O_X}^\bb L}\perf X$ as ringed spaces. We can reinterpret Balmer's work as saying that reconstruction of a variety $X$ from its derived category $\perf X$ can be done by identifying and specifying a "geometric" tt-structure on $\perf X$ and moreover saying that classification of smooth projective varieties with derived categories triangulated equivalent to $\cal T$, called \textbf{Fourier-Mukai partners} of $\cal T$, can be done by classifying "geometric" tt-structures on $\perf X$. Thus, to pursue those studies, we want to make sense of a notion of ”geometric” tt-structures and comparison among them. In this paper, we will first naively set “geometric” tt-structures to be tt-structures that are equivalent to $(\perf X, \tens_{\ecal O_X}^\bb L)$ for some smooth projective variety $X$ and later we will give some other attempts to define “geometric” tt-structures more categorically. 

For the comparison of tt-structures, we will use a ringed space $\spec_\vartriangle \cal T$ recently defined by Matsui \cite{Matsui_2021} only using the triangulated category structure of $\cal T$, which is called the \textbf{triangular spectrum} of $\cal T$. One crucial fact about this space shown by Matsui in the same paper is that given a fairly reasonable tt-structure $\tens$ on $\cal T$, we know that there is a morphism $$\spec_\tens \cal T \subset  \spec_\vartriangle \cal T$$ of ringed spaces that is a topological embedding on the underlying topological spaces. In particular, if a triangulated category $\cal T$ is triangulated equivalent to $\perf X$ for some variety $X$, then $X$ can be viewed as a subspace of $\spec_\vartriangle \cal T$ by
\[
X \iso \spec_{\tens_{\ecal O_X}^\bb L} \perf X \subset \spec_\vartriangle \cal T. 
\]
Letting $\FM \cal T$ denote the set of isomorphism classes of Fourier-Mukai partners of $\cal T$, we can rephrase this result as saying that $\spec_\vartriangle \cal T$ contains any variety $X\in \FM \cal T$ as a tt-spectrum corresponding to $(\perf X, \tens_{\ecal O_X}^\bb L)$ although $\spec_\vartriangle \cal T$  may contain multiple copies of the same Fourier-Mukai partners and those Fourier-Mukai partners can possibly intersect each other in $\spec_\vartriangle \cal T$. Here are some examples illustrating how several tt-structures on $\cal T$ interact with each other geometrically via corresponding tt-spectra in the triangular spectrum:  
\begin{example} Let $\cal T$ be a triangulated category with $X \in \FM \cal T$. 
    \begin{enumerate}
        \item If $X$ is an elliptic curve, then the triangular spectrum $\spec_\vartriangle \cal T$ is a disjoint union of infinitely many copies of $X$. (See Lemma \ref{elliptic key}.) 
        \item If $X$ is an abelian variety (resp. a surface) and $X'$ is a non-isomorphic Fourier-Mukai partner of $X$, then their copies in $\spec_\vartriangle \cal T$ do not intersect with each other. (See Corollary \ref{abel FM disjoint} (resp. Corollary \ref{surface FM disjoint}).)
        \item If $X$ is a surface containing a $(-2)$-curve, then $\spec_\vartriangle \cal T$ contains at least two copies of $X$ which intersect along the complement of the $(-2)$-curve. (See Corollary \ref{-2 twist}.)
        \item If $X$ is connected with $X' \in \FM \cal T$ via a standard/Mukai flop, then at least one pair of their copies in $\spec_\vartriangle \cal T$ intersect with each other along the complement of the flopped subvarieties. (See Example \ref{flop}.) \qedhere
    \end{enumerate}
\end{example}
\subsection*{Overview}
\subsubsection*{Fourier-Mukai loci}
Motivated by those observations, we will construct and study a subspace $\spec^\fm \cal T$ in the triangular spectrum $\spec_\vartriangle \cal T$, called the \textbf{Fourier-Mukai locus}, which is defined to be the union of all the copies of the Fourier-Mukai partners of $\cal T$ immersed into $\spec_\vartriangle \cal T$ as above. In particular, we will see one of the main results of this paper.
\begin{theorem}[Theorem \ref{scheme}]
     Let $\cal T$ be a triangulated category with $\FM \cal T \neq \emp$. Then, $\spec^{\fm}\cal T$ is naturally a smooth scheme locally of finite type, into which any $X \in \FM \cal T$ is immersed as an open subscheme. Moreover, the Krull dimension $\dim \spec^\fm \cal T$ equals the Krull dimension $\dim X$ for any $X \in \FM \cal T$. 
\end{theorem}
Indeed, the claim is non-trivial since copies of tt-spectra in the triangular spectra a priori may not glue as ringed spaces. 
\begin{notation}
   With notations as above, let $\spec_{\tens,X}\cal T$ denote the open subscheme of $\fmspec \cal T$ given by the union of all  copies of $X$ in $\spec_\vartriangle \cal T$. 
\end{notation}
Now, given the scheme structures on $\fmspec \cal T$ and $\spec_{\tens,X} \cal T$, it is interesting to ask their properties as schemes. For brevity, given a property $\cal P$ of a scheme, we will say a scheme $X$ satisfies \textbf{tt-$\cal P$} if $\spec_{\tens,X} \perf X$ satisfies $\cal P$. In particular, we will focus on tt-irreducibility and tt-separatedness of schemes. Along the way, we will observe important symmetries of $\spec_{\tens,X} \cal T$ with respect to triangulated autoequivalences of $\cal T$ and in particular show that each connected component is isomorphic to each other (Theorem \ref{irreducible component}).

In this paper, we will exhibit some specific computations of Fourier-Mukai loci to indicate their relations with geometry. Here are some interesting examples:
\begin{example}
    Let $\cal T$ be a triangulated category with a smooth projective variety $X \in \FM \cal T$. 
    \begin{enumerate}
        \item If $X$ is a smooth projective variety with ample (anti-)canonical bundle, then $\spec^\fm\cal T \iso X$ essentially by the result of Bondal-Orlov. In particular, $X$ is tt-irreducible and tt-separated. (See Corollary \ref{fm ample}.)
        \item If $X$ is an elliptic curve, then $\spec^\fm \cal T$ is a disjoint union of infinitely many copies of $X$. In particular, $X$ is tt-separated, but not tt-irreducible. (See Lemma \ref{elliptic key}.)
        \item If $X$ is a simple abelian variety, then all of its copies in $\spec^\fm \cal T$ are disjoint. In particular, $X$ is tt-separated, but not tt-irreducible in general. (See Lemma \ref{simple} and Lemma \ref{abelian not tt-irreducible}.) 
        \item If $X$ is a toric variety, then any copies of $X,Y \in \FM \cal T$ in $\spec^\fm \cal T$ intersect with each other along open sets containing tori. In particular, $X$ is tt-irreducible. However, $X$ is in general not tt-separated. (See Corollary \ref{toric} and Remark \ref{toricrem}.) 
        \item If $X$ is a surface containing a $(-2)$-curve, then $X$ is not tt-separated. Moreover, $X$ is in general not tt-irreducible either. (See Example \ref{K3}.)
        \item If $X$ is a Calabi-Yau threefold, then each irreducible component of $\spec^\fm \cal T$ containing a copy of $X$ contains all the copies of smooth projective Calabi-Yau threefolds that are birationally equivalent to $X$.   Moreover, $X$ is neither tt-separated nor tt-irreducible in general. (See Example \ref{flop} and Remark \ref{complicated birational}.)\qedhere
    \end{enumerate}
\end{example}
In particular, Fourier-Mukai loci provide novel geometric intuitions to reinterpret behaviors of varieties in birational geometry. Moreover, in Section \ref{flops and K-equiv} we will observe relations of Fourier-Mukai loci with the DK hypothesis proposed by Kawamata \cite{Kawamata2002DEquivalenceAK}, which predicts that $K$-equivalent smooth projective varieties are Fourier-Mukai partners. In particular, we will observe that some properties of Fourier-Mukai loci should be a very strong evidence for the DK hypothesis to be true. More precisely, we propose the following version of the DK hypothesis:
\begin{conjecture}[Conjecture \ref{deep dk}]
    Let $X$ be a smooth projective variety. Then, every connected component of $\fmspec \perf X$ containing a copy of $X$ contains all smooth projective varieties that are $K$-equivalent to $X$ as open subschemes.
\end{conjecture}
Note that this version of the conjecture has a topological necessary condition to be true.
\subsubsection*{Categorical characterization}
Note that studies of Fourier-Mukai loci so far are a priori geometric by definition in the sense that we define "geometric" tt-structures as tt-structures that are equivalent to $(\perf X,\tens_{\ecal O_X}^\bb L)$. In particular, many computations in examples above turn out to be applications of results from birational geometry. Hence, it is natural to try to give categorical characterizations of Fourier-Mukai loci to potentially give backward applications of studies of derived categories to birational geometry. Thus, in the latter part of this paper, we will try to give more categorical characterizations of "geometric" tt-structures and Fourier-Mukai loci. In this paper, we will focus on the underlying topological spaces and for clarity let $\spc^\fm \cal T$ (resp. $\spc_\vartriangle \cal T$) denote the underlying topological space of the ringed space $\spec^\fm \cal T$ (resp. $\spec_\vartriangle \cal T$).  

First, we will discuss comparison of the Fourier-Mukai locus with the \textbf{Serre invariant locus} $\spc^\ser \cal T$ in a triangular spectrum. Here, we should note that the underlying set of a triangular spectrum consists of some special types of triangulated subcategories and thus we can talk about Serre functor invariance of the points in the triangular spectrum. By construction it will be straightforward to see
\[
 \spc^\fm \cal T \subset \spc^\ser \cal T \subset \spc_\vartriangle \cal T. 
\]
Note that $\spc^\ser \cal T$ is defined purely triangulated categorically and in \cite{HO22}, Hirano-Ouchi gave a beautiful reinterpretation of the Bondal-Orlov (topological) reconstruction by showing $\spc^\ser \perf X \iso X$ when $X$ is a smooth projective variety with ample (anti-)canonical bundle. In this paper, we will see the following: 
\begin{theorem}[Corollary \ref{FM ample}, Theorem \ref{curves}]
    Let $\cal T$ be a triangulated category with $X \in \FM \cal T$. Suppose $X$ is a smooth projective curve or a smooth projective variety with ample (anti-)canonical bundle. Then
    \[
\spc^\fm \cal T= \spc^\ser \cal T.\qedhere
\]
\end{theorem}
Since the cases of elliptic curves and non-elliptic curves are two extremes in terms of structures of triangular spectra, we propose the following conjecture:
\begin{conjecture}[Conjecture \ref{main conjecture}]    
Let $\cal T$ be a triangulated category with $X \in \FM\cal T$. Then we have
    \[
    \spc^\fm \cal T = \spc^\ser \cal T. 
    \]
    In particular, $\dim \spc^\ser  \cal T = \dim \spc^\fm \cal T = \dim X$.
\end{conjecture}
If the conjecture holds, then it means Fourier-Mukai loci can be obtained more categorically by considering Serre invariant loci. Moreover, it is also noteworthy that we can similarly define the \textbf{proper Fourier-Mukai locus} $\spc^\sf{pFM} \cal T$ by taking smooth proper varieties into consideration. Indeed, proper Fourier-Mukai loci are more natural in a dg-categorical point of view and a priori give a better approximation to Serre invariant loci by noting the inclusions:  
\[
\spc^\fm \cal T\subset \spc^\sf{pFM} \cal T \subset \spc^\ser  \cal T. 
\]
See Remark \ref{proper dg} for more discussions. 

In the last section, we try to give less geometric characterizations of "geometric" tt-structures compared to the naive definition we have discussed. First, we will consider a tt-structure $\tens$ on $\cal T$ such that $\spec_{\tens} \cal T$ is a smooth projective variety (resp. a not necessarily irreducible smooth projective variety), called an \textbf{almost geometric tt-structure} (resp. a \textbf{pseudo-geometric tt-structure}). It turns out that both almost geometric tt-structures and pseudo-geometric tt-structures may not be geometric. Indeed, we can easily see some pseudo-geometric tt-structures are originated in representation theory, which actually gives some interesting descriptions of the whole triangular spectra in some cases (Example \ref{p1} (iii)). For another candidate, we note that given $\tens_{\ecal O_X}^\bb L$ for a smooth projective variety $X$, the Serre functor can be written as $- \tens_{\ecal O_X}^\bb L \omega_X[\dim  X]$ for the canonical bundle $\omega_X$, which is an "invertible" object in $(\perf X,\tens_{\ecal O_X}^\bb L)$. Thus, we define a \textbf{Serre-geometric tt-structure} to be a tt-structure so that the Serre functor can be written in such a form. By writing the union of tt-spectra corresponding to Serre-geometric tt-structures by $\spc^\sg \cal T$, we will easily see
\[
\spc^\fm  \cal T \subset \spc^\sf{pFM} \cal T  \subset \spc^\sg \cal T \subset \spc^\ser  \cal T. 
\]
In particular, $\spc^\sg \cal T$ is defined categorically yet a priori closer to $\spc^\fm  \cal T$ compared to the Serre invariant locus. Hence, we finish this paper by proposing a conjecture that is a priori weaker than the one above, yet more plausible from a dg-categorical point of view as a future research direction: 
\begin{conjecture}[Conjecture \ref{weaker}]
    Let $\cal T$ be a triangulated category with $\FM \cal T \neq \emp$. Then we have
    \[
    \spc^\sf{pFM} \cal T = \spc^\sg \cal T. \qedhere
    \]
\end{conjecture}
\subsection*{Acknowledgement} \ 

The author is grateful to his advisor, David Nadler, for helpful discussions and continuous supports. Additionally, the author would like to thank Paul Balmer, Sira Gratz, Yuki Hirano, Hiroki Matsui and Genki Ouchi for their insightful suggestions and helpful comments.

\section{Preliminaries}\label{preliminaries}
Throughout this paper, we will work over an algebraically closed field $k$ of characteristic $0$. In particular, a scheme and an algebra are assumed to be over $k$ and a variety is assumed to be an integral and separated scheme of finite type over $k$.   
\begin{notation} For conventions on triangulated categories and (compatible) monoidal structures, we will mainly follow \cite{sandersthesis} except for part (ii).
    \begin{enumerate}
        \item A triangulated category is assumed to be $k$-linear and \textbf{essentially small} (i.e., having a set of isomorphism classes of objects) and functors/structures are $k$-linear.
        \item When we refer to a triangulated functor $\eta$, we will not fix a natural isomorphism $\eta(-[1])\iso \eta(-)[1]$ following the convention of \cite{HuyBook}. In particular, being a triangulated functor is a property of a $k$-linear functor between triangulated categories and thus a natural isomorphism of triangulated functors means a natural isomorphism of $k$-linear functors.  
        \item Subcategories are assumed to be full and \textbf{replete} (i.e., closed under isomorphism), which is automatic for triangulated subcategories.
        \item When we refer to a (symmetric) monoidal category, we will only write its underlying category $\cal C$, the binary functor $\tens$, and possibly the identity object $\bb 1_\cal C$ and we will omit the associator, left and right unitors, and braidings from notations.
        \item When we refer to a strong monoidal functor between monoidal categories $(\cal C,\tens_\cal C,\bb 1_\cal C)$ and $(\cal D,\tens_\cal D,\bb 1_\cal D)$, we will write its underlying functor $F:\cal C \to \cal D$ and possibly the isomorphism $\epsilon_F:\bb 1_\cal D \to F(\bb 1_\cal C)$, but we will omit the natural isomorphism $\mu_F:F(-)\tens_\cal D F(-) \to F(-\tens_\cal C -)$ from notations. \qedhere
    \end{enumerate}
      
\end{notation}
\begin{definition}\ 
\begin{enumerate}
    \item A \textbf{tt-category (tensor triangulated category)} is a triangulated category $\cal T$ with a $k$-linear symmetric monoidal structure $\tens:\cal T \times \cal T \to \cal T$ which is triangulated in each variable, called a \textbf{tt-structure}. 
    \item A strong symmetric monoidal functor (resp. a strong symmetric monoidal equivalence) between tt-categories is said to be a \textbf{tt-functor} (resp. a \textbf{tt-equivalence}) if its underlying functor is a triangulated functor. \qedhere
\end{enumerate}
\end{definition}
Although tt-categories are ubiquitous in many fields (cf. e.g. \cite{balmer2020guide}), we will primarily focus on ones coming from algebraic geometry and quiver representations. Before introducing them, let us fix some notations for the latter. For example, see \cite{Elagin_2022}.
\begin{construction}
Let $ \sf{Q} $ be a quiver and let $ \sf{Q} _0$ (resp. $ \sf{Q} _1$) denote the set of vertices (resp. arrows) of $ \sf{Q} $ together with the source and the target maps denoted by $s,t: \sf{Q} _1 \to  \sf{Q} _0$. Now let $R = k  \sf{Q} $ denote the path algebra of $ \sf{Q} $ and let $M$ be a right $R$-module. For a vertex $v \in  \sf{Q} _0$, let $M_v$ denote the subspace $Me_v \subset M$, where $e_v \in R$ denotes a path of length $0$ at $v$. Note we have a decomposition  
\[
M = \oplus_{v \in  \sf{Q} _0} M_v 
\]
by $k$-vector spaces and a quiver arrow $a$ from $u$ to $v$ induces a $k$-linear map $M_v \to M_u$. Therefore, giving a right $R$-module $M$ is equivalent to giving a collection $\{M_v\mid v \in  \sf{Q} _0\}$ of $k$-vector spaces and a collection $\{M_a:M_{s(a)} \to M_{t(a)} \mid a \in  \sf{Q} _1\}$ of $k$-linear maps and this data is called a \textbf{representation} of a quiver $ \sf{Q} $. Note we can similarly define maps of representations so that they are equivalent to homomorphisms of right $R$-modules. In particular, the category of right $k  \sf{Q} $-modules and the category of representations of a quiver $ \sf{Q} $ are equivalent. 

Now, for a vertex $v \in  \sf{Q} _0$, let $P_v$ denote the right  submodule $e_vR \subset R$. Then we clearly have 
\[
k { \sf{Q}}  = \oplus_{v\in { \sf{Q}} _0} P_v. 
\]
Therefore, $P_v$ are projective and indeed it can be shown that any indecomposable projective module is isomorphic to $P_v$ for a unique vertex $v \in  \sf{Q} _0$. Moreover, the Nakayama functor $\bb S$ (cf. Example \ref{serre functor examples}) sends each $P_v$ to an indecomposable injective module $I_v:=\bb S(P_v)$ and it can be shown that any indecomposable injective module is isomorphic to $I_v$ for a unique vertex $v \in  \sf{Q} _0$. Also, for $v \in  \sf{Q} _0$, let $S_v$ denote the corresponding simple right $R$-module $k$, on which the idempotent $e_v$ acts by the identity and any other path acts by zero. 

Furthermore, if $ \sf{Q} $ is finite and has no oriented cycles, then we can order $ {\mathsf{Q} _0} = \{1,\dots, n\}$ compatibly with arrows and then it is a standard result that $\bra{P_1,\dots,P_n}$ is a full and strong exceptional collection of $\perf R$ (cf. Example \ref{example of tt cateogry} (ii)) and $R \iso \End_R(\oplus_i P_i)$. From now on, we will assume a finite quiver with no oriented cycles is ordered. 

Finally, note that given two representations $M= (\{M_v\mid v \in  \sf{Q} _0\},\{M_a:M_{s(a)} \to M_{t(a)} \mid a \in  \sf{Q} _1\})$ and $M' = (\{M_v'\mid v \in  \sf{Q} _0\},\{M_a':M_{s(a)}' \to M_{t(a)}' \mid a \in  \sf{Q} _1\})$ of a quiver $\sf Q$ (or equivalently given two right $k\sf Q$-modules $M$ and $M'$), we can define a \textbf{quiver tensor product} $M\tens_\sf{rep} M'$ to be the vertex-wise and arrow-wise tensor product
\[
(\{M_v\tens_k M_v'\mid v \in  \sf{Q} _0\},\{M_a\tens_k M'_a:M_{s(a)}\tens_k M_{s(a)} \to M_{t(a)} \tens_k  M_{t(a)} \mid a \in  \sf{Q} _1\}).
\]
Note the quiver tensor product defines a symmetric monoidal structure on the category of right $k\sf Q$-modules. 
\end{construction}
Now, the following are main examples we will deal with:
\begin{example}\label{example of tt cateogry} \ 
    \begin{enumerate}
        \item Let $X$ be a noetherian scheme and let $\D(X)$ denote the derived category of $\ecal O_X$-modules. A complex $\ecal F$ of $\ecal O_X$-modules is said to be \textbf{perfect} if it is locally quasi-isomorphic to a bounded complex of locally free $\ecal O_X$-modules of finite rank. Define $\perf X \subset \D(X)$ to be the full subcategory on perfect complexes. Then, the derived tensor product $\tens_{\ecal O_X}^\bb L$ gives a tt-structure on $\perf X$. Recall when $X$ is a smooth variety, $\perf X$ is equivalent to the bounded derived category $\D^b\coh X$ of coherent $\ecal O_X$-modules.  
        \item Let $\sf{Q}$ be a finite quiver with no oriented cycle and let $\D(k\sf{Q})$ denote the derived category of right $k\sf Q$-modules. A complex of right $k\sf Q$-modules are said to be \textbf{perfect} if it is quasi-isomorphic to a bounded complex of finitely generated projective right $k \sf Q$-module. Let $\perf k \sf Q \subset \D(k\sf Q)$ denote the full subcategory on perfect complexes. Then, the derived tensor product $\tens_{\sf{rep}}^\bb L$ gives a tt-structure on $\perf k\sf Q$. Note that since $k\sf Q$ is finite-dimensional (and hence noetherian) and has finite global dimension, we see that $\perf k\sf Q$ is equivalent to the bounded derived category $\D^b \mod^\fg k\sf Q$ of finitely generated right $k\sf Q$-modules.\qedhere 
    \end{enumerate}
\end{example}

In particular, those triangulated categories admit Serre-functors.
\begin{definition}
    Let $\cal T$ be a triangulated category with finite dimensional Hom-spaces. A \textbf{Serre functor} $\bb S : \cal T \simeq \cal T$ is a triangulated autoequivalence such that for any $F,G \in \cal T$, there is a functorial isomorphism
        \[
        \hom_\cal T (F,G) \iso \hom_\cal T(B,\bb S(A))^*,
        \]
        where $(-)^*$ denotes the dual of a vector space over $k$. It is well-known that a Serre functor, if exists, is unique up to natural isomorphism and commutes with any $k$-linear autoequivalence of $\cal T$ (e.g. \cite{HuyBook}*{Lemma 1.30}).
\end{definition}
\begin{example}\label{serre functor examples} \ 
\begin{enumerate}
    \item  If $\cal T = \perf X$ for a smooth proper variety $X$, then we have $\bb S = -\tens_{\ecal O_X} \omega_X[\dim X]$, where $\omega_X$ denotes the canonical bundle of $X$. In particular, if $X$ has the trivial canonical bundle, then $\bb S$ is given by the shift functor $[\dim X]$.
    \item If $\cal T = \perf kQ$ for a finite quiver $\sf Q$ with no oriented cycle, then a Serre functor $\bb S$ is given by the Nakayama functor, which is the composition $\bb R \hom_{k\sf Q}(-,k\sf Q) \circ \bb R \hom_k(-,k)$. \qedhere
\end{enumerate}
\end{example}
Let us recall some facts about the following simplest case:
\begin{example}\label{a2}
    Let $\sf Q$ be the $\sf A_2$-quiver $1 \to 2$ and let $R = k \sf Q$. Then it is easy to see there are only three indecomposable representations:
    \begin{itemize}
        \item $(k \leftarrow 0) \iso e_1 R = P_1 = S_1$;
        \item $(0 \leftarrow k) \iso S_2 = I_2$;
        \item $(k \overset{\sim}{\leftarrow} k) \iso e_2R = P_2 = I_1$; 
    \end{itemize}
    and thus the Nakayama functor $\bb S$ acts by $\bb S(P_1) = I_1 = P_2$, $\bb S(P_2) = I_2 = S_2$, and $\bb S(S_2) = P_1[1]$, where the last equality follows since $\bb S$ is triangulated and we have a distinguished triangle $P_1 \to P_2 \to I_2 \to P_1[1]$.      
\end{example}
Now, we are going to introduce tt-spectra. 
\begin{construction}
    Let $\Th\cal T $ denote the lattice of \textbf{thick subcategories} (i.e., triangulated subcategories that are closed under direct summand) with respect to inclusions.  For a full subcategory $\cal E$ of $\cal T$, define
    \[
    V(\cal E): = \{\cal I \in \Th\cal T  \mid \cal I \cap \cal E = \emp \}. 
    \] 
    It is easy to see we can define a topology on $\Th\cal T $ by defining closed subsets to be subsets of the form $V(\cal E)$ for some full subcategory $\cal E$ of $\cal T$. Moreover, for $F \in \cal T$, let $\bra F$ denote the smallest thick subcategory of $\cal T$ containing $F$.   
\end{construction}
\begin{definition}[\cite{Balmer_2002}]
    Let $(\cal T,\tens)$ be a tt-category. A thick subcategory $\cal I$ of $\cal T$ is said to be an \textbf{ideal} if for any $F \in \cal I$ and $G \in \cal T$, we have $F \tens G \in \cal I$ and an ideal $\cal P$ is said to be a \textbf{prime ideal} if it is a proper subcategory of $\cal T$ and $F\tens G \in \cal P$ implies $F \in \cal P$ or $G \in \cal P$. Define the \textbf{tt-spectrum} $\spc_\tens \cal T \subset \Th\cal T $ of $(\cal T,\tens)$ to be the subspace of prime ideals of $(\cal T,\tens)$.
\end{definition}
Now, let us construct a structure sheaf on a tt-spectrum. 
\begin{construction}[\cite{Balmer_2005}]\label{str sheaf of tt}
    Let $(\cal T,\tens)$ be a tt-category with unit $\bb 1$. For an open subset $U \subset \spc_\tens \cal T$, set $\cal T^U:=\cap_{\cal P\in U} \cal P$ and let $\cal T(U)$ denote the tt-category
    \[
    \cal T(U):= (\cal T/\cal T^U)^\natural
    \]
    with unit $\bb 1_U$, where $(-)^\natural$ denotes the idempotent completion. Now, define a structure sheaf $\ecal O_{\cal T,\tens}$ on $\spc_\tens \cal T$ to be the sheafification of the presheaf $\ecal O_{\cal T,\tens}^\sf{pre}$ given by
    \[
    U \mapsto \End_{\cal T(U)}(\bb 1_U). 
    \]
    Here, we need to fix some models of localization and idempotent completion to ensure well-definedness of restriction ring homomorphisms (cf. \cite{Balmer_2002}*{\href{https://www.math.ucla.edu/~balmer/Pubfile/Reconstr.pdf}{Proposition 5.3}}). Let $\spec_\tens \cal T$ denote the ringed space $(\spc_\tens \cal T, \ecal O_{\cal T,\tens})$, which we again call the \textbf{tt-spectrum} of $(\cal T, \tens)$. Note that $\spec_\tens \cal T$ is canonically a ringed space over $k$, i.e., there is a canonical morphism $\spec_\tens \cal T \to \spec k$ of ringed spaces. 
\end{construction}
Let us make the following observation for future ease. 
\begin{lemma}\label{identical tt-spectrum}
    Let $\cal T$ be a triangulated category and let $\tens$ and $\tens'$ be tt-structures with units $\bb 1$ and $\bb 1'$, respectively. Suppose for any $X,Y \in \cal T$, $X\tens Y \iso X\tens' Y$. Then, $\spc_\tens\cal T = \spc_{\tens'} \cal T \subset \Th \cal T$. If we moreover have $\bb 1 = \bb 1'$, then the structure sheaves $\ecal O_{\cal T,\tens}$ and $\ecal O_{\cal T,\tens'}$ are the same (not just naturally isomorphic) as contravariant functors from the category of open subsets of $\spc_\tens\cal T = \spc_{\tens'} \cal T$ to the category of sets.
\end{lemma}
\begin{proof}
    Let $\cal I$ be a thick subcategory of $\cal T$. Since $\cal I$ is replete, it is a prime ideal with respect to $\tens$ if and only if it is with respect to $\tens'$, i.e., $\spc_\tens \cal T = \spc_{\tens'}\cal T$. Now, if we moreover have $\bb 1 = \bb 1'$, then for any open subspace $U \subset \spc_\tens \cal T = \spc_{\tens'} \cal T$, we have $\ecal O_{\cal T,\tens}^\sf{pre}(U) = \End_{\cal T(U)}(\bb 1_U) = \End_{\cal T(U)}(\bb 1_U')  = \ecal O_{\cal T,\tens'}^\sf{pre}(U)$ and since the restriction maps do not depend on the choices of tt-structure by construction, we are done.
\end{proof}
Now, tt-spectra are functorial in the following sense.
\begin{lemma}[\cite{Balmer_2005}*{\href{https://arxiv.org/pdf/math/0409360.pdf}{Proposition 3.6}}, \cite{gallauer2019tt}*{\href{https://arxiv.org/pdf/1708.00834.pdf}{Lemma A.1}}]\label{functoriality tt}
    Let $(\eta,\epsilon_\eta):(\cal T,{\tens_{\cal T}})\to (\cal T',\tens_{\cal T'})$ be a tt-functor. Then, we have a canonical morphism of ringed spaces over $k$: Over the underlying topological space, we have
    \[
    \spc_\tens(\eta):\spc_{\tens_{\cal T'}} \cal T' \to \spc_{\tens_{\cal T}} \cal T;\quad \cal P \mapsto \eta\inv(\cal P).
    \]
    A sheaf map on the structure sheaves is given as follows. For an open subset $U \subset \spec_{\tens_\cal T} \cal T$, define a $k$-algebra map
    \[
    \End_{\cal T(U)}(\bb 1_U) \to \End_{\cal T'(\spc_\tens(\eta)\inv(U))}(\bb 1_{\spc_\tens(\eta)\inv(U)});\quad f \mapsto \epsilon_\eta\inv \circ \eta(f)\circ \epsilon_\eta
    \]
    and we can sheafify to obtain a desired map. In particular, $\spec_\tens$ defines a functor from the category of tt-categories over $k$ to the category of ringed spaces over $k$. Moreover, if we restrict $\spec_\tens$ to the category of rigid tt-categories (cf. Definition \ref{trig cat terminology}), then $\spec_\tens$ factors through the category of locally ringed spaces over $k$ (with local morphisms). 
\end{lemma}
\begin{remark}\label{subtlty of natural iso}
    Let $(\eta,\epsilon_\eta)$ and $(\eta',\epsilon_{\eta'})$ be tt-functors $(\cal T,{\tens_{\cal T}})\to (\cal T',\tens_{\cal T'})$ and suppose the underlying functors $\eta$ and $\eta'$ are naturally isomorphic by $\alpha:\eta \simeq \eta'$. Then, $\spc_\tens(\eta)=\spc_\tens(\eta')$ since triangulated subcategories are replete, but $\spec_\tens(\eta)$ and $\spec_\tens(\eta')$ may a priori differ as morphisms of ringed spaces. Note if we moreover have $\alpha_\bb 1 \circ \epsilon_\eta = \epsilon_{\eta'}$ (e.g. if $\alpha$ is a monoidal natural isomorphism), then $\spec_\tens(\eta)$ and $\spec_\tens(\eta')$ clearly induce the same morphism of ringed spaces. 
\end{remark}
On the other hand, triangular spectra are defined as follows: 
\begin{definition}[\cite{Matsui_2021}]
    Let $\cal T$ be a triangulated category. A thick subcategory $\cal P$ of $\cal T$ is said to be \textbf{prime} if there exists the smallest element in the subposet $\{\cal Q \in \Th\cal T  \mid \cal P \subsetneq \cal Q\} \subset \Th \cal T$. Define the \textbf{triangular spectrum} $\spc_\vartriangle \cal T \subset \Th \cal T $ to be the subspace of prime thick categories of $\cal T$.
\end{definition}
To define a structure sheaf, let us recall the following notions:
\begin{definition}\label{center}
     Let $\cal T$ be a triangulated category. Define the \textbf{center} $Z(\cal T)$ of a triangulated category $\cal T$ to be the commutative ring of natural transformations $\alpha: \id_\cal T \to \id_\cal T$ with $\alpha[1] = [1] \alpha$, where the multiplication is defined by composition. We say an element $\alpha \in Z(\cal T)$ is \textbf{locally nilpotent} if $\alpha_M$ is a nilpotent element in the endomorphism ring $\End_\cal T(M)$ for each $M$. Let $Z(\cal T)_\sf{lnil}$ denote the ideal of locally nilpotent elements and write $Z(\cal T)_\sf{lrad}:= Z(\cal T)/Z(\cal T)_\sf{lnil}$. 
\end{definition}
\begin{construction}\label{str sheaf of t}
Let $\cal T$ be a triangulated category. Define the structure sheaf $\ecal O_{\cal T,\vartriangle}$ on $\spc_\vartriangle \cal T$ to be the sheafification of the presheaf $\ecal O_{\cal T,\vartriangle}^\sf {pre}$ given by
\[
U \mapsto Z(\cal T(U))_\sf{lred}. 
\]
Let $\spec_\vartriangle\cal T$ denote the ringed space $(\spc_\vartriangle \cal T, \ecal O_{\cal T,\vartriangle})$. 
\end{construction}
For triangular spectra, we cannot obtain full functoriality and some cares are needed. 
\begin{definition}
    A triangulated functor $\eta:\cal T \to \cal T'$ between triangulated categories $\cal T,\cal T'$ is said to be \textbf{dense} if for any $M' \in \cal T'$ there exist $M \in \cal T$ and $N'\in \cal T'$ such that $\eta(M)\iso M'\oplus N'$. For example, essentially surjective functors and the idempotent completion functor are dense. 
\end{definition}
\begin{lemma}[\cite{matsui2023triangular}*{\href{https://arxiv.org/pdf/2301.03168.pdf}{Proposition 4.2}}]\label{functoriality t}
    Let $\eta:\cal T \to \cal T'$ be a fully faithful dense triangulated functor and fix a quasi-inverse $\eta\inv: \im \eta \to \cal T$ so that $\eta\inv(\eta(X)) = X$ for $X \in \cal T$, where $\im \eta$ denotes the essential image of $\eta$. Then, we have a canonical morphism 
    \[
    \spc_\vartriangle(\eta):\spc_{\vartriangle} \cal T' \to \spc_\vartriangle \cal T;\quad \cal P \mapsto \eta\inv(\cal P),
    \]
    where a map on structure sheaves is given as follows: for an open subset $U \subset \spc_\vartriangle \cal T$, the ring homomorphism 
    \[
     Z(\cal T'(\spc_\vartriangle(\eta)\inv(U)))_\sf{lred} \to Z(\cal T(U))_\sf{lred};\quad \alpha \mapsto \eta\inv(\alpha_{\eta_{(-)}})
    \] 
    is an isomorphism and we take its inverse. 
\end{lemma}
To state relations of tt-spectra and triangular spectra, let us recall the following notions:
\begin{definition}\label{trig cat terminology}
    A tt-category $(\cal T,\tens)$ with unit $\bb 1$ is said to be:
    \begin{enumerate}
        \item \textbf{monogenic} (resp. \textbf{locally monogenic}) if $\bra{\bb 1} = \cal T$ (resp. if for any prime ideal $\cal P \in \spc_\tens \cal T$, there exists a quasi-compact open neighborhood $U \subset \spc_\tens \cal T$ such that $\bra{\bb 1_U} = \cal T(U)$);
        \item \textbf{rigid} if the functor $-\tens M$ admits a right adjoint $[M,-]$ for any object $M \in \cal T$ and if the canonical map $[M,\bb 1] \tens N \to [M,N]$ is an isomorphism for any object $M,N \in \cal T$.\qedhere
    \end{enumerate}
\end{definition}
\begin{lemma}[\cite{matsui2023triangular}*{\href{https://arxiv.org/pdf/2301.03168.pdf}{Theorem 1.2}}]\label{noetherian hypothesis}
    Let $(\cal T,\tens)$ be an idempotent complete, rigid, and locally monogenic tt-category and suppose $\spc_\tens \cal T$ is a noetherian topological space. Then, if a thick subcategory $\cal P$ is a prime ideal, then it is a prime thick subcategory. Moreover, the inclusion $\spc_\tens \cal T \subset \spc_\vartriangle \cal T$ is an embedding of topological spaces, i.e., it is a homeomorphism onto its image.   
\end{lemma}
\begin{remark}
    The supposition of the lemma above is satisfied for $(\perf X, \tens_{\ecal O_X}^\bb L)$ for a noetherian scheme (cf. \cite{matsui2023triangular}*{\href{https://arxiv.org/pdf/2301.03168.pdf}{Corollary 3.4}}). Note this particular case, which is crucially used in this paper, is also directly shown in the published paper \cite{Matsui_2021}*{\href{https://arxiv.org/pdf/2102.11317.pdf}{Corollary 3.8 (1)}}. On the other hand, $(\perf k\sf Q,\tens_{\sf{rep}})$ is not rigid in general (e.g. \cite{sandersthesis}*{\href{https://escholarship.org/content/qt47p1p5v2/qt47p1p5v2.pdf}{Example 5.6.13}}), so we cannot apply the lemma.   
\end{remark}
Now, the following are some examples of computations of those spectra. 
\begin{theorem}[\cite{Balmer_2010}*{\href{https://www.math.ucla.edu/~balmer/Pubfile/TTG.pdf}{Theorem 54}}\label{balmer}]
    Let $X$ be a quasi-compact and quasi-separated scheme. Then, there is an isomorphism 
    \[
    \cal S_X:X \overset{\sim}{\to} \spec_{\tens_{\ecal O_X}^\bb L} \perf X 
    \]
    of ringed spaces, where the continuous map on the underlying topological spaces is given by
    \[
    x\mapsto \cal S_X(x):= \{ \ecal F \in \perf X \mid x \not \in \supp \ecal F\}.
    \]
    Here, for $\ecal F \in \perf X$ we define the \textbf{support} of $\ecal F$ to be
    \[
    \supp \ecal F :=\{x \in X \mid \text{$\ecal F_x \not \iso 0$ in $\perf \ecal O_{X,x}$}\}.\qedhere
    \]
\end{theorem}
\begin{theorem}[\cite{SirLiu13}*{\href{https://www.pure.ed.ac.uk/ws/portalfiles/portal/16965040/Recovering_Quivers_from_Derived_Quiver_Representations.pdf}{(2.1.5.1), (2.2.4.1)}}]\label{quiver tt}
    Let $ Q $ be a finite quiver without oriented cycles. Then,
    \[
    \spec_{\tens_{\sf{rep}}^\bb L} \perf k Q  \iso \bigsqcup_{v \in  Q _0} \spec k
    \]
    as ringed spaces, where each $\spec k$ corresponds to the thick subcategory $\bra{P_v}$ for $v \in  Q _0$. 
\end{theorem}
\begin{theorem}[\cite{matsui2023triangular}*{\href{https://arxiv.org/pdf/2301.03168v1.pdf}{Theorem 1.2}}]\label{matsui}
    Let $X$ be a reduced projective scheme. Then, there is an open immersion morphism
    \[
    \cal S_X^\vartriangle: X \to \spec_\vartriangle \perf X 
    \]
    of ringed spaces, whose underlying continuous map agrees with the one in Theorem \ref{balmer}. In particular, it is an embedding of topological spaces with image $\spc_{\tens_{\ecal O_X}^\bb L} \perf X$. Moreover, it is an open immersion of ringed spaces if the image $\spc_{\tens_{\ecal O_X}^\bb L} \perf X$ is open in $\spc_\vartriangle \perf X$. 
\end{theorem}
Now, let us see some examples to get some sense of those definitions. We will come back to those examples later. 
\begin{example}\label{basic example} \ 
    \begin{enumerate}
        \item \cite{matsui2023triangular}*{\href{https://arxiv.org/pdf/2301.03168v1.pdf}{Corollary 4.9}} Let $X = \bb P^1_k$. Then we have
        \[
        \spec_\vartriangle \perf \bb P^1_k \iso \bb P^1_k \sqcup \bigsqcup_{i \in \bb Z} \spec k
        \]
        as ringed spaces, where the underlying topological space of each $\spec k$ corresponds to a prime thick subcategory of the form $\bra{\ecal O_{\bb P^1_k}(i)}$ for $i \in \bb Z$. 
        \item \cite{matsui2023triangular}*{\href{https://arxiv.org/pdf/2301.03168v1.pdf}{Corollary 4.10}} Let $E$ be an elliptic curve. Then we have
        \[
        \spec_\vartriangle \perf E \iso E \sqcup \bigsqcup _{(r,d) \in I} E_{r,d}
        \]
        as ringed spaces, where $I := \{(r,d) \in \bb Z \mid r> 0, \gcd(r,d) =1\}$ and $E_{r,d}$ is a copy of $E$ for each $(r,d) \in I$. 
        \item \cite{GrSt_2023}*{\href{https://arxiv.org/pdf/2205.13356v2.pdf}{Example 5.1.3}} Let $R = k\sf A_2$ be the path algebra of the $\sf A_2$-quiver. Then we have
        \[
        \spec_\vartriangle \perf R \iso  \spec k \sqcup \spec k \sqcup \spec k
        \]
        as ringed spaces, where the underlying topological space of each $\spec k$ corresponds to points $\bra{P_0}$, $\bra{P_1}$, $\bra{S_1}$, respectively. Here, the structure sheaf $\ecal O_{R,\vartriangle}$ of $\spec_\vartriangle \perf R$ can be computed as follows: First of all, since $\spec_\vartriangle \perf R$ is discrete, it suffice to find $\ecal O_{R,\vartriangle}(\bra{P_0})$, $\ecal O_{R,\vartriangle}(\bra{P_1})$, and $\ecal O_{R,\vartriangle}(\bra{S_1})$, which are isomorphic to the stalks of $\ecal O_{R,\vartriangle}$ at the corresponding points. On the other hand, we know
        \[
        \ecal O_{R,\vartriangle}^\sf{pre}(\bra{P_0}) = \ecal O_{R,\vartriangle}^\sf{pre}(\bra{P_1}) =\ecal O_{R,\vartriangle}^\sf{pre}(\bra{S_1})=Z(\perf k)_\sf{lrad} = k.
        \]
        Since stalks are preserved under sheafification, we obtain the desired structure sheaf $\ecal O_{R,\vartriangle}$. \qedhere
    
    \end{enumerate}
\end{example}
Let us also define some groups that will be useful in the following sections. 
\begin{notation}\label{auttens}\ 
\begin{enumerate}
    \item  Let $(\cal T,\tens)$ be a tt-category. Let $\aut (\spc_\vartriangle \cal T)$ (resp. $\aut (\spc_\tens \cal T)$) denote the group of automorphisms of $\spc_\vartriangle \cal T$ (resp. $\spc_\tens \cal T$) as a topological space and let $\aut (\spec_\vartriangle \cal T)$ (resp. $\aut (\spec_\tens \cal T)$) denote the group of automorphisms of $\spec_\vartriangle \cal T$ (resp. $\spec_\tens \cal T$) as a ringed space.     
    \item For a triangulated category $\cal T$, let $\auteq \cal T$ denote the group of natural isomorphism classes of triangulated autoequivalences of $\cal T$. 
    \item Let $(\cal T, \tens)$ be a tt-category. Let $\auteq(\cal T,\tens) \subset \auteq \cal T$ denote the subgroup of autoequivalences naturally isomorphic to ones that can be equipped with structures of tt-autoequivalences. \textbf{By abuse of notations, we will say elements in $\auteq(\cal T,\tens)$ are tt-autoequivalences of $(\cal T, \tens)$ unless doing so would lead to serious confusions.} Moreover, let $\pic(\cal T,\tens) \subset \auteq \cal T$ denote the submonoid of autoequivalences naturally isomorphic to $-\tens K$ for some $K \in \cal T$.  \qedhere
\end{enumerate}
\end{notation}
\begin{remark}
    For a tt-category $(\cal T,\tens)$, it is also reasonable and possibly more reasonable from a categorical viewpoint to set $\auteq(\cal T, \tens)$ to be the group of monoidal natural isomorphism classes of tt-autoequivalences of $(\cal T,\tens)$. The reason why we instead decided to use Notation \ref{auttens} is that when $\cal T \simeq \perf X$ for a variety $X$, which is the main object of focus in this paper, we have some good understandings of $\auteq \cal T$ (cf. Example \ref{ample}) and moreover in this case tt-geometry only sees natural isomorphism classes of the underlying functors of tt-equivalences (cf. Lemma \ref{key classical algebraic geometry}). 
\end{remark}
Let us see $\pic(\cal T,\tens)$ is indeed a group.
\begin{definition}[\cite{BaFa_2006}]
    Let $(\cal T, \tens)$ be a tt-category and let $\mathbb 1$ denote its unit. An object $F \in \cal T$ is said to be \textbf{$\tens$-invertible} if there exists $G \in \cal T$ such that $F \tens G \iso \bb 1$. 
\end{definition}
\begin{lemma}\label{pic}
    Let $(\cal T, \tens)$ be a tt-category and take $\eta \in \auteq \cal T$. Now, if $\eta = - \tens K \in \pic(\cal T,\tens)$ for some $K \in \cal T$, then $K$ is $\tens$-invertible. In particular, $\pic(\cal T,\tens)$ is a group isomorphic to the group of isomorphism classes of $\tens$-invertible objects with multiplication given by the tt-structure $\tens$.  
\end{lemma}
\begin{proof}
    Let $\xi$ denote a quasi-inverse of $\eta$. Then, it is easy to see $\xi(\bb 1)$ is an inverse of $K$, where $\bb 1$ denotes the unit of $(\cal T,\tens)$. 
\end{proof}
Moreover, we have the following justification for the notation. 
\begin{lemma}[\cite{BaFa_2006}*{\href{https://arxiv.org/pdf/math/0605094.pdf}{Proposition 6.4}}]
Let $X$ be a scheme. Then, there is a split short exact sequence:
\[
0 \to \pic X \to \pic (\perf X,\tens_{\ecal O_X}^\bb L) \to C(X;\bb Z) \to 0 
\]
of abelian groups, where $ C(X;\bb Z)$ denotes the abelian group of locally constant functions from $X$ to $\bb Z$. In particular, if $X$ is connected, then $\pic (\perf X,\tens_{\ecal O_X}^\bb L) \iso \pic (X) \times \bb Z[1]$.    
\end{lemma}
Finally, let us consider some group actions on spectra. At the topological level, we have the following:
\begin{lemma}
    Let $(\cal T, \tens)$ be a tt-category. Then, we have a group action $\rho^\vartriangle: \auteq \cal T \to \aut(\spc_\vartriangle \cal T)$ given by 
    \[
    \rho^\vartriangle: \auteq \cal T \ni \tau \mapsto (\spc_\vartriangle(\tau)\inv:\cal P \mapsto \tau(\cal P)) \in  \aut(\spc_\vartriangle \cal T),
    \]
     which restricts to a group action $\rho^\tens_\top: \auteq(\cal T, \tens) \to \aut(\spc_\tens \cal T)$. Moreover, the restriction $\pic(\cal T,\tens) \to \aut (\spc_\tens \cal T)$ of $\rho^\vartriangle$ is the trivial action. 
\end{lemma}
\begin{proof}
Note a triangulated equivalence (resp. a tt-equivalence) induces an isomorphism of triangular spectra (resp. tt-spectra) by Lemma \ref{functoriality t} (resp. by Lemma \ref{functoriality tt} and Remark \ref{subtlty of natural iso}). Now, the action is well-defined since triangulated subcategories are replete and hence the action does not depend on choices of representatives in $\auteq \cal T$. The last claim is trivial by the definition of ideals.     
\end{proof}
\begin{notation}\label{rho trig notation}
    Let $\cal T$ be a triangulated category and let $\tau \in \auteq \cal T$. By slight abuse of notation, let $\tau$ also denote the (restrictions of) isomorphism $\rho^\vartriangle:\spc_\vartriangle \cal T \to \spc_\vartriangle \cal T$, which maps $\cal P$ to $\tau(\cal P)$. 
\end{notation}
When we consider ringed space structures of tt-spectra, we should be more careful due to Remark \ref{subtlty of natural iso}. The following lemma from classical algebraic geometry (e.g. \cite{Har77}*{Proposition I.3.5}) ensures that we do not need to be too careful about this subtlety in this paper. 
\begin{lemma}\label{class alg result}
    Let $X$ and $Y$ be reduced scheme locally of finite type over an algebraically closed field $k$ and let $f,g:X\to Y$ be two morphisms of $k$-schemes. If $f$ and $g$ agree on the set of closed points, then $f = g$.  
\end{lemma} 
As a direct corollary, we get the following:
\begin{lemma}\label{key classical algebraic geometry}
    Let $(\eta_1,\epsilon_1),(\eta_2,\epsilon_2)$ be tt-equivalences $(\cal T, \tens_{\cal T}) \overset{\sim}{\to} (\cal T',\tens_{\cal T'})$ and suppose that $\spec_{\tens_\cal T}\cal T$ and $\spec_{\tens_{\cal T'}} \cal T'$ are isomorphic to reduced schemes locally of finite type over $k$. Then, if the morphisms
    \[
\spec_\tens(\eta_1),\spec_\tens(\eta_2):\spec_{\tens_\cal T}\cal T \to \spec_{\tens_{\cal T'}} \cal T'
    \]
     agree on closed points, then $\spec_\tens(\eta_1) = \spec_\tens(\eta_2)$ as morphisms of locally ringed spaces over $k$. In particular, if the underlying functors $\eta_1$ and $\eta_2$ are (not necessarily monoidally) naturally isomorphic, then we have $\spec_\tens(\eta_1) = \spec_\tens(\eta_2)$ as morphisms of locally ringed spaces.
\end{lemma}
\begin{proof}
    The first part directly follows from Lemma \ref{class alg result}. The latter part follows since if $\eta_1$ and $\eta_2$ are naturally isomorphic, then they send a prime $\tens$-ideal to the same prime $\tens$-ideal. 
\end{proof}
In particular, we have the following construction, which justifies the definition of $\auteq (\cal T, \tens)$ in this paper. 
\begin{construction}\label{action on locally ringed}
    Let $(\cal T, \tens)$ be a tt-category with $\spec_\tens \cal T$ being a reduced scheme locally of finite type over $k$. Let us construct a group action 
    \[
    \rho^\tens:\auteq(\cal T,\tens) \to \aut(\spec_\tens \cal T),
    \]
    which agrees with $\rho_\top^\tens:\auteq(\cal T,\tens) \to \aut(\spc_\tens \cal T)$ on the underlying topological space. Take $\eta \in \auteq(\cal T, \tens)$ and take a tt-autoequivalence $(\zeta, \epsilon_\zeta)$ with $\eta \iso \zeta$. Then, define $\rho^\tens(\eta) = \spec_\tens(\zeta)\inv$, which does not depend on the choice of $(\zeta,\epsilon_\zeta)$ by Lemma \ref{key classical algebraic geometry}. Again, by Lemma \ref{key classical algebraic geometry}, the map $\rho^\tens$ is a group homomorphism and in particular, for $\eta \in \auteq(\cal T,\tens)$, we have $\rho^\tens(\eta \inv) = \rho^\tens(\eta)\inv$. Clearly, $\rho^\tens(\eta)$ agrees with $\rho^\tens_\top(\eta)$ on the underlying topological space.   
\end{construction}
\begin{notation}
    Let $(\cal T, \tens)$ be a tt-category with $\spec_\tens \cal T$ being a reduced scheme locally of finite type over $k$ and take $\eta \in \auteq (\cal T,\tens)$. By slight abuse of notation, let $\eta$ also denote the isomorphism $\rho^\tens(\eta):\spec_\tens \cal T \to \spec_\tens \cal T$, which maps $\cal P$ to $\eta(\cal P)$. 
\end{notation}

\section{The tt-spectrum of a Fourier-Mukai partner and actions by autoequivalences}\label{tt-spectra of FM partners}
Since Lemma \ref{noetherian hypothesis} in particular tells us that for a smooth projective variety $X$, the corresponding triangular spectra $\spc_\vartriangle \perf X$ contains all the Fourier-Mukai partners of $X$ (see Notation \ref{fm defintion}) as subspaces given as tt-spectra of some tt-structure on $\perf X$ (not necessarily just $\tens_{\ecal O_X}^\bb L$!), it is natural to ask how those Fourier-Mukai partners interact with each other in $\spc_\vartriangle \perf X$. To approach the question, we should be careful that for a Fourier-Mukai partner $Y$ of $X$, there could be multiple tt-structures $\tens$ on $\perf X$ with $(\perf X,\tens) \simeq (\perf Y,\tens_{\ecal O_{Y}}^\bb L)$ and that the embedding of $Y \iso \spc_\tens \perf X$ as a subspace of $\spc_\vartriangle \perf X$ given by Lemma \ref{noetherian hypothesis} depends on a choice of such tt-structures even if the corresponding tt-spectra are all isomorphic to $Y$. In this section, we will focus on a locus in a triangular spectrum that corresponds to a different choices of tt-structures coming from a single Fourier-Mukai partner and in the next section we will consider tt-structures coming from all the Fourier-Mukai partners. First, let us introduce some notions.       
\begin{notation}\label{fm defintion}
    Let $\cal T$ be a triangulated category. Let $\FM \cal T $ denote the set of isomorphisms classes of smooth projective varieties whose derived categories are triangulated equivalence to $\cal T$, called \textbf{Fourier-Mukai partners} of $\cal T$. For a smooth projective variety $X$, let $\FM(X) : = \FM (\perf X)$ denote the isomorphism classed of Fourier-Mukai partners of $X$. Given a triangulated equivalence $\eta: \cal T \simeq \cal T'$ and a subgroup $G \subset \auteq \cal T$, write $G_\eta := \eta \cdot G \cdot \eta\inv \subset \auteq \cal T'$.
\end{notation}
First, let us make the following observation.
\begin{construction}\label{full data}
    Let $\cal T$ be a triangulated category with $X \in \FM \cal T $ and fix a triangulated equivalence $\eta:\perf X \simeq \cal T$ and a quasi-inverse $\zeta$ together with natural isomorphisms $\alpha:\zeta \circ \eta \iso \id_{\perf X}$ and $\beta:\eta \circ \zeta \iso \id_\cal T$. Let $\tens_{X,\eta,\zeta,\alpha,\beta}$ denote a tt-structure on $\cal T$ transported from the tt-category $(\perf X,\tens_{\ecal O_X}^\bb L)$ under $\eta$, $\zeta$, $\alpha$, and $\beta$. In particular, for objects $M,N \in \cal T$, define $M \tens_{X,\eta,\zeta,\alpha,\beta} N : = \eta(\zeta(M)\tens_{\ecal O_X}^\bb L \zeta(N))$, and let $\spec_{\tens_{X,\eta,\zeta,\alpha,\beta}}\cal T$ denote the corresponding tt-spectrum. Note that since we have a canonical tt-equivalence $(\perf X, \tens_{\ecal O_X}^\bb L) \simeq (\cal T,\tens_{X,\eta,\zeta,\alpha,\beta})$, we have a canonical isomorphisms $\spec_{\tens_{X,\eta,\zeta,\alpha,\beta}} \cal T \iso \spec_{\ecal O_X^\bb L}\perf X \iso X$. 
\end{construction}
The preceding construction indeed does not depend on several choices.
\begin{lemma}\label{not depending on choice}
    Let $\cal T$ be a triangulated category with $X \in \FM \cal T$ and take $(\eta,\zeta,\alpha,\beta)$ and $(\eta',\zeta',\alpha',\beta')$ as in Construction \ref{full data}. 
    \begin{enumerate}
        \item If $\eta$ and $\eta'$ are naturally isomorphic, then $$\spc_{\tens_{X,\eta,\zeta,\alpha,\beta}} \cal T = \spc_{\tens_{X,\eta',\zeta',\alpha',\beta'}} \cal T \subset \spc_\vartriangle \cal T,$$ which will be denoted by $\spc_{\tens_{X,\eta}} \cal T$. By abuse of notation, for (a natural isomorphism class represented by) $\eta:\perf X \simeq \cal T$, we will omit $\zeta,\alpha,\beta$ from notations and considerations and write $\tens_{X,\eta}$ instead of $\tens_{X,\eta,\zeta,\alpha,\beta}$ when we are only interested in topological structures.  
        \item If $(\eta,\zeta)$ and $(\eta',\zeta')$ are identical (not just naturally isomorphic), then 
        \[
        \spec_{\tens_{X,\eta,\zeta,\alpha,\beta}}\cal T  = \spec_{\tens_{X,\eta,\zeta,\alpha',\beta'}} \cal T,
        \]
        which will be denoted by $\spec_{\tens_{X,\eta,\zeta}} \cal T$. Here, the equality means they are identical, i.e., $\spc_{\tens_{X,\eta,\zeta,\alpha,\beta}} \cal T = \spc_{\tens_{X,\eta,\zeta,\alpha',\beta'}} \cal T \subset \spc _\vartriangle \cal T$ and their structure sheaves are the same functors. By abuse of notation, we will omit $\alpha,\beta$ from notations and considerations and write $\tens_{X,\eta,\zeta}$ instead of $\tens_{X,\eta,\zeta,\alpha,\beta}$.  
    \end{enumerate}
    In particular, we will omit natural isomorphisms $\alpha$ and $\beta$ from notations and considerations from now on. 
\end{lemma}
\begin{proof}
    Part (i) follows from Lemma \ref{identical tt-spectrum}. For part (ii), $\spc_{\tens_{X,\eta,\zeta,\alpha,\beta}} \cal T  = \spc_{\tens_{X,\eta,\zeta,\alpha',\beta'}} \cal T \subset \spc _\vartriangle \cal T$ by part (i) and the structure sheaves are identical again by Lemma
    \ref{identical tt-spectrum}. 
\end{proof}
\begin{remark}\label{remark on choice}
    It will turn out that even for the second case in the preceding lemma we only need to pay attention to natural isomorphism classes of equivalences $\perf X \simeq \cal T$ in most situations in the end (cf. Notation \ref{tt-spectra of FM}). 
\end{remark}
Now, the following is the main construction in this section.
\begin{construction}\label{specx} Let $\cal T$ be a triangulated category with $X \in \FM \cal T$. By slight abuse of notation, define the \textbf{{tt-spectrum} of a Fourier-Mukai partner} $X$ of $\cal T$ to be
  \[
  \spc_{\tens,X} \cal T := \bigcup_{\eta:\perf X \simeq \cal T} \spc_{\tens_{X,\eta}}\cal T  \subset \spc_\vartriangle \cal T,
  \]
  where the union is taken over all natural isomorphism classes of triangulated equivalences $\perf X \simeq \cal T$ noting Lemma \ref{not depending on choice} (i) and \textbf{the topology is given as the one generated by open subsets of each tt-spectrum $\spc_{\tens_{X,\eta}}\cal T$, which is a priori finer than the subspace topology}.\footnote{It turns out that these two topologies indeed agree (cf. \cite{ito2024new}*{\href{https://arxiv.org/pdf/2405.16776}{Proposition 4.2.}}).} Here, also note although each $\spc_{\tens,\eta} \cal T$ is isomorphic (to $X$), it can give a different subspace of $\spc_\vartriangle \cal T$ depending on a choice $\eta$ of natural isomorphism classes. 
  
  Now, fix a natural isomorphism class of triangulated equivalences $\eta:\perf X \simeq \cal T$. Note that another choice $\eta':\perf X\simeq \cal T$ of a natural isomorphism class bijectively corresponds to an element $\tau \in \auteq\cal T $ with $\eta' \iso \tau\circ \eta$ and $\tau$ induces a homeomorphism  $\spc_{\tens_{X,\eta}}\cal T  \iso \spc_{\tens_{X,\tau\circ \eta}}\cal T = \spc_{\tens_{X,\eta'}}\cal T $, where it may be true that $\spc_{\tens_{X,\eta}}\cal T  \neq \spc_{\tens_{X,\tau\circ \eta}}\cal T $ as subspaces of $\spc_\vartriangle \cal T$. Now, note we can also write
  \[
  \spc_{\tens,X} \cal T = \bigcup_{\tau \in \auteq\cal T } \spc_{\tens_{X,\tau \circ \eta}}\cal T  \subset \spc_\vartriangle \cal T.
  \]
  Note another choice $X' \in \FM\cal T $ with $X \iso X'$ will just amount to another choice $\eta'$ instead of $\eta$ in the formula right above and hence we have $\spc_{\tens,X}\cal T = \spc_{\tens,{X'}}\cal T$. Thus, $\spc_{\tens,X} \cal T$ does not depend on a choice of representatives $X$ of an isomorphism class in $\FM \cal T$. Moreover, note the action $\auteq \cal T  \to \aut(\spc_\vartriangle\cal T)$ restricts to 
  \[
  \rho^X_\top: \auteq \cal T  \to \aut(\spc_{\tens,X} \cal T).
  \]
  
  Geometrically, $\spc_{\tens,X}\cal T$ is a not necessarily finite nor disjoint union of copies of $X$ and the action of $\auteq \cal T$ isomorphically maps one copy of $X$ onto another copy or to itself. To observe symmetry with respect to restrictions of the action to a subgroup $G \subset \auteq \cal T $, let us define a \textbf{$G$-orbit} of $\eta:\perf X\simeq \cal T$ to be
  \[
  \spc_{\tens,X}^{G\cdot \eta} \cal T := \bigcup_{\tau \in G} \spc_{\tens_{X,\tau \circ \eta}}\cal T  = G\cdot \spc_{\tens_{X,\eta}}\cal T \subset \spc_{\tens,X} \cal T,
  \]
  which may depend on the initial choice $\spc_{\tens_{X,\eta}}\cal T$ of a copy of $X$ in $\spc_\vartriangle \cal T$, i.e., the choice of $\eta$.  
\end{construction}

\begin{remark} \ 
    \begin{enumerate}
        \item In the next section, we will see a tt-spectrum of a Fourier-Mukai partner has a natural scheme structure as a union of tt-spectra of tt-structures (Corollary \ref{tt-spectra of FM}), but in general, the union is not finite nor disjoint. 
        \item We cannot expect the correspondence $X \mapsto \spc_{\tens,X}\perf X$ defines a functor at least in a straightforward manner since a triangulated functor does not induce a reasonable group homomorphism between groups of autoequivalences in general. For example, it defines a functor if we consider the category of smooth projective varieties with morphisms defined to be ones inducing equivalences on derived categories.   \qedhere
    \end{enumerate} 
\end{remark}
The following is a natural question to which we will come back later several times:
\begin{question}\label{tte}
     Let $\cal T$ be a triangulated category with $X,Y \in \FM \cal T$. Then, is it true that $X \iso Y$ if and only if $$\spc_{\tens,X} \cal T = \spc_{\tens,Y} \cal T\subset \spc_\vartriangle \cal T?$$  
     For example, it is true if $\spc_{\tens,X} \cal T$ is a disjoint union of copies of $X$ and $\spc_{\tens,Y} \cal T$ is a disjoint union of copies of $Y$ (e.g. Example \ref{basic example} (ii) or Lemma \ref{simple}). See Remark \ref{tte ae}, Remark \ref{tt-seprated is tte} and Observation \ref{bg obs} for more discussions.
 \end{question}
In the rest of this section, we will observe symmetries of tt-spectra of Fourier-Mukai partners with respect to group actions. First, let us introduce some notions based on behaviors of $\spc^{G\cdot \eta}_{\tens,X} \cal T$.
 \begin{definition} \label{generator}
     Let $\cal T$ be a triangulated category with $X \in \FM\cal T $. We say a subgroup $G \subset \auteq \cal T$ is:
     \begin{enumerate}
         \item an \textbf{$\eta$-stabilizer} if $\spc^{G\cdot \eta}_{\tens,X} \cal T = \spc_{\tens_{X,\eta}}\cal T \iso X$ for a triangulated equivalence $\eta: \perf X \simeq \cal T$;
         \item a \textbf{universal stabilizer} of $X$ if $G$ is an $\eta$-stabilizer for any triangulated equivalence $\eta: \perf X \simeq \cal T$;
         \item an \textbf{$\eta$-generator} if $\spc^{G\cdot \eta}_{\tens,X}\cal T = \spc_{\tens,X}\cal T$ for a triangulated equivalence $\eta: \perf X \simeq \cal T$;
         \item a \textbf{universal generator} of $X$ if $G$ is an $\eta$-generator for any triangulated equivalence $\eta: \perf X \simeq \cal T$.
     \end{enumerate}
     An element of an $\eta$-stabilizer (resp. a universal stabilizer) will also be called an $\eta$-stabilizer (resp. a universal stabilizer). 
 \end{definition}
We can formulate maximum symmetries of tt-spectra of Fourier-Mukai partners as follows:  
 \begin{construction}\label{stabilizers}
     Let $\cal T$ be a triangulated category with $X \in \FM\cal T $. Let us define the following subgroups of $\auteq \cal T$, where $\rho^X_\top: \auteq \cal T\to \aut \spc_{\tens,X}\cal T$ denotes the natural action defined in Construction \ref{specx}.
     \begin{enumerate}
         \item For a triangulated equivalence $\eta: \perf X \simeq \cal T$, let 
         \[\stab(\cal T,\eta):= ({\rho^X_\top})\inv ( \{f\in \aut \spc_{\tens,X}\cal T\mid  f(\spc_{\tens_{X,\eta}}\cal T) = \spc_{\tens_{X,\eta}}\cal T\})\]
         denote the \textbf{maximum $\eta$-stabilizer}. 
         \item Let
         \[
         \stab(\cal T,X):= ({\rho^X_\top})\inv ( \{f\in \aut \spc_{\tens,X}\cal T\mid  f(\spc_{\tens_{X,\eta}}\cal T) = \spc_{\tens_{X,\eta}}\cal T \text{ for all $\eta:\perf X\simeq \cal T$}\})
         \]
         denote the \textbf{maximum universal stabilizer} of $X$.
         \item Let 
         \[
         \stab^\fix(\cal T,X) := ({\rho^X_\top})\inv (\{\id_{\spc_{\tens,{X}}\cal T}\})
         \]
         denote the stabilizer group on $\spc_{\tens,{X}}\cal T$.
     \end{enumerate}
     We clearly have the inclusions $\stab^\fix(\cal T,X) \subset \stab(\cal T,X) \subset \stab(\cal T,\eta)$. 
 \end{construction}
 \begin{remark}\label{tte ae}
 Note to give a positive answer to Question \ref{tte}, it suffices to show that $X$ and $Y$ are isomorphic if and only if $\stab(\cal T,X) = \stab(\cal T,Y)$ since if $\spc_{\tens,X} \cal T = \spc_{\tens,Y}\cal T$, then $\stab(\cal T,X) = \stab(\cal T,Y)$.  
 \end{remark}
Now, we have the following immediate consequences:
 \begin{lemma}
     Let $\cal T$ be a triangulated category with $X \in \FM\cal T $ and let $G\subset \auteq \cal T$ be a subgroup. 
     \begin{enumerate}
         \item The subgroup $\bb Z[1]\subset \stab^\fix(\cal T,X) \subset \auteq \cal T$ generated by shifts is a universal stabilizer of $X$ (for any $X \in \FM \cal T$). 
         \item If $G$ is an $\eta$-stabilizer (resp. a universal stabilizer of $X$), then so is any subgroup $H\subset G$. In particular, $G$ is an $\eta$-stabilizer (resp. a universal stabilizer of $X$) if and only if $G$ is a subgroup of $\stab(\cal T,\eta)$ (resp. $\stab(\cal T,X)$).    
         \item The entire group $\auteq \cal T$ is a universal generator of $X$.
         \item If $G$ is a universal generator of $X$  (resp. an $\eta$-generator), so is any supergroup $H \supset G$. \qedhere
     \end{enumerate}
 \end{lemma}
 \begin{proof}
     Part (i) and part (iii) are obvious. For part (ii), if $G$ is an $\eta$-stabilizer and we have a subgroup $H \subset G$, then we have 
     \[
     \spc_{\tens,\eta} \cal T \subset \spc_{\tens,X}^{H \cdot \eta}\cal T \subset \spc_{\tens,X}^{G \cdot \eta}\cal T = \spc_{\tens,\eta} \cal T,
     \]
     so $H$ is an $\eta$-stabilizer. The rest of the claims follow similarly. 
 \end{proof}
 \begin{remark} \ 
 \begin{enumerate}
     \item A minimal $\eta$-generator, if exists, is a priori not unique. See Example \ref{ample} for a case when it is unique and see Remark \ref{minimal generator for elliptic curves?} for a case when we possibly have multiple minimal $\eta$-generators. See also Remark \ref{minimal generator candidate} for a possible strategy to find such a generator.
     \item Note $\eta$-stabilizers can be described purely in the language of tt-geometry since specifying a triangulated equivalence $\eta: \perf X \simeq \cal T$ is equivalent to specifying a tt-structure on $\cal T$ that is tt-equivalent to $(\perf X,\tens_{\ecal O_X}^\bb L)$. In other words, $\eta$-stabilizers describe "local" symmetries of the single tt-spectrum $\spc_{\tens_{X,\eta}}\cal T \subset \spc_{\tens,X}\cal T$. On the other hand, universal stabilizers describe global symmetries of the whole tt-spectrum $\spc_{\tens,X}\cal T$, which cannot be seen solely from tt-geometry. Indeed, for any triangulated category $\cal T$ with $X \in \fm\cal T $, there exists a nontrivial (simultaneous) universal stabilizer for any $X \in \FM \cal T$ given by a Serre functor (cf. Lemma \ref{universal recovery}).\qedhere
 \end{enumerate}    
 \end{remark}
 First, let us see some examples of $\eta$-stabilizers.  
 \begin{lemma}
     Let $\cal T$ be triangulated category with $X \in \FM\cal T $ and fix a triangulated equivalence $\eta:\perf X \simeq \cal T$. If $\tau \in \auteq(\cal T,\tens_{X,\eta})$ or $\tau \in \pic(\cal T,\tens_{X,\eta})$ (cf. Notation \ref{auttens}), then we have
     \[
     \spc_{\tens_{X,\eta}}\cal T = \tau(\spc_{\tens_{X,\eta}}\cal T) = \spc_{\tens_{X,\tau\circ \eta}} \cal T \subset \spc_\vartriangle \cal T
     \]
     (cf. Notation \ref{rho trig notation}). In particular, any subgroup $G \subset  \auteq(\cal T,\tens_{X,\eta}) \ltimes \pic(\cal T,\tens_{X,\eta})$ is an $\eta$-stabilizer.
 \end{lemma}
 \begin{proof}
     First, if $\tau \in \auteq(\cal T,\tens_{X,\eta})$ (and hence $\tau \inv \in \auteq (\cal T,\tens_{X,\eta})$), then for any $M,N \in \cal T$ we have $$M \tens_{X,\tau\circ\eta} N \iso \tau\circ\eta((\tau\circ\eta)\inv M \tens_{\ecal O_X^\bb L} (\tau\circ\eta)\inv N) \iso \tau(\tau\inv (M)\tens_{X,\eta}\tau\inv(N))\iso M \tens_{X,\eta} N$$ and hence $\spc_{\tens_{X,\eta}}\cal T = \spc_{\tens_{X,\tau\circ \eta}} \cal T$ by Lemma \ref{identical tt-spectrum}. If $\tau \in \pic (\cal T,\tens_{X,\eta})$, then it preserves (prime) ideals (as they are  replete) and hence it indeed induces the identity map $\spc_{\tens_{X,\eta}}\cal T = \spc_{\tens_{X,\tau\circ \eta}} \cal T$.
 \end{proof}
 There is a standard example of such a subgroup:
\begin{example}\label{standard autoeq}
Let $X$ be a smooth projective variety. First, recall that there are three standard types of autoequivalences of $\perf X$, namely, shifts, derived push-forwards by automorphisms of $X$ and usual tensor products with line bundles on $X$ (as line bundles are flat). An autoequivalence is said to be a \textbf{standard autoequivalence} if it is generated by those three types, or more precisely, if it is an element in the following subgroup:
 \[
 A(X) := \aut X \ltimes \pic X \times \bb Z[1] \subset \auteq (\perf X). 
 \]  
Here, we use semi-direct product since for any $f\in \aut(X)$ and $\ecal L \in \pic (X)$, we have $$\bb Rf_*\circ (\ecal L\tens_{\ecal O_X} -)\circ (\bb Rf_*)\inv \iso \bb R f_* (\ecal L\tens_{\ecal O_X} \bb Lf^*( -)) \iso \bb Rf_*\ecal L \tens _{\ecal O_X}- \in \pic (X)$$ by $(\bb R f_*)\inv \iso \bb Lf^*$ and the projection formula. Noting $A(X) \subset \auteq(\perf X,\tens_{\ecal O_X}^\bb L) \ltimes \pic(\perf X,\tens_{\ecal O_X}^\bb L)$, we in particular see that if there is a triangulated equivalence $\eta:\perf X \simeq \cal T$, then the corresponding subgroup $A(X)_\eta \subset \auteq(\cal T,\tens_{X,\eta}) \ltimes \pic(\cal T,\tens_{X,\eta})$ is an $\eta$-stabilizer, but generally not a universal stabilizer of $X$. 
\end{example}
When $\cal T$ is good enough, those constructions indeed give the maximum universal stabilizer of $X$. 
 \begin{corollary}
     Let $\cal T$ be a triangulated category with $X \in \FM\cal T $. Suppose that 
     \[\auteq \perf X  = \auteq(\perf X,\tens_{\ecal O_X}^\bb L) \ltimes \pic(\perf X,\tens_{\ecal O_X}^\bb L).\] 
     Then, $\auteq\cal T $ is a universal stabilizer (and hence the maximum universal stabilizer) of $X$, i.e., 
     \[
     \spc_{\tens,X} \cal T  \iso X. \qedhere
     \]
 \end{corollary}
 \begin{proof}
      Since any triangulated equivalence $\eta:\perf X\simeq \cal T$ induces a tt-equivalence of $(\perf X,\tens_{\ecal O_X}^\bb L)$ and $(\cal T,\tens_{X,\eta})$, we have $\auteq\cal T  = \auteq(\cal T,\tens_{X,\eta}) \ltimes \pic(\cal T,\tens_{X,\eta})$ for any $\eta$.
 \end{proof}
 \begin{example}\label{ample}
     By a theorem of Bondal-Orlov (\cite{bondal_orlov_2001}*{\href{https://arxiv.org/pdf/alg-geom/9712029.pdf}{Theorem 3.1}}), we have $\auteq \perf X = A(X)$ for a smooth projective variety $X$ with ample (anti-)canonical bundle (e.g. Fano varieties), so in such a case, the supposition of the corollary holds and for any $\eta:\perf X \simeq \cal T$ we have
\[
\spc_{\tens,{X}} \cal T =  \spc_{\tens_{X,\eta}}\cal T  \iso X. 
\]
Moreover, in this case, $\FM\cal T$ is a singleton so there is no choice for $X \in \FM\cal T $. Also, the minimal $\eta$-generator is the trivial group. 
 \end{example}
One natural question is the following:
 \begin{question}
     Is there a weaker condition on $X$ that ensures $\auteq \cal T$ is a universal stabilizer?
 \end{question}
For any triangulated category $\cal T$ with $X\in \FM\cal T $, we have an interesting universal stabilizer of $X$.
\begin{notation}
    Let $\cal T$ be a triangulated category with a Serre functor. Define $\Ser \cal T \subset \auteq \cal T$ to be the subgroup generated by Serre functors. Note $\Ser \cal T$ is generated by a single isomorphism class by the uniqueness of Serre functors. 
\end{notation}
Essentially, by Example \ref{serre functor examples} (ii), we have the following:
\begin{lemma}\label{universal recovery}
    Let $\cal T$ be a triangulated category with $\FM\cal T  \neq \emp$. Then, for any $X \in \FM \cal T$, we have \[\Ser\cal T \subset \stab^\fix(\cal T,X) \subset \stab(\cal T, X). \qedhere\]  
\end{lemma}
\begin{proof}
Let $\bb S$ denote a Serre functor of $\cal T$ and take $X \in \FM\cal T$ and a triangulated equivalence $\eta:\perf X \simeq \cal T$. Note any point in $\spc_{\tens_{X,\eta}}\cal T$ is of the form $\eta(\cal S_X(x))$ for some (not necessarily closed) point $x \in X$. Now, by the uniqueness of Serre functors and the commutativity with equivalences, we have $\bb S (\eta (\cal S_X(x))) = \eta (\bb S (\cal S_X(x))) = \eta (\omega_X[\dim X] \tens_{\ecal O_X}^\bb L \cal S_X(x)) = \eta (\cal S_X(x))$. Thus, $\bb S \in \stab^\fix(\cal T,X)$ for any $X \in \FM \cal T$.
\end{proof}
\begin{remark}
    With the same notation as above, we can also see $\Ser\cal T  \subset \pic (\cal T,\tens_{X,\eta})$ for any $X \in \FM\cal T$ and any triangulated equivalence $\eta:\perf X \simeq \cal T$, which gives another way of viewing the proof. 
\end{remark}
\begin{obs}
    Let $\cal T$ be a triangulated category with $X \in \FM \cal T$ and $\eta: \perf X \simeq \cal T$. We have the following chain of inclusions:
    \begin{center}
        \begin{tikzcd}
{\bb Z[1] \times \Ser \cal T} \arrow[rr, hook] &  & { \stab^\fix(\cal T,X)} \arrow[rr, hook] &  & { \stab(\cal T,X)} \arrow[rd, hook]                                                  &                       \\
                                               &  &                                          &  &                                                                                      & { \stab(\cal T,\eta)} \\
                                               &  & A(X)_\eta \arrow[rr, hook]               &  & {\pic(\cal T,\tens_{X,\eta}) \ltimes \auteq (\cal T,\tens_{X,\eta})} \arrow[ru, hook] &                      
\end{tikzcd}
    \end{center}
    It is natural to ask if there is any inclusion that is indeed the equality. 
\end{obs}
We moreover have the following natural questions, which we will not focus on in this paper, but may be something interesting to study in the future. 
 \begin{question} \ 
     \begin{enumerate}
         \item How can we determine or characterize maximum stabilizers? This question seems to be related to positivity of the (anti-)canonical bundle of a smooth projective variety $X$. For example, the maximum universal stabilizer of $X$ is the whole group if $X$ has the ample (anti-)canonical bundle and philosophically it gets closer to $\Ser\cal T \times \bb Z[1]$ as the canonical bundle gets closer to the trivial bundle (cf. Corollary \ref{new tt-separated}). 
         \item We can ask the same question for minimal generators, if exist. Note minimal generators are not necessarily unique. Moreover, it is interesting to ask when they are finitely generated. See also Remark \ref{minimal generator for elliptic curves?} for some discussions in the case of elliptic curves.
     \end{enumerate}
 \end{question}
\section{Fourier-Mukai loci and birational equivalence}\label{Fourier-Mukai loci and birational equivalence}
\subsection{Construction}
In this subsection, we finally take the union of tt-spectra of all Fourier-Mukai partners and observe their interactions.  
\begin{definition}[]
    Let $\cal T$ be a triangulated category. The \textbf{Fourier-Mukai locus} of $\spc_\vartriangle \cal T$ is defined to be
    \[
    \spc^\fm  \cal T := \bigcup_{X \in \FM \cal T } \spc_{\tens,X} \cal T \subset \spc_\vartriangle \cal T 
    \]
    \textbf{equipped with the topology generated by open subsets of each tt-spectrum}.\footnote{Again, this topology is a priori finer than the subspace topology in $\spc_\vartriangle \cal T$, but these topologies in fact agree by \cite{ito2024new}*{\href{https://arxiv.org/pdf/2405.16776}{Proposition 4.2.}}. Nevertheless, the rest of this section and the next section does not depend on this fact. On the other hand, in Section \ref{Comparison with Serre invariant loci} and later, we consider the subspace topology on $\spc^\fm \cal T$, so to avoid confusions, readers can safely remember these two topologies agree.}
    A subgroup $G \subset \auteq \cal T$ is said to be a \textbf{universal stabilizer} (resp. a \textbf{universal generator}) of $\cal T$ if it is a universal stabilizer (resp. a universal generator) of $X$ for any $X \in \FM \cal T$. 
\end{definition}

\begin{example}
    Let $\cal T$ be a triangulated category with $\FM \cal T\neq \emp$. The subgroup $\Ser \cal T \times \bb Z[1] \subset \auteq \cal T$ is a universal stabilizer of $\cal T$. 
\end{example}
It is natural to ask how intersections of tt-spectra look in Fourier-Mukai loci. It indeed turns out that they are open and we can moreover show Fourier-Mukai loci have natural scheme structures. First, let us note the following notion and result from \cite{Voet_2020}.
\begin{definition}\label{birational FM partners}
    Let $X$ and $Y$ be smooth projective varieties. A triangulated equivalence $\Phi_{\ecal P}:\perf X \to \perf Y$ with Fourier-Mukai kernel $\ecal P \in \perf(X\times Y)$ is said to be \textbf{birational} if there is an open (dense) subset $U \subset X$ such that $\ecal P|_{U\times Y}$ is isomorphic to the structure sheaf of the graph of an open immersion, in which case $X$ and $Y$ are said to be \textbf{birational Fourier-Mukai partners}. Note in particular that birational Fourier-Mukai partners $X$ and $Y$ are \textbf{birationally equivalent}, i.e., there exists a birational map $X \ratmap Y$. 
\end{definition}

\begin{remark}
    Note that birationally equivalent Fourier-Mukai partners may not be birational Fourier-Mukai partners. See Observation \ref{bg obs}.  
\end{remark}

As pointed out in the proof of \cite{Voet_2020}*{\href{https://digital.lib.washington.edu/researchworks/bitstream/handle/1773/46506/Voet_washington_0250E_22022.pdf?sequence=1}{Proposition 4.1.6}}, we have the following as a corollary of \cite{HuyBook}*{Corollary 6.14}:
\begin{corollary}\label{improved h}
Let $X$ and $Y$ be a smooth projective variety with $\Phi:\perf X \simeq \perf Y$ and suppose there exist closed points $x_0 \in X$ and $y_0 \in Y$ such that $\Phi(k(x_0)) = k(y_0)$. Then, there exist an open neighborhood $U$ of $x_0$ and an open immersion $f: U \inj Y$ with $f(x_0) = y_0$ such that $\Phi (k(x)) \iso k(f(x))$ for all closed points. In particular, $\Phi$ is birational if and only if such closed points $x_0 \in X$ and $y_0 \in Y$ exist. Moreover, if $\Phi$ and $\Phi'$ are naturally isomorphic, then $\Phi$ is birational if and only if $\Phi'$ is birational. 
\end{corollary}
For applications to our situations, let us make several observations. First of all, we note that an open immersion in Corollary \ref{improved h} can be canonically constructed.
\begin{construction}\label{canonical open}
Let $X$ and $Y$ be smooth projective varieties (or more generally Gorenstein projective varieties) and let $\Phi:\perf X \simeq \perf Y$ be a triangulated equivalence. In \cite{Cal17}, it was shown that if $\Phi$ is a birational equivalence, then $X$ and $Y$ are canonically birationally equivalent in the following sense, where we will only see a rough sketch of the arguments. First, let $\cal M_X$ denote Inaba's moduli space of simple objects of $\perf X$, which is an algebraic space (\cite{inaba2002toward}). Then there is a canonical open immersion $X \inj \cal M_X$ sending each point $x$ to the corresponding skyscraper sheaf $k(x)$ (\cite{Cal17}*{\href{https://johncalab.github.io/stuff/derbir-journal.pdf}{Corollary 3.2}}). Now, it is moreover shown that $\Phi$ induces a canonical open immersion $\cal M_X \inj \cal M_Y$ sending $k(x_0)$ to $k(y_0)$ if and only if $\Phi(k(x_0)) \iso k(y_0)$  (\cite{Cal17}*{\href{https://johncalab.github.io/stuff/derbir-journal.pdf}{Proposition 3.3}}).  Therefore, we obtain the following pull-back diagram of open immersions:
\DisableQuotes
\begin{center}
\begin{tikzcd}
U \arrow[rr, "\iota_Y", hook] \arrow[d, "\iota_X"', hook] &                          & Y \arrow[d, hook] \\
X \arrow[r, hook]     \arrow[rru, "\phi_\Phi", dashed]                                    & \cal M_X \arrow[r, hook] & \cal M_Y         
\end{tikzcd}
\end{center}
Now, if $\Phi$ is moreover birational, then the open subscheme $U$ is not empty and hence we canonically obtain a birational map
\[
\phi_\Phi:X \ratmap Y; \quad \iota_X(U) \overset{\iota_X\inv}{\to} U \overset{\iota_Y}{\to}\iota_Y(U).
\]
For this reason, $\phi_\Phi$ will be said to be the \textbf{birational map associated to $\Phi$} and the image $\iota_X(U)$ of $U$ in $X$ will be said to be the \textbf{(maximal) domain of definition} of $\Phi$. Note that closed points of the domain of definition of $\Phi$ consist precisely of closed points $x\in X$ with $\Phi(k(x))\iso k(y)$ for some $y \in Y$.   
\end{construction}
Note the following immediate corollary. 
\begin{lemma}\label{cocycle}
    Let $X,Y,Z$ be smooth projective varieties.
    \begin{enumerate}
        \item If birational equivalences $\Phi,\Phi':\perf X \simeq \perf Y$ are naturally isomorphic, then $\Phi$ and $\Phi'$ have the same domain of definition and moreover $\phi_\Phi = \phi_{\Phi'}$ as birational maps.
        \item If $\Phi:\perf X \simeq \perf Y$ and $\Psi:\perf Y \simeq \perf Z$ are birational equivalences, then $\phi_{\Psi} \circ \phi_{\Phi} = \phi_{\Psi \circ \Phi}$ as birational maps.  \qedhere
    \end{enumerate}
\end{lemma}
\begin{proof} \ 
\begin{enumerate}
    \item The first claim is clear. The second claim follows by Lemma \ref{class alg result}.
    \item Follows by part (i) and by comparing the following diagram with Construction \ref{canonical open}:
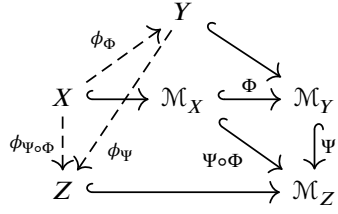
    \begin{center}
        \DisableQuotes
        \begin{tikzcd}
                                                                                              & Y \arrow[rd, hook] \arrow[ldd, "\phi_\Psi", dashed, near end]                    &                                  \\
X \arrow[r, hook] \arrow[ru, "\phi_\Phi", dashed] \arrow[d, "\phi_{\Psi\circ \Phi}"', dashed] & \cal{M}_X \arrow[r, "\Phi", hook] \arrow[rd, "\Psi \circ \Phi"', hook] & \cal M_Y \arrow[d, "\Psi", hook] \\
Z \arrow[rr, hook]                                                                            &                                                                        & \cal M_Z                        
\end{tikzcd}
    \end{center}\qedhere
\end{enumerate}
     
\end{proof}
\begin{remark} Use the same notation as in Construction \ref{canonical open}. 
\begin{enumerate}
    \item The maximal domain of definition of $\Phi$ does not necessarily agree with the maximal domain of definition of $\phi_\Phi$, i.e., the maximal open subset of $X$ to which $\phi_\Phi$ can be extended, which is well-defined by the reduced-to-separated lemma. For example, see Lemma \ref{birational spherical} and Example \ref{K3}. 
    \item We can construct the map more explicitly. Indeed, under the same notations as in Corollary \ref{improved h}, we can directly get an open immersion by setting a sheaf homomorphism to respect canonical ring isomorphisms 
    \[
    \Gamma(V,\ecal O_{V})\iso \End_{\ecal O_V}(\ecal O_V) \iso Z(\perf V)_{\sf{lred}} \overset{}{\iso} Z(\perf f(V))_{\sf{lred}}\iso \End_{\ecal O_{f(V)}}(\ecal O_{f(V)}) \iso \Gamma(f(V),\ecal O_{f(V)})
    \]
    for any open subset $V \subset U$ by \cite{rouquier_2010}*{\href{https://www.math.ucla.edu/~rouquier/papers/leeds.pdf}{Proposition 4.14}} and the proof of \cite{rouquier_2010}*{\href{https://www.math.ucla.edu/~rouquier/papers/leeds.pdf}{Theorem 4.19}} (cf. Definition \ref{center} for notations), where the middle isomorphism is given as in Lemma \ref{functoriality t}. By Lemma \ref{class alg result}, this map needs to glue to the map constructed in Construction \ref{canonical open}. \qedhere
\end{enumerate}
\end{remark}
Let us also note the following observations:
\begin{lemma}\cite{HO22}*{\href{https://arxiv.org/pdf/2112.13486v3.pdf}{Lemma 3.2, 3.3}}\label{intersec ref}
    Let $X$ and $Y$ be smooth varieties and suppose that we have a fully faithful functor $\Phi:\perf X \inj \perf Y$ that admits a right adjoint functor. Then, for $x\in X$ and $y \in Y$, we have $\Phi(\cal S_X(x)) = \cal S_Y(y)$ if and only if $\Phi(k(x)) \iso k(y)[l]$ for some $l \in \bb Z$. In particular, if $\Phi$ is a triangulated equivalence and $\supp \Phi(k(x)) = Y$ for any closed point $x \in X$, then we have
    \[
    \Phi(\spc_{\tens_{\ecal O_X}^\bb L} \perf X) \cap \spc_{\tens_{\ecal O_Y}^\bb L}\perf Y = \emp. \qedhere
    \]
\end{lemma}
\begin{corollary}\label{flat kernel}
    Suppose the kernel $\ecal P \in \perf (X\times Y)$ of a Fourier-Mukai equivalence $\Phi_\ecal P: \perf X \simeq \perf Y$ is a locally free sheaf on $X \times Y$. Then,  
    \[
    \Phi(\spc_{\tens_{\ecal O_X}^\bb L} \perf X) \cap \spc_{\tens_{\ecal O_Y}^\bb L}\perf Y = \emp. \qedhere
    \]
\end{corollary}
\begin{proof}
    Since $\ecal P$ is flat over $X$, we see that $\Phi_\ecal P k(x) = \ecal P|_{\{x\}\times Y}$ for any closed point $x \in X$ by \cite{HuyBook}*{Example 5.4}, which is locally free and hence $\supp \Phi_\ecal P(k(x)) = Y$.
\end{proof}
The following is the key lemma:
\begin{lemma}\label{main intersection}
     Let $\cal T$ be a triangulated category with $X_i \in \FM \cal T$ and $\eta_i:\perf X_i \simeq \cal T$ with quasi-inverses $\eta_i\inv$ for $i =1,2$. Set $\Phi = \eta_2\inv \circ \eta_1:\perf X_1 \simeq \perf X_2$. 
     \begin{enumerate}
         \item We have $U := \spc_{\tens_{X_1,\eta_1}} \cal T \cap \spc_{\tens_{X_2,\eta_2}} \cal T \neq \emp$ if and only if $\Phi$ is birational up to shift. Therefore, $X_1$ and $X_2$ are birational Fourier-Mukai partners if and only if $\spc_{\tens,{X_1}} \cal T \cap \spc_{\tens,{X_2}} \cal T \neq \emp$.
         \item Suppose $\Phi$ is a birational equivalence. Then, the maximal domain of definition of $\Phi$ is $\eta_1\inv(U)$. Therefore, the birational map $\phi_\Phi$ associated to $\Phi$ is an isomorphism 
         \[
         \phi_{\Phi}: \spec_{\tens_{\ecal O_{X_1}}^\bb L} \perf X_1 \supset \eta_1\inv(U) \overset{\iso}{\to} \eta_2\inv(U) \subset \spec_{\tens_{\ecal O_{X_2}}^\bb L} \perf X_2.
         \]
         Note we get the following canonical automorphism of $U$ (as a ringed space):
         \[
        \eta_2 \circ \phi_\Phi \circ\eta_1\inv : \spec_{\tens_{X_1,\eta_1,\eta_1\inv}} \cal T\supset  U \to \eta_1 \inv(U) \to \eta_2\inv(U) \to U \subset \spec_{\tens_{X_2,\eta_2,\eta_2\inv}} \cal T. \qedhere
         \]
     \end{enumerate} 
\end{lemma}
\begin{proof}
    The first part follows from Lemma \ref{intersec ref} and Corollary \ref{improved h}, noting that $\eta_1$ induces an isomorphism of
    \[
    \spec_{\tens_{\ecal O_{X_1}}^\bb L} \perf X_1 \cap \eta_1\inv \circ\eta_2 (\spec_{\tens_{\ecal O_{X_2}}^\bb L}\perf X_2) \iso U
    \]
    and that the shift functors induce the identity on tt-spectra. For the second part, note again by Corollary \ref{improved h}, $U$ is an open subscheme of $\spec_{\tens_{X_1,\eta_1}} \cal T$ and hence $\eta_1\inv(U)$ is an open subscheme of $\spec_{\tens_{\ecal O_{X_1}}^\bb L} \perf X_1$. Now, by Construction \ref{canonical open} and Lemma \ref{intersec ref}, we can see that closed points of $\eta_1\inv(U)$ coincide with closed points of the domain of definition of $\Phi$, which is also an open subscheme of $\spec_{\tens_{\ecal O_{X_1}}^\bb L} \perf X_1$. Therefore, they need to coincide, noting $\spec_{\tens_{\ecal O_{X_1}}^\bb L} \perf X_1$ is a variety. 
\end{proof}
Now, we can show one of the main results in this paper.
 \begin{theorem}\label{scheme}
     Let $\cal T$ be a triangulated category with $\FM \cal T \neq \emp$. Then, $\spc^{\fm}\cal T$ can be equipped with a locally ringed space structure so that it admits an open covering by locally ringed open subspaces of the form $\spec_{\tens_{X,\eta, \eta\inv}}\cal T$ for each $X \in \FM \cal T$ and $\eta:\perf X \simeq \cal T$ with quasi-inverse $\eta \inv$. Let $\spec^\fm \cal T$ denote the locally ringed space. In particular, $\spec^{\fm}\cal T$ is a smooth scheme locally of finite type over $k$. Moreover, the Krull dimension $\dim \spec^\fm \cal T$ equals the Krull dimension $\dim X$ for any $X \in \FM \cal T$. 
\end{theorem}
\begin{proof}
    Take $X_i \in \FM \cal T$ and $\eta_i:\perf X_i \simeq \cal T$ with quasi-inverses $\eta_i\inv$ for $i =1,2$. It suffices to think of the case when $\Phi:= \eta_2 \circ \eta_1\inv$ is birational up to shift so that
    \[
    \spc_\vartriangle \cal T \supset U := \spc_{\tens_{X_1,\eta_1,\eta_1\inv}} \cal T \cap \spc_{\tens_{X_2,\eta_2,\eta_2\inv}} \cal T \neq \emp
    \]
    since otherwise the corresponding tt-spectra are disjoint by Lemma \ref{main intersection} (i). We may furthermore assume $\Phi$ is birational by shifting since shifts induces the identity in each spectrum (by Lemma \ref{class alg result}) and commutes with equivalences. Then, we can glue $\spec_{\tens_{X_1,\eta_1,\eta_1\inv}} \cal T$ and $\spec_{\tens_{X_2,\eta_2,\eta_2\inv}} \cal T$ to a scheme along $U$ by the canonical automorphism $\eta_2 \circ \phi_\Phi \circ \eta_1\inv$ given in Lemma \ref{main intersection} (ii). Note those canonical automorphisms satisfy the cocycle conditions over any triple intersection since the birational maps associated to the corresponding birational equivalences satisfy the cocycle conditions by Lemma \ref{cocycle} (ii). Therefore, we see locally ringed spaces of the form $\spec_{\tens_{X,\eta,\eta\inv}}\cal T$ can be glued to give a desired locally ringed space structure on $\spec^\fm \cal T$.   

    Now, since $\spec^\fm \cal T$ admits an open cover by spaces of the form $\spec_{\tens_{X,\eta}} \iso X$ and each $X \in \FM \cal T$ has same dimension $n$ (e.g. \cite{Kawamata2002DEquivalenceAK}*{Theorem 1.4}), we see that
    \[
    \dim \spec^\fm \cal T = \sup_{X \in \FM \cal T} \dim X = n. \qedhere
    \]
\end{proof}
\begin{remark}
    Should a reader get confused with the proof or want more geometric interpretation, the following interpretation may help: We may think of $\spec^\fm \cal T$ as a scheme obtained by gluing Fourier-Mukai partners $X$ and $Y$ of $\cal T$ by using the birational maps associated to birational equivalences (and their inverses) between $\perf X$ and $\perf Y$. The underlying topological space "happens to be" homeomorphic to the Fourier-Mukai locus $\spc^\fm \cal T$ whose topology indeed agrees with the subspace topology in $\spc_\vartriangle \cal T$ by \cite{ito2024new}*{\href{https://arxiv.org/pdf/2405.16776}{Proposition 4.2.}}.   
\end{remark}
\begin{corollary}\label{fm ample}
    Let $\cal T$ be a triangulated category with $X \in \FM \cal T$ with ample (anti-)canonical bundle. Then as schemes, we have
    \[
    \fmspec \cal T \iso X.\qedhere
    \]
\end{corollary}
\begin{proof}
    This follows from the fact that $\FM\cal T $ is a singleton and Example \ref{ample}.  
\end{proof}
\begin{notation}\label{tt-spectra of FM}
    Let $\cal T$ be a triangulated category with $X \in \FM \cal T$. Let $G\subset \auteq \cal T$ be a subgroup and take a triangulated equivalence $\eta: \perf X \simeq \cal T$. Note that $\spc_{\tens,X} \cal T$, $\spc_{\tens_{X,\eta}} \cal T$, and $\spc^{G\cdot \eta}_{\tens,X}\cal T$ (cf. Construction \ref{specx}) can be naturally equipped with locally ringed space structures as canonical open subschemes of $\spec^\fm\cal T$. Let $\spec _{\tens,X}\cal T$, $\spec_{\tens_{X,\eta}} \cal T$ and $\spec_{\tens,X}^{G\cdot \eta}\cal T$ denote the corresponding schemes. Note that for any $\eta':\perf X \simeq \cal T$ naturally isomorphic to $\eta$, we have $$\spec_{\tens_{X,\eta}} \cal T = \spec_{\tens_{X,\eta'}} \cal T \subset \spec^\fm \cal T$$ and $$\spec_{\tens,X}^{G\cdot \eta}\cal T = \spec_{\tens,X}^{G\cdot \eta'}\cal T \subset \spec^\fm \cal T$$ (as open subschemes), so given a natural isomorphism class $\eta$ of equivalences $\perf X\simeq \cal T$, we may use the same notations $\spec_{\tens_{X,\eta}} \cal T$ and $\spec_{\tens,X}^{G\cdot \eta}\cal T$ (cf. Remark \ref{remark on choice}). Moreover, note that for any quasi-inverse $\eta\inv$ of $\eta$, there is a canonical isomorphism $\spec_{\tens_{X,\eta,\eta\inv}} \cal T \iso \spec_{\tens_{X,\eta}} \cal T$.
\end{notation}
Moreover, the scheme structures on tt-spectra of Fourier-Mukai partners are compatible with the actions of autoequivalences. More precisely, we have the following:
\begin{lemma}\label{compatible action}
    Let $\cal T$ be a triangulated category with $X\in \FM \cal T$. The action $\rho^X_\top: \auteq \cal T \to \aut\spc_{\tens,X} \cal T$ can be made into an action $\rho^X: \auteq \cal T \to \aut \spec_{\tens,X} \cal T$ in a natural way described in the following proof.
\end{lemma}
\begin{proof}
    Take an autoequivalence $\tau:\cal T\simeq \cal T$ with quasi-inverse $\tau\inv$. First, note that $\tau$ induces an isomorphism $$\tau:\spec_{\tens_{X,\eta,\eta\inv}} \cal T \iso \spec_{\tens_{X,\tau\circ \eta,\eta\inv\circ \tau \inv}} \cal T$$ of ringed spaces for any $\eta:\perf X \simeq \cal T$ with quasi-inverse $\eta\inv$. Moreover, for any $\eta_1,\eta_2:\perf X\simeq \cal T$ with quasi-inverses $\eta_1\inv$ and $\eta_2\inv$ and with $\Phi:=\eta_2\inv \circ \eta_1$ birational, we have the following commutative diagram
    \begin{center}
        \DisableQuotes
        \begin{tikzcd}
{\spec_{\tens_{X,\eta_1,\eta_1\inv}}\cal T \cap \spec_{\tens_{X,\eta_2,\eta_2\inv}} \cal T} \arrow[rr, "\tau"] \arrow[d, "\eta_2 \circ \phi_\Phi\circ \eta_1\inv"'] &  & {\spec_{\tens_{X,\tau\circ\eta_1,\eta_1\inv\circ \tau \inv}} \cal T \cap \spec_{\tens_{X,\tau\circ\eta_2,\eta_2\inv\circ \tau \inv}}\cal T } \arrow[d, "(\tau\circ \eta_2) \circ \phi_\Phi \circ \eta_1 \inv \circ \tau \inv"] \\
{\spec_{\tens_{X,\eta_1,\eta_1\inv}}\cal T  \cap \spec_{\tens_{X,\eta_2,\eta_2\inv}}\cal T } \arrow[rr, "\tau"']                                                       &  & {\spec_{\tens_{X,\tau\circ\eta_1,\eta_1\inv \circ \tau \inv}}\cal T  \cap \spec_{\tens_{X,\tau\circ\eta_2,\eta_2\inv\circ \tau \inv}}\cal T }                                                                                  
\end{tikzcd}
    \end{center}
    of isomorphisms, where the vertical arrows are exactly the gluing maps in the proof of Theorem \ref{scheme}. Hence, the actions of $\tau$ on each tt-spectra glue to an automorphism $\rho^X(\tau)\inv$ of $\spec_{\tens,X}\cal T$ as a scheme. Since the automorphism $\rho^X(\tau)\inv$ agrees with $\rho^X_\top(\tau)\inv$ on the underlying topological space, we see by Lemma \ref{class alg result} that the association defines a group homomorphism $\rho^X:\auteq \cal T \to \aut \spec_{\tens,X}\cal T$, noting $\spec_{\tens,X}\cal T$ is a reduced scheme locally of finite type.
\end{proof}
\subsection{Basic properties as schemes}
Now, let us observe some basic properties of Fourier-Mukai loci and tt-spectra of Fourier-Mukai partners as schemes and their relations with (birational) geometry. 
\begin{definition}
    Let $X$ be a smooth projective variety. Let $X_\tens$ denote the scheme $\spec_{\tens,X}\perf X$. For a property $\cal P$ of a scheme, we say $X$ satisfies \textbf{tt-$\cal P$} if $X_\tens$ satisfies $\cal P$.   
\end{definition}
\begin{example}
    Any smooth projective variety is tt-smooth and tt-locally of finite type by Corollary \ref{tt-spectra of FM}. So, we are mainly interested in if a smooth projective variety is tt-separated, tt-irreducible, or tt-quasi-compact. Since a universally closed morphism is quasi-compact and in this paper only examples of tt-quasi-compact varieties will be smooth projective varieties with ample (anti-)canonical bundles, we will not elaborate on tt-universally closedness, let alone tt-projectivity. 
\end{example}
\begin{notation}
Let $X$ be a smooth projective variety. Let $\operatorname{BAuteq} X \subset \auteq \perf X$ denote the subgroup of birational autoequivalences (cf. Corollary \ref{improved h}). Note by definition $\stab(\perf X,\id_{\perf X}) \subset \operatorname{BAuteq} X \ltimes \bb Z[1]$.     
\end{notation}
First, let us consider tt-separatedness. In general, we should not expect Fourier-Mukai loci (or tt-spectra of Fourier-Mukai partners) to be separated (i.e., a smooth projective variety is in general not tt-separated), which agrees with an intuition that gluing projective varieties along isomorphic open subsets produces a non-separated scheme. See also Example \ref{K3} (ii). 
\begin{lemma}\label{tt-s}
    Let $\cal T$ be a triangulated category with $\FM \cal T\neq \emp$. Then the following are equivalent:
    \begin{enumerate}
        \item $\spec^\fm \cal T$ is a separated scheme;
        \item For any $X_i \in \FM \cal T$ and $\eta_i: \perf X_i \simeq \cal T$ for $i =1,2$, if $\eta_2\inv\circ \eta_1$ is a birational triangulated equivalence up to to shifts, i.e., if \[\spec_{\tens_{X_1,\eta_1}} \cal T \cap \spec_{\tens_{X_2,\eta_2}} \cal T \neq \emp,\]
        then we have
        \[
        \spec_{\tens_{X_1,\eta_1}} \cal T = \spec_{\tens_{X_2,\eta_2}} \cal T. 
        \]
        In particular, we have $X_1 \iso X_2$ (by Lemma \ref{main intersection} (ii)). 
        \item $\spec^\fm \cal T$ is a disjoint union of copies of Fourier-Mukai partners of $\cal T$. In particular, the union of isomorphism classes of connected components of $\spec^\fm \cal T$ coincides with $\FM \cal T$.  
    \end{enumerate}
    In particular, separatedness of $\spec^\fm \cal T$ can be checked topologically. Moreover, if $\fmspec \cal T$ is separated, then each $X \in \FM \cal T$ is tt-separated.  
\end{lemma}
\begin{proof}
    The last two claims are clearly equivalent and the last claim clearly implies the first, so it suffices to show the first implies the second. We show its contrapositive. Suppose there exist $X_i \in \FM \cal T$ and $\perf X_i \simeq \cal T$ for $i =1,2$ such that $U:= \spec_{\tens_{X_1,\eta_1}} \cal T \cap \spec_{\tens_{X_2,\eta_2}} \cal T \neq \emp$, but $ \spec_{\tens_{X_1,\eta_1}} \cal T \neq \spec_{\tens_{X_2,\eta_2}} \cal T$. Then, $X:= \spec_{\tens_{X_1,\eta_1}} \cal T \cup \spec_{\tens_{X_2,\eta_2}} \cal T \subset \spec^\fm \cal T$ is not a separated scheme since otherwise $\spec_{\tens_{X_1,\eta_1}} \cal T$ is closed in $X$ (as it is an image of a proper scheme in a separated scheme) and hence $U$ is a closed and open proper subscheme of $\spec_{\tens_{X_2,\eta_2}} \cal T $, which is absurd. Therefore, $\spec^\fm \cal T$ has a non-separated open subscheme and hence is not separated either.    
\end{proof}
By the same arguments, we get the following:
\begin{corollary}\label{tt-seprated}
    Let $X$ be a smooth projective variety. Then the following are equivalent:
    \begin{enumerate}
        \item $X$ is tt-separated;
        \item For any $\eta_i: \perf X  \simeq \cal T$ for $i =1,2$, if $\eta_2\inv\circ \eta_1$ is a birational triangulated equivalence up to shifts, i.e., if \[\spec_{\tens_{X,\eta_1}} \cal T \cap \spec_{\tens_{X,\eta_2}} \cal T \neq \emp,\]
        then we have
        \[
        \spec_{\tens_{X,\eta_1}} \cal T = \spec_{\tens_{X,\eta_2}} \cal T. 
        \]
        In particular,  $X_\tens$ is a disjoint union of copies of X.
    \end{enumerate}
    In particular, tt-separatedness can be checked topologically. See also Corollary \ref{new tt-separated} for more equivalent conditions. Moreover, (ii) implies separatedness of $\spec^\fm \cal T$ can be checked topologically. 
\end{corollary}
\begin{remark}\label{tt-seprated is tte}
    Note that if $X$ and $Y$ are tt-separated, then $X_\tens = Y_\tens$ implies $X \iso Y$, which gives a positive answer to Question \ref{tte}.
\end{remark}
Moreover, the following corollary of Lemma \ref{main intersection} is useful for understanding Fourier-Mukai loci:
\begin{corollary}\label{disjoint}
    Let $\cal T$ be a triangulated category with $X,Y \in \FM\cal T$. If $X$ and $Y$ are not birational Fourier-Mukai partners (in particular, not birationally equivalent), then
    \[
    \spec_{\tens,X} \cal T \cap \spec_{\tens,Y} \cal T = \emp. 
    \]
    Moreover, if any birational Fourier-Mukai partner is isomorphic, in which case $\cal T$ is said to be of \textbf{disjoint Fourier-Mukai type}, then we have
    \[
    \spec^\fm \cal T = \bigsqcup_{X \in \FM \cal T} \spec_{\tens,X} \cal T.
    \] 
    In particular, $\spec^\fm \cal T$ is separated if and only if $\cal T$ is of disjoint Fourier-Mukai type and each $X \in \FM \cal T$ is tt-separated. 
\end{corollary}
\begin{proof}
The first claim follows since if $\spec_{\tens,X} \cal T \cap \spec_{\tens,Y} \cal T \neq \emp$, then there exist $\eta_X: \perf X \simeq \cal T$ and $\eta_Y: \perf Y \simeq \cal T$ such that $\spc_{\tens,\eta_X} \cal T \cap \spc_{\tens,\eta_Y} \cal T \neq \emp$, so $\eta_Y\inv \circ \eta_X$ is birational by Lemma \ref{main intersection} (i), i.e., $X$ and $Y$ are birational Fourier-Mukai partners. The second claim follows by the first claim and the last claim follows by Lemma \ref{tt-s}.     
\end{proof}
Now, let us investigate tt-irreducibility. First, we note that irreducible components of a Fourier-Mukai locus or a tt-spectrum of a Fourier-Mukai partner coincide with connected components by the following easy lemma.
\begin{lemma}\label{scheme irreducible}
    Let $X$ be a scheme admitting an open cover $\{U_i\}_{i \in I}$ by irreducible open subschemes $U_i$. Then a connected component and an irreducible component coincide. In particular, such a component is given as a union of mutually intersecting irreducible open subschemes of the form $U_i$.
\end{lemma}
\begin{proof}
    First, we can write 
    \[
    X = \bigsqcup_{j \in J} V_j,
    \]
    where each $V_j$ is a union of mutually intersection irreducible open subschemes of the form $U_i$.
    Note that each $V_j$ is irreducible and hence connected, which suffices for a proof. 
\end{proof}
In general, a smooth projective variety is not tt-irreducible, but we have an important symmetry for irreducible components, which gives a useful geometric intuition for $X_\tens$:
\begin{notation}
    Let $X$ be a smooth projective variety and let $\eta \in \auteq\perf X$. Let $X_\tens(\eta)$ denote the irreducible component of $X_\tens$ containing $\spec_{\tens_{X,\eta}} \perf X$, which is canonically an open subscheme of $\spec^\fm \cal T$ by Lemma \ref{scheme irreducible}. 
\end{notation}
\begin{lemma}\label{id part}
    Let $X$ be a smooth projective variety. Then, we have
    \[
    X_\tens(\id_{\perf X}) = \spec^{\operatorname{BAuteq} X \cdot \id_{\perf X}}_{\tens,X} \perf X. \qedhere
    \]
\end{lemma}
\begin{proof}
    By Lemma \ref{main intersection}, we see that for $\eta \in \auteq \perf X$, $\spec_{\tens_{X,\eta}} \perf X \cap \spec_{\tens_{X,\id_{\perf X}}}\perf X \neq \emp$ if and only if $\eta \in \operatorname{BAuteq} X$, so the claim follows from Lemma \ref{scheme irreducible}. 
\end{proof}
\begin{theorem}\label{irreducible component}
    Let $X$ be a smooth projective variety. Then, any $\eta\in \auteq \perf X$ induces an isomorphism
    \[
    \eta: X_\tens(\id_{\perf X})\overset{\iso}{\to} X_\tens(\eta).
    \]
    In particular, each irreducible component is isomorphic to $X_\tens(\id_{\perf X})$. Moreover, we can take a family $\{\eta_i \in \auteq \perf X\}_{i \in I}$ of autoequivalences of $\perf X$ such that
    \[
    X_\tens = \bigsqcup_{i \in I} X_\tens(\eta_i). 
    \]
    See Construction \ref{elliptic symmetric} for an example of this expression in the case of an elliptic curve. 
\end{theorem}
\begin{proof}
    By Lemma \ref{compatible action}, any $\eta \in \auteq \perf X$ induces an automorphism of $X_\tens$ (as a scheme), which in particular induces an isomorphism of connected components.  
\end{proof}
\begin{remark}\label{minimal generator candidate}
    Using the same notations as in Theorem \ref{irreducible component}, we can see that a subgroup of $\auteq \perf X$ generated by $\operatorname{BAtueq} \perf X$ and $\{\eta_i\}_{i \in I}$ is an $\id_{\perf X}$-generator. To obtain a minimal $\id_{\perf X}$-generator, we need to be clever about choices of generators (e.g. we may have some relations among $\eta_i$'s), but the author has no progress in this direction. See Remark \ref{minimal generator for elliptic curves?} for the case of an elliptic curve.  
\end{remark}
By Lemma \ref{id part} and Theorem \ref{irreducible component}, we can make the following quick observations on tt-separatedness. 
\begin{corollary}\label{new tt-separated}
Let $X$ be a smooth projective variety. Then, the following are equivalent:
\begin{enumerate}
    \item $X$ is tt-separated;
    \item $X_\tens(\id_{\perf X})$ is separated;
    \item $X_\tens(\id_{\perf X}) = \spec_{\tens_{X,\id_{\perf X}}}\perf X \iso X$;
    \item $\operatorname{BAuteq} X$ is an $\id_{\perf X}$-stabilizer of $X$.
    \item $\stab(\perf X, \id_{\perf X}) = \operatorname{BAuteq} X \times \bb Z [1]$. \qedhere
\end{enumerate}
\end{corollary}
\begin{proof}
Conditions (i) and (ii) are equivalent by Theorem \ref{irreducible component}. Conditions (ii) and (iii) are equivalent by Corollary \ref{tt-seprated} and Lemma \ref{id part}. Conditions (iii) and (iv) are equivalent by Lemma \ref{id part}. Conditions (iv) and (v) are equivalent since we have $\stab(\perf X, \id_{\perf X}) \subset \operatorname{BAuteq} X \times \bb Z [1]$ by definition. 
\end{proof}
Hence, tt-separatedness of $X$ means birational maps associated to birational autoequivalences of $\perf X$ need to be an automorphism of $X$. Now, we can also make some observations for tt-irreducibility.  
\begin{lemma}\label{ttirr}
    Let $X$ be a smooth projective variety. Then the following are equivalent:
    \begin{enumerate}
        \item $X$ is tt-irreducible;
        \item $X$ is tt-connected;
        \item $\operatorname{BAuteq} X$ is an $\id_{\perf X}$-generator. 
    \end{enumerate}
    Note that $X$ is tt-irreducible and tt-separated if and only if $X_\tens  =\spec_{\tens_{X,\id_{\perf X}}}\perf X \iso X$. 
\end{lemma}
\begin{proof}
    By Lemma \ref{scheme irreducible} and Theorem \ref{irreducible component}, all the conditions are equivalent to $X_\tens = X_\tens(\id_{\perf X})$.       
\end{proof}
\begin{remark}
    In \cite{Uehara2017ATF}, a smooth projective variety $X$ is said to be \textbf{of $K$-equivalent type} if 
    \[
    \operatorname{BAuteq} X \ltimes \bb Z [1] = \auteq \perf X 
    \]
    (cf. Remark \ref{Uehara K-equiv}). Note if $X$ is of $K$-equivalent type, then $X$ is tt-irreducible by Lemma \ref{ttirr} noting shifts are universal stabilizers. 
\end{remark}

Let us examine specific examples to better understand which varieties are or are not tt-separated (resp. tt-irreducible). We will explore these in greater detail later. Notably, it's intriguing that all potential combinations occur quite naturally. 
\begin{example}\label{ttp examples} Let $X$ be a smooth projective variety. 
    \ \begin{enumerate}
        \item If $X$ has the ample (anti-)canonical bundle, then it is tt-separated and tt-irreducible by Corollary \ref{fm ample}. 
        \item If $X$ is an abelian variety, then it is not tt-irreducible by Lemma \ref{abelian not tt-irreducible}. If $X$ is moreover simple, then it is tt-separated by Lemma \ref{simple}. The author expects a general abelian variety to be tt-separated (Conjecture \ref{abelian conjecture}).
        \item If $X$ is a toric variety, then it is tt-irreducible by Lemma \ref{toric} but in general not tt-separated by Remark \ref{toricrem}. 
        \item If $X$ is a $K3$ surface, then it is in general neither tt-separated nor tt-irreducible by Example \ref{K3} and Lemma \ref{K3 disjoint}. \qedhere
    \end{enumerate}
\end{example}

Finally, let us briefly consider quasi-compactness. 
\begin{definition}
    Let $\cal T$ be a triangulated category with $\FM \cal T\neq \emp$. We say $\cal T$ has a \textbf{finite Fourier-Mukai cover} if there is a finite subset $S \subset \FM \cal T$ such that $\spc^\fm \cal T = \bigcup _{X \in S}\spc_{\tens,X}\cal T$. 
\end{definition}
\begin{remark}
    Clearly, if $\FM \cal T$ is a finite set, then $\cal T$ has a finite Fourier-Mukai cover. Also, the assumption that $\FM \cal T$ is finite is not too special. Indeed, Kawamata conjectured that for any smooth projective variety $X$, $\FM \perf X$ is finite (\cite{Kawamata2002DEquivalenceAK}*{Conjecture 1.5}) and it is indeed shown that the conjecture holds for curves and surfaces (\cite{Kawamata2002DEquivalenceAK}, \cite{HuyBook}), abelian varieties and varieties with ample (anti-)canonical bundle (\cite{FAVERO20121955}), and toric varieties (\cite{Kawtoric}). On the other hand, it is shown that there is a counter-example to the conjecture by considering an infinite family of blow-ups of $\bb P^3$ at certain 8 points, which are all connected by Cremona transformations. (\cite{Lesi15}). In this paper, we will not pursue if there exists a triangulated category $\cal T$ with infinite $\FM \cal T$, but with a finite Fourier-Mukai cover, but it should be interesting to understand such a question.  
\end{remark}
\begin{lemma}\label{tt-qc}
    Let $\cal T$ be a triangulated category with $\FM \cal T \neq \emp$. Then, the following are equivalent:
    \begin{enumerate}
        \item $\spec^\fm \cal T$ is quasi-compact;
        \item $\cal T$ has a finite Fourier-Mukai cover $\spc^\fm \cal T = \bigcup _{X \in S}\spc_{\tens,X}\cal T$ such that $X$ is tt-quasi-compact for any $X\in S$. \qedhere
    \end{enumerate} 
\end{lemma}
\begin{proof}
    First, suppose $\spec^\fm \cal T$ is quasi-compact. Clearly, $\cal T$ has a finite Fourier-Mukai cover $\spc^\fm \cal T = \bigcup _{X \in S}\spc_{\tens,X}\cal T$. Since $\spec^\fm \cal T$ is of finite type, it is in particular noetherian and hence each $\spc_{\tens,X}\cal T$ is quasi-compact. The other direction is clear.   
\end{proof}
By Theorem \ref{irreducible component}, we can also see the following: 
\begin{lemma}
    Let $X$ be a smooth projective variety. Then, the following are equivalent:
    \begin{enumerate}
        \item $X$ is tt-quasi-compact;
        \item $X_\tens$ has only finitely many connected component and $X_\tens(\id_{\perf X})$ is quasi-compact. \qedhere
    \end{enumerate}
\end{lemma}
\begin{obs} Let $\cal T$ be a triangulated category with $\FM \cal T \neq \emp$.  
Note that tt-quasi-compactness of a smooth projective variety $X$ can fail for two reasons:
\begin{enumerate}
    \item $X_\tens$ has infinitely many (quasi-compact) connected components. This is the case for abelian varieties (Example \ref{K3} (i)).   
    \item $X_\tens$ has (finitely many) non-quasi-compact connected components. This is generally the case for toric varieties (Lemma \ref{toric} and Remark \ref{toricrem}). 
\end{enumerate}
It is indeed hard to come up with tt-quasi-compact smooth projective varieties except for ones with ample (anti-)canonical bundle. From observations so far, it is reasonable to expect the following.
\end{obs}

\begin{conjecture}
    A smooth projective variety $X$ is tt-separated and tt-quasi-compact if and only if it has ample (anti-)canonical bundle. Note the first condition is equivalent to saying $X_{\tens}$ is a finite disjoint union of copies of $X$. 
\end{conjecture}

\section{Examples}
Now let us apply results so far to compute Fourier-Mukai loci for some specific cases. 
\subsection{Abelian varieties} 
First, let us see the case of abelian varieties. 
\begin{lemma}\label{abel FM disjoint}
    Let $\cal T$ be a triangulated category with an abelian variety $X \in \FM \cal T$. Then, $\cal T$ is of disjoint Fourier-Mukai type, i.e, \[
        \spec^\fm \cal T = \bigsqcup_{X \in \FM \cal T} \spec_{\tens,X} \cal T.\qedhere
        \]
\end{lemma}
\begin{proof}
    By \cite{HuyNie11}*{\href{https://arxiv.org/pdf/0801.4747.pdf}{Theorem 0.4}}, any Fourier-Mukai partner of $X$ is also an abelian variety. Moreover, it is well-known that any birationally equivalent abelian varieties are isomorphic. 
\end{proof}
In particular, when $X\not \iso \hat X$, this gives a different proof and generalization of \cite{HO22}*{\href{https://arxiv.org/pdf/2112.13486v3.pdf}{Example 3.4}}, which claims the Fourier-Mukai transform $\perf X \to \perf \hat X$ associated to the Poincare bundle is not birational even up to shifts for any abelian variety $X$. Note \cite{HO22}*{\href{https://arxiv.org/pdf/2112.13486v3.pdf}{Example 3.4}} gives the following:
\begin{lemma}\label{abelian not tt-irreducible}
    An abelian variety that is isomorphic to its dual is not tt-irreducible. 
\end{lemma}
On the other hand, we can also see that an abelian variety is likely to be tt-separated. For this, let us recall some constructions and results from \cite{Orlov_2002}. 
\begin{construction}
Let $X$ be an abelian variety and take an autoequivalence $\Phi_\ecal K \in \auteq \perf X$ with kernel $\ecal K$. Then, we can obtain an automorphism 
\[
f_\ecal K: X \times \hat X \overset{\sim}{\to} X \times \hat X 
\]
of abelian varieties as follows. First, let us define an autoequivalence $F_\ecal K: \perf (X\times \hat X) \simeq \perf (X \times \hat X)$ by the following commutative diagram:
\[
\xymatrix{
\perf(X \times \hat X) \ar[r]^{F_\ecal K}\ar[d]_{\id\times \Phi_\ecal P}& \perf(X\times \hat X) \ar[d]^{\id\times \Phi_\ecal P} \\
\perf(X \times X)\ar[d]_{\mu_*} &  \perf(X \times X) \ar[d]^{\mu_*}\\
\perf(X \times X) \ar[r]_{\Phi_\ecal K \times \Phi_{\ecal K_R}}&  \perf(X \times X)
}
\]
where $\ecal P$ is the Poincar\'{e} bundle on $X \times \hat X$, the map $\mu:X \times X \to X\times X$ is defined by $(x,y) \mapsto (x+y,y)$, and $\ecal K_R:= \ecal K^\vee [\dim X]$. Now, in \cite{Orlov_2002}, Orlov showed that the equivalence $F_\ecal K$ maps any skyscraper sheaf to a skyscraper sheaf and moreover the corresponding automorphism 
\[
f_\ecal K: X \times \hat X \overset{\sim}{\to} X \times \hat X 
\]
(cf. \cite{HuyBook}*{Corollary 5.23}) is an automorphism of abelian varieties. Moreover, the construction
\[
\gamma:\auteq \perf X \to \aut (X\times \hat X);\quad \Phi_\ecal K \mapsto f_\ecal K
\]
is a group homomorphism. 
\end{construction}
To illustrate Orlov's result, let us introduce the following notions.
\begin{definition}
    Let $X$ be an abelian variety. Then, define the group of \textbf{symplectic automorphisms} of $X \times \hat X$ (with respect to the natural symplectic form) to be 
    \[
    \sp (X\times \hat X):= \l\{\begin{pmatrix}
        f_1 & f_2 \\ f_3 & f_4
    \end{pmatrix} \in \aut(X \times \hat X) \mid \begin{pmatrix}
        f_1 & f_2 \\ f_3 & f_4
    \end{pmatrix}\begin{pmatrix}
        \hat f_4 & - \hat f_2 \\ -\hat f_3 & \hat f_1
    \end{pmatrix} = \id_{X\times \hat X} \r\},
    \]
    where $\aut(X \times \hat X)$ denotes the group of automorphisms of abelian varieties and $\hat f_i$ denotes the transpose of $f_i$. Here, we are writing an automorphism $f:X\times \hat X \to X\times \hat X$ with matrix form, where $f_1:X \to X$, $f_2:\hat X \to X$, etc. We say a symplectic automorphism $f$ is \textbf{elementary} if $f_2$ is an isogeny.  
\end{definition}
\begin{theorem}[\cite{Orlov_2002}*{\href{https://arxiv.org/pdf/alg-geom/9712017.pdf}{Theorem 4.14}}]\label{Orlov abelian}
    For any abelian variety $X$,  we have a short exact sequence
        \[
        0 \to \bb Z[1] \times (X \times \hat X)_k \to \auteq \perf X \overset{\gamma}{\to} \sp(X \times \hat X) \to 1,
        \]
        where $(X \times \hat X)_k = X_k \times \pic^0(X)$ denotes the group of $k$-points of $X \times \hat X$.
\end{theorem}
\begin{proof}[Idea of a proof]
    First of all, the map $ \bb Z[1] \times (X \times \hat X)_k \to \auteq \perf X$ is given by sending $(n,x,\ecal L)$ to $t_x^*\circ (-\tens_{\ecal O_X}^\bb L \ecal L)[n]$, where $t_x:X \to X$ denotes the translation $y\mapsto x+y$. The following examples show $\bb Z[1] \times (X \times \hat X)_k$ is contained in the kernel of $\gamma$. 
    
    Let us also mention Orlov's arguments to show $\gamma$ is surjective. First, Orlov shows that for any elementary symplectic automorphism $f \in \sp(X \times \hat X)$, there exists a (simple semi-homogeneous) vector bundle $\ecal E$ on $X\times X$ such that $\Phi_\ecal E$ is an autoequivalence and $f_\ecal E = f$ (\cite{Orlov_2002}*{Construction 4.10}). Then Orlov claims that any symplectic automorphism can be factored into a composition of elementary symplectic automorphisms.             
\end{proof}
Let us see some examples of this construction:
\begin{example}[\cite{Ploog_2005}*{\href{https://www.mathematik.hu-berlin.de/~ploog/PAPERS/Ploog_phd.pdf}{Example 4.5}}]\label{non-isogeny example} Let $X$ be an abelian variety. 
    \begin{enumerate}
        \item Let $\ecal L$ be a line bundle on $X$. Then $\gamma(-\tens_{\ecal O_X}^\bb L \ecal L) = \begin{pmatrix}\id_X & 0 \\ \phi_\ecal L & \id_{\hat X}\end{pmatrix}$, where $\phi_\ecal L:X \to \hat  X$ is defined by $x \mapsto  t^*_x \ecal L \tens \ecal L^\vee$ for a translation $t_x: y\mapsto y+x$. Recall $\phi_\ecal L = 0$ if and only if $\ecal L \in \hat X$ and $\phi_\ecal L$ is an isomorphism if and only if $(X,\ecal L)$ is {principally polarized}. 
        \item For a translation $t_x:y \mapsto y+x$ for some $x \in X$, we have $\gamma({t_x}_*) = \id_{X\times \hat X}$.
        \item We clearly have $\gamma([n]) = \id_{X \times \hat X}$. 
        \item For an automorphism $\phi:X \to X$ of abelian varieties, we have $\gamma(\phi_*) = \begin{pmatrix}
            \phi\inv & 0 \\ 0 & \hat \phi
            \end{pmatrix}$. 
    \end{enumerate}
    Note none of the images of $\gamma$ above is elementary.    
\end{example}
\begin{corollary}\label{isogeny}
    Let $\Phi \in \auteq \perf X$. If $\gamma(\Phi)$ is an elementary symplectic automorphism, then 
    \[
    \spec_{\tens_{X,\id_{\perf X}}} \perf X \cap \spec_{\tens_{X,\Phi}} \perf X = \emp. 
    \]
    In particular, such an autoequivalence $\Phi$ is not birational (even up to shifts). 
\end{corollary}
\begin{proof}
    Take an autoequivalence $\Phi \in \auteq \perf X$ with $\gamma(\Phi)$ elementary. By Theorem \ref{Orlov abelian}, there is a vector bundle $\ecal E$ on $X \times X$ such that $\Phi = \Phi_{\ecal E} \circ t_x^*\circ (-\tens_{\ecal O_X}^\bb L \ecal L)[n]$ for some $(n,x,\ecal L) \in \bb Z[1] \times (X \times \hat X)_k$. Thus, we are done by Corollary \ref{flat kernel} since the first three compositions send a skyscraper sheaf to a skyscraper sheaf (up to shifts). 
\end{proof}
If we moreover assume $X$ is a simple abelian variety, then we can say more. 
\begin{lemma}\label{simple}
    Let $X$ be a simple abelian variety. Then $X$ is tt-separated. 
\end{lemma}
\begin{proof}
By Corollary \ref{new tt-separated}, it suffices to show $X_\tens(\id_{\perf X}) = \spec_{\tens_{X,\id_{\perf X}}} \perf X$, i.e., for any birational autoequivalence $\Phi \in \auteq \perf X$,
\[
\spec_{\tens_{X,\id_{\perf X}}} \perf X = \spec_{\tens_{X,\Phi}} \perf X. 
\]
Now, by Corollary \ref{isogeny}, if $\gamma(\Phi)$ is elementary, then $\Phi$ is not birational, so we can exclude. On the other hand, in the proof of \cite{LOPEZMARTIN201792}*{\href{https://arxiv.org/pdf/1702.00232.pdf}{Theorem 2.14}}, Mart\'in-Prieto showed that if $\gamma(\Phi)$ is not elementary, then we have  either $\Phi = \phi_*$ for some automorphism $\phi:X \to X$ or $\gamma(\Phi) = \begin{pmatrix}
    f_1 & 0 \\ f_3 & f_4
\end{pmatrix}$ with $f_3:X \to \hat X$ being an isogeny. In the former case, $\Phi = \phi_*$ is birational, but it satisfies $\spec_{\tens_{X,\id_{\perf X}}} \perf X = \spec_{\tens_{X,\Phi}} \perf X$ by Example \ref{standard autoeq}, so there is nothing to show. In the latter case, it suffices to see $\Phi$ is not birational (even up to shift). First, set 
\[
\hat \Phi := {\Phi_{\tilde {\ecal P}}} \circ  \Phi \circ {\Phi_{\tilde{\ecal P}}}\inv\in \auteq \perf \hat X,
\]
where $\tilde{\ecal P}$ denotes the Poincar\'e bundle on $\hat{\hat X} \times \hat X$ so that we have $\Phi_{\tilde {\ecal P}}: \perf \hat{\hat X} \simeq \perf \hat X$ and we fixed an identification $\hat{\hat X} \iso X$, which fixes an identification $\perf \hat{\hat X} \simeq \perf X$. By the same argument as in the proof of \cite{LOPEZMARTIN201792}*{\href{https://arxiv.org/pdf/1702.00232.pdf}{Theorem 2.14}}, we see that $$\gamma(\hat \Phi) = \begin{pmatrix}
    f_4 & - f_3 \\ 0 & f_1
\end{pmatrix} \in \sp(\hat X\times X) = \sp(\hat X \times \hat {\hat X}).$$ Thus, $\hat \Phi$ is elementary, so the corresponding Fourier-Mukai kernel $\ecal K$ is a vector bundle on $\hat X \times \hat X$. Now, noting $\Phi_{\tilde{\ecal P}}$ sends a skyscraper sheaf on $X = \hat{\hat{X}}$ to the corresponding line bundle of degree $0$ on $\hat X$, it suffices to see that $\hat \Phi$ does not send a line bundle of degree $0$ to a line bundle of degree $0$ (even up to shift). Therefore, the following lemma is sufficient.

\begin{enmlem}
    Let $X$ and $Y$ be smooth projective varieties with $\dim X \geq 1$ and let $\ecal E$ be a locally free $\ecal O_{X\times Y}$-module of finite rank. Then, for the canonical projection $p_Y:X\times Y \to Y$, $\bb R{p_Y}_* \ecal E$ is not isomorphic to any locally free $\ecal O_Y$-module of finite rank even up to shift. In particular, if we have a Fourier-Mukai equivalence $\Phi_\ecal E:\perf X \simeq \perf Y$ with $\ecal E$ locally free, then $\Phi_\ecal E$ cannot send a locally free $\ecal O_X$-module of finite rank to a locally free $\ecal O_Y$-module of finite rank (even up to shift). 
\end{enmlem}
\begin{inproof}
    The latter claim clearly follows from the first claim. For the first claim, assume $\bb R{p_Y}_*\ecal E[i]$ is isomorphic to a locally free $\ecal O_Y$-module $\ecal F$ for some $i \in \bb Z$. By comparing cohomology sheaves, we see $i = 0$. Now, take affine open subsets $U \subset X$ and $V\subset Y$ so that $\ecal E|_{U\times V}$ and $\ecal F|_V$ are free of rank $m$ and $n$, respectively. Consider the following pullback diagram:
    \begin{center}
        \DisableQuotes
        \begin{tikzcd}
U \times V \arrow[r, hook] \arrow[d, "p_V"'] & X \times Y \arrow[d, "p_Y"] \\
V \arrow[r, hook]                            & Y                          
\end{tikzcd}
    \end{center}
Then, by flat base change, we have 
\[
(\bb R {p_V}_*\ecal O_{U\times V})^{\oplus m} \iso \bb R {p_V}_*\ecal E|_{U\times V}\iso \ecal F|_V \iso {\ecal O_V}^{\oplus n}, 
\]
i.e., $(\bb R{p_V}_*\ecal O_{U\times V})^{\oplus m} \iso {\ecal O_V}^{\oplus n}$. Now, since ${p_V}$ is an affine morphism, we have
\[
({p_V}_* \ecal O_{U\times V})^{\oplus m} \iso \ecal O_V^{\oplus n}
\]
which is absurd since $\dim X \geq 1$ by supposition and hence $U$ is not finite over $\spec k$. 
\end{inproof}
\end{proof}
From Example \ref{non-isogeny example} and (the proof of) the case of simple abelian varieties, the author expects the following:
\begin{conjecture}\label{abelian conjecture}
    Any abelian variety is tt-separated. 
\end{conjecture}
\begin{remark}
   The conjecture is solved affirmatively in \cite{ito2024new}. 
\end{remark}
\subsection{Spherical twists and surfaces}
Next, let us make some observations on spherical twists and the case of surfaces. 
First, let us recall the following notions.
\begin{definition}
    Let $X$ be a smooth projective variety. An object $\ecal F \in \perf X$ is said to be \textbf{spherical} if \begin{enumerate}
        \item $\ecal F \iso \ecal F \tens \omega_X$;
        \item $\hom(\ecal F,\ecal F[i]) = \begin{cases}
            k & \text{if $i = 0,\dim X$} \\ 0 & \text{otherwise}.
        \end{cases}$
    \end{enumerate}
    Note if $\ecal F$ is spherical, then $\ecal F^\vee$, $\ecal F[i]$ for any $i \in \bb Z$ and $\ecal F \tens \ecal L$ for any line bundle $\ecal L$ are spherical. 
\end{definition}
\begin{construction}
Let $X$ be a smooth projective variety and take an object $\ecal E \in \perf X$. Define $\ecal P_\ecal E \in \perf X\times X$ to be
\[
\ecal P_\ecal E := \operatorname{cone} (p_1^*\ecal E^\vee \tens p_2^* \ecal E \to \ecal O_\Delta),
\]
where $p_i:X\times X \to X$ denotes the $i$-th projection, $\ecal O_\Delta$ denotes the structure sheaf on the (image of) diagonal $\Delta\subset X\times X$, and the morphism is given by the composition
\[
p_1^*\ecal E^\vee \tens p_2^* \ecal E \to \Delta_*\Delta^* (p_1^*\ecal E^\vee \tens p_2^* \ecal E) = \Delta_*(\ecal E^\vee \tens \ecal E) \to \Delta_*\ecal O_X = \ecal O_{\Delta}.  
\]
When $\ecal E$ is spherical, the associated Fourier-Mukai transform 
\[
T_\ecal E:= \Phi_{\ecal P_\ecal E}:\perf X \to \perf X
\]
is called the \textbf{spherical twist} associated to $\ecal E$. It is a standard result (e.g. \cite{HuyBook}*{Proposition 8.6}) that any spherical twist is an autoequivalence. 
\end{construction}

Recall the following result:
\begin{lemma}\cite{Uehara2017ATF}*{\href{https://arxiv.org/pdf/1704.00292.pdf}{Example 4.2}}\label{sph}
    Let $X$ be a smooth projective variety and let $\ecal E \in \perf X$ be a spherical object. If $x \in \supp \ecal E$, then $\supp T_\ecal E (k(x)) = \supp \ecal E$. If $x \in  X \setminus \supp \ecal E$, then $T_\ecal E(k(x)) = k(x)$. 
\end{lemma}
Now, we get the following useful corollary:
\begin{corollary} \label{birational spherical}
    Let $X$ be a smooth projective variety of dimension $\geq1$ and let $\ecal E \in \perf X$ be a spherical object. Then, $T_\ecal E$ is birational if and only if $\supp \ecal E \neq X$. If $\supp \ecal E$ is moreover not a singleton, then the associated birational map $X \ratmap X$ is given by the identity with maximal domain of definition (of $T_\ecal E$) given by $X \setminus \supp \ecal E$; in particular, $\spec_{\tens_{\ecal O_X}^\bb L}\perf X \cap T_{\ecal E}(\spec_{\tens_{\ecal O_X}^\bb L}\perf X) \iso X \setminus \supp \ecal E$. 
\end{corollary}
\begin{proof}
    The first part follows from Lemma \ref{sph}. In particular, in such a case, the domain of definition of $T_\ecal E$ contains $X\setminus \supp \ecal E$. If $\supp \ecal E$ is moreover not a singleton, then for any $x \in \supp \ecal E$, $\supp T_\ecal E (k(x))$ is not a singleton (i.e., $x$ is not in the domain of definition of $T_\ecal E$ by Lemma \ref{intersec ref}), which shows the second part. 
\end{proof}
In particular, we can see the following by Corollary \ref{tt-seprated}:
\begin{corollary}\label{spherical not tt}
    Let $X$ be a smooth projective variety (of dimension $\geq 1$). If there exists a spherical object $\ecal E \in \perf X$ with $\supp \ecal E$ neither being $X$ nor being a singleton, then $X$ is not tt-separated. 
\end{corollary}
\begin{remark}
    This is another evidence for Conjecture \ref{abelian conjecture} since an abelian variety of dimension at least $2$ has no spherical object (e.g. \cite{HuyBook}*{Example 8.10(iv)}). 
\end{remark}
 \begin{example}\label{elliptic twist} 
    One of the easiest examples of spherical twists arises for elliptic curves. Let $X$ be an elliptic curve. Recall as a variant of \cite{Orlov_2002}, we have the following short exact sequence:
    \[
    1 \to \aut X \ltimes \pic^0 X \times \bb Z[2] \to \auteq \perf X \overset{\theta}{\to} \sl(2,\bb Z) \to 1, 
    \]
    where $\theta$ is given by the action on the image of the charge map $Z:K(X) \surj \bb Z^2,\ecal F \mapsto (\rank \ecal F,\chi (\ecal F))$. Since the author could not find a proof of this fact, let us include it for completeness. The surjectivity of $\theta$ follows by noting that the images of the spherical twists $T_{\ecal O_{X}}$ and $T_{k(x)}$ for a closed point $x \in X$ generate $\sl(2,\bb Z)$ (cf. \cite{SeiTho01}). On the other hand, to show $\aut X \ltimes \pic^0 X \times \bb Z[2] = \ker \theta$,  it suffices to show any $\Phi \in \ker \theta$ is a standard autoequivalence since then by the Riemann-Roch theorem, the line bundle part needs to have degree $0$ and the shift part needs to have even degree. Indeed, since any object in $\perf X$ is a direct sum of shifts of coherent sheaves, any $\Phi \in \ker \theta$ sends any indecomposable skyscraper sheaf to an indecomposable skyscraper sheaf (up to shift) as desired. 

    Now, noting that the images of $\ecal O_X$ and $k(x)$ under $\theta$ generate $\sl(2,\bb Z)$, we see that any autoequivalence can be written as compositions of spherical twists and standard equivalences. Moreover, it is easy to see that $T_{k(x)} \simeq \ecal O_X(x) \tens - $ (e.g. \cite{HuyBook}*{Example 8.10 (i)}) and hence is indeed a standard equivalence. This illustrates that for the latter claim in Corollary \ref{birational spherical} we really need to exclude the case when the support of a spherical object is a singleton since in this example the birational map associated to $T_{k(x)}$ is the identity on the whole $X$, not just on $X \setminus \{x\}$. Note, on the other hand, that $T_{\ecal O_X}$ is not birational by Corollary \ref{birational spherical}. 

    Now, let us make some observations on $X_\tens$. Recall we have the following relation (called a \textbf{Braid group relation}) (\cite{SeiTho01}):
    \[
    T_{\ecal O_X} T_{k(x)} T_{\ecal O_X} = T_{k(x)} T_{\ecal O_X} T_{k(x)}.
    \]
   Noting that $T_{k(x)}$ is a standard autoequivalence, this relation says
   \[
   \spec_{\tens_{X,T_{\ecal O_X} T_{k(x)} T_{\ecal O_X}}}\perf X = \spec_{\tens_{X, T_{k(x)} T_{\ecal O_X}}}\perf X \subset X_\tens.
   \]
   In other words, $T_{\ecal O_X}$ is a $T_{k(x)} T_{\ecal O_X}$-stabilizer, but not an $\id_{\perf X}$-stabilizer. Also, see Construction \ref{elliptic symmetric} for a thorough description of $X_\tens$. 
\end{example}

Spherical twists give the following key example for Fourier-Mukai loci corresponding to surfaces. 
\begin{example}\label{-2 curve}
    (\cite{HuyBook}*{Example 8.10(iii)}) Let $X$ be a smooth projective surface and $C \subset X$ a \textbf{$(-2)$-curve}, i.e., a smooth (irreducible) rational curve with $C^2 = -2$. Then $\ecal O_C$ is a spherical object. More generally, the pull-back $\ecal O_C(k)$ of $\ecal O_{\bb P^1}(k)$ under an isomorphism $C \iso \bb P^1$ is spherical. Note any smooth irreducible rational curve in a $K3$ surface is a $(-2)$-curve.  
\end{example}
By Corollary \ref{birational spherical}, we get the following:
\begin{corollary}\label{-2 twist}
    Let $X$ be a smooth projective surface. Suppose $X$ admits a $(-2)$-curve $C$. Then,
    \[
    \spec_{\tens_{X,\id_{\perf X}}} \perf X \cap \spec_{\tens_{X,T_{\ecal O_C}}} \perf X \iso X\setminus C. 
    \]
    In particular, $X$ is not tt-separated. 
\end{corollary}
\begin{remark}
    Note that by blowing up a surface twice and looking at the strict transform of the exceptional curve of the first blow-up, we can always get a surface with a $(-2)$-curve. Therefore, tt-separatedness is not birational invariance. 
\end{remark}
In the rest of this subsection, let us consider specific cases of $K3$ surfaces although we will indeed see the following claim holds for any surface in the next subsection (Corollary \ref{surface FM disjoint}). 
\begin{lemma}\label{K3 disjoint}
     Let $\cal T$ be a triangulated category with a $K3$ surface $X \in \FM \cal T$. Then, $\cal T$ is of disjoint Fourier-Mukai type, i.e, \[
        \spec^\fm \cal T = \bigsqcup_{X \in \FM \cal T} \spec_{\tens,X} \cal T.\qedhere
        \]
\end{lemma}
\begin{proof}
    It is well-known that a Fourier-Mukai partner is also a $K3$ surface (e.g. \cite{HuyBook}*{Corollary 10.2}). Moreover, note any birationally equivalent $K3$ surfaces are isomorphic since they are minimal models.
\end{proof}
In particular, this gives a different proof and generalization of \cite{HO22}*{\href{https://arxiv.org/pdf/2112.13486v3.pdf}{Example 3.5}} when a $K3$ surface $X$ is not isomorphic to the moduli space $M_H(v)$ of $H$-semistable sheaves for an ample divisor $H$ with a Mukai vector $v = (r,c,d) \in \h^*(X,\bb Z)$ such that $c \in \operatorname{NS} (X)$, $r\geq 0$, and $\gcd(r,cH,d) = 1$ (cf. \cite{HuyBook}*{Proposition 10.20}). 
\begin{example}\label{K3}
    In contrast with the case of (simple) abelian varieties, a $K3$ surface $X$ is not tt-separated in general. For example, if we have the Picard number $\rho(X) \geq 12$, then $X$ admits a $(-2)$-curve (\cite{huybrechts_2016}*{\href{https://www.math.uni-bonn.de/people/huybrech/K3Global.pdf}{Corollary 14.3.8}}).  Moreover, if $X$ contains a $(-2)$-curve and $\aut X$ is infinite, then there are infinitely many $(-2)$-curves in $X$ (\cite{huybrechts_2016}*{\href{https://www.math.uni-bonn.de/people/huybrech/K3Global.pdf}{Corollary 8.4.7}}) and hence $X$ is far from being separated. On the other hand, there are $K3$ surfaces with no $(-2)$-curves (while having the infinite automorphism group) (cf. \cite{huybrechts_2016}*{\href{https://www.math.uni-bonn.de/people/huybrech/K3Global.pdf}{15.2.5}}).

    Note that when $\rho(X) \geq 12$, we have $\FM \perf X =\{X\}$ (\cite{huybrechts_2016}*{\href{https://www.math.uni-bonn.de/people/huybrech/K3Global.pdf}{Corollary 16. 3.8}}), so although such a $K3$ surface has a $(-2)$-curve, the whole Fourier-Mukai locus can be simpler than other $K3$ surfaces with multiple Fourier-Mukai partners. 
\end{example}
\subsection{Toric varieties}
Let us briefly look at cases of toric varieties. First, let us recall the following result:
\begin{theorem}[{\cite{Kawtoric}*{\href{https://projecteuclid.org/journals/michigan-mathematical-journal/volume-62/issue-2/Derived-categories-of-toric-varieties-II/10.1307/mmj/1370870376.full}{Theorem 5}}}] \label{kawamata toric} Let $X$ be a smooth projective toric variety and $Y$ a smooth projective variety. Suppose there is a triangulated equivalence $\Phi:\perf X \simeq \perf Y$. Then, $Y$ is also a toric variety and $\Phi$ is birational (after shifting) with the associated birational map being a toric map. Moreover, there exist only finitely many such birational maps when $X$ is fixed and $Y$ is varied.    
\end{theorem}
As a direct corollary, we have the following:
\begin{lemma}\label{toric}
Let $\cal T$ be a triangulated category with a toric variety $X \in \FM \cal T$. Then any $Y \in \cal T$ is toric and any equivalence $\perf X \simeq \perf Y$ is birational up to shifts. In particular, $X$ is tt-irreducible and $\fmspec \perf X$ is irreducible.
\end{lemma} 

\begin{remark}\label{toricrem}
    Note Theorem \ref{kawamata toric} is not enough to show that $X$ is tt-quasi-compact since in general there may be birational autoequivalences such that the associated birational maps agree, but the embedding into the triangular spectrum is different over the complement of the domain of definition. Also, note $X$ is not tt-separated in general by Corollary \ref{-2 twist} since a toric surface may contain a $(-2)$-curve.
\end{remark}
Note that a toric variety has a big anti-canonical bundle (for example \cite{BPtoric}*{\href{https://arxiv.org/pdf/1010.1717.pdf}{Lemma 11}}). Indeed, any smooth projective variety with big (anti-)canonical bundle is tt-irreducible by the following result:
\begin{lemma}[Kawamata] \cite{Uehara2017ATF}*{\href{https://arxiv.org/pdf/1704.00292.pdf}{Proposition 4.4}}
    Let $X$ be a smooth projective variety with big (anti-) canonical bundle. Then, $X$ is of $K$-equivalent type, i.e.,  \[
    \operatorname{BAuteq} X \times \bb Z[1] = \auteq \perf X.
    \]
    In particular, $X$ is tt-irreducible. Moreover, $\fmspec \perf X$ is irreducible. 
\end{lemma}
\begin{question}
    If $X$ is tt-irreducible, what can we say about positivity of its (anti-)canonical bundle? 
\end{question}
\subsection{Flops and K-equivalences}\label{flops and K-equiv}
Now, let us move on to a more general approach to study Fourier-Mukai loci using birational geometry. First of all, let us note the following easy observation:
\begin{lemma}\label{birat non-iso}
    Let $\cal T$ be a triangulated category with $X,Y \in \FM \cal T$. Suppose $X$ and $Y$ are birational Fourier-Mukai partners that are not isomorphic to each other. Then, for a birational triangulated equivalence $\xi \inv \circ \eta: \perf X \simeq \perf Y$ for $\eta:\perf X \simeq \cal T$ and $\xi:\perf Y\simeq \cal T$, we have $$\spec_{\tens_{X,\eta}} \cal T \cap \spec_{\tens_{Y,\xi}} \cal T \neq \emp,$$ which is a strict open subscheme of both $X$ and $Y$. In particular, $\cal T$ is not of disjoint Fourier-Mukai type and hence $\spec^\fm \cal T$ is not separated.
\end{lemma}
\begin{proof}
    The intersection is not empty since $\xi \inv \circ \eta$ is birational. Also, the intersection cannot be the whole space since $X$ and $Y$ are not irreducible and not isomorphic to each other. 
\end{proof}
To see some examples of birational Fourier-Mukai partners that are not isomorphic to each other, let us introduce the following standard notions in birational geometry (e.g. \cite{Kawamata2002DEquivalenceAK}, \cite{KorMor98}, and \cite{toda_2006}).
\begin{definition}
    Let $X$ and $Y$ be projective varieties with only canonical singularities. 
    \begin{enumerate}
        \item A birational map $\alpha:X\ratmap Y$ is said to be \textbf{crepant}  if there exist a smooth projective variety $Z$ and birational morphisms $f:Z \to X$ and $g:Z \to Y$ such that  $\alpha\circ f = g$ and $f^*K_X \sim_\bb Q g^*K_Y$.
        \item A birational map $\alpha:X\ratmap Y$ is said to be a \textbf{flop} if there exist a normal projective variety $W$ and crepant birational morphisms $\phi:X \to W$ and $\psi:Y \to W$ such that
        \begin{itemize}
            \item $\phi = \psi \circ \alpha$;
            \item $\phi$ and $\psi$ are isomorphisms in codimension $1$;
            \item the relative Picard numbers of $\phi$ and $\psi$ are $1$;
            \item for any $\phi$-ample divisor $H$ on $X$, $-H'$ is $\psi$-ample, where $H'$ is the strict transform of $H$ on $Y$. \qedhere
        \end{itemize}
    \end{enumerate}
\end{definition}
Here are some examples of flops and applications to our situations.
\begin{example}\label{flop} \ 
    \begin{enumerate}
        \item (See \cite{HuyBook}*{Section 11.3} for details.) Let $X_1$ be a smooth projective variety of dimension $2k+1$ and suppose $X_1$ contains a closed subvariety $Y_1$ isomorphic to $\bb P^k$ whose normal bundle is isomorphic to $\ecal O_{\bb P^k}(-1)^{\oplus k+1}$. Now, writing the blow-up along $Y_1$ by $p_1:\tilde X \to X_1$, we see that its exceptional divisor $E$ is isomorphic to $\bb P^k \times \bb P^k$ and moreover that we can blow-down to the other projection, which yields a birational morphism $p_2:\tilde X \to X_2$. Then, the birational map $p_2 \circ p_1\inv:X_1 \ratmap X_2$ is called the \textbf{standard flop}. In short, the {standard flop} is constructed so that it fits into the following diagram:
        \begin{center}\DisableQuotes
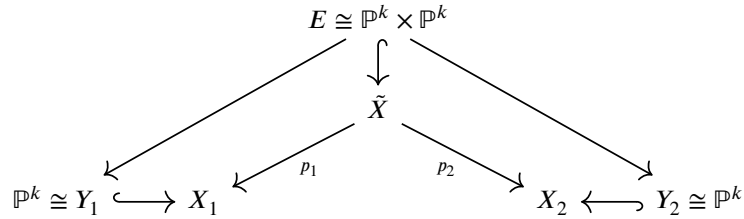
            \begin{tikzcd}
                                 &     & E \iso \bb P^k \times \bb P^k \arrow[d, hook] \arrow[lldd] \arrow[rrdd] &     &                                 \\
                                 &     & \tilde X \arrow[ld, "p_1"] \arrow[rd, "p_2"']                           &     &                                 \\
\bb P^k \iso Y_1 \arrow[r, hook] & X_1 &                                                                         & X_2 & Y_2 \iso\bb P^k \arrow[l, hook]
\end{tikzcd}
        \end{center}
        where the normal bundle of $Y_i \iso \bb P^k$ in $X_i$ is isomorphic to $\ecal O_{\bb P^k}(-1)^{\oplus k+1}$ (so that $\omega_{X_i}|_{\bb P^k} \iso \ecal O_{\bb P^k}$) and $p_i$ is the blow-up along $Y_i$ with exceptional divisor $E$. In particular, it can be shown that the birational map $p_2 \circ p_1\inv:X_1 \ratmap X_2$ is indeed a flop by considering contractions $X_i \to W$ of $\bb P^k$ for each $i$. Here, note that $X_2$ is not necessarily projective in general (cf. Remark \ref{complicated birational}), so let us suppose $X_2$ is projective in this paper. By \cite{bondal1995semiorthogonal}*{\href{https://arxiv.org/abs/alg-geom/9506012}{Theorem 3.6}}, the functor
        \[
        \Phi = {p_2}_*\circ {p_1}^* = \Phi_{\ecal O_Z}:\perf X_1 \to \perf X_2
        \]
        is an equivalence, where $Z = X_1 \times_W X_2 \subset X_1\times X_2$. Since $p_1$ and $p_2$ are isomorphisms away from exceptional loci, $\Phi$ is moreover birational and indeed the associated birational map $X_1 \ratmap X_2$ is the isomorphism $X_1\setminus Y_1 \iso X_2\setminus Y_2$, i.e., the standard flop itself (cf. \cite{HO22}*{\href{https://arxiv.org/pdf/2112.13486v3.pdf}{Lemma 3.7}}). Note in general $X_1$ and $X_2$ are not isomorphic, in which case Lemma \ref{birat non-iso} reproduces \cite{HO22}*{\href{https://arxiv.org/pdf/2112.13486v3.pdf}{Example 3.9}} and simplifies a proof of \cite{HO22}*{\href{https://arxiv.org/pdf/2112.13486v3.pdf}{Remark 3.10}} although the proofs are essentially the same. 
        \item (See \cite{HuyBook}*{Section 11.4} for details.) Let $X_1$ be a smooth projective variety of dimension $2k$ and suppose $X_1$ contains a closed subvariety $Y_1$ isomorphic to $\bb P^k$ whose normal bundle is isomorphic to the cotangent bundle $\Omega_{Y_1}$. As in the construction of the standard flop, we can construct a birational map  $X_1 \ratmap X_2$, called the \textbf{Mukai flop}, that fits into the following diagram:
        \begin{center}\DisableQuotes
            \begin{tikzcd}
                                 &     & E \iso \bb P(\Omega_{Y_1}) \arrow[d, hook] \arrow[lldd] \arrow[rrdd] \arrow[r, hook] & Y_1\times Y_2 &                                 \\
                                 &     & \tilde X \arrow[ld, "p_1"] \arrow[rd, "p_2"']                                        &               &                                 \\
\bb P^k \iso Y_1 \arrow[r, hook] & X_1 &                                                                                      & X_2           & Y_2 \iso(\bb P^k)^* \arrow[l, hook]
\end{tikzcd}
        \end{center}
        where the normal bundle of $Y_i$ in $X_i$ is isomorphic to $\Omega_{Y_i}$, $(\bb P^k)^*$ denotes the dual projective space of $\bb P^k$, $p_1$ is the blow-up of $X_1$ along $Y_1$, and $p_2$ is a birational morphism contracting $E$ to the other direction, noting $\ecal O_E(E) \iso \ecal O_E(-1,-1)$ and the Fujiki-Nakano criterion (cf. \cite{HuyBook}*{Remark 11.10}). Again, $X_2$ is not necessarily projective, so let us suppose it is projective in this paper. However, in contrast to the standard flop, the functor 
        \[
        \Phi = {p_2}_*\circ {p_1}^*:\perf X_1 \to \perf X_2
        \]
        is \textbf{not} even fully faithful for $n>1$ by \cite{nami}. In the paper, Namikawa nevertheless proves the following: First, consider the contractions $X_i \to W$ of $Y_i$ and write $Z' = X_1\times_W X_2$ with projections $q_i :\hat X \to X_i$. Then, the functor
        \[
        \Psi = {q_2}_*\circ {q_1}^* = \Phi_{\ecal O_{Z'}}:\perf X_1 \to \perf X_2
        \]
        is an equivalence. In particular, since $q_1$ and $q_2$ are isomorphisms away from exceptional loci, $\Psi$ is birational and indeed the associated birational map $X_1 \ratmap X_2$ is the isomorphism $X_1\setminus Y_1 \iso X_2\setminus Y_2$, i.e., the Mukai flop itself.  
        \item Let us focus on the case of threefolds. First, recall the following important results: 
        \begin{enmthm}[Koll\'ar]\cite{kollar1989flops}*{\href{https://www.cambridge.org/core/services/aop-cambridge-core/content/view/S0027763000001240}{Theorem 4.9.}}\label{kollar flop}
            Let $X_1$ and $X_2$ be projective threefolds with $\bb Q$-factorial terminal singularities. Suppose $K_{X_1}$ and $K_{X_2}$ are both nef. Then any birational map $X_1 \ratmap X_2$ can be written as a finite composite of flops. Moreover, each flop does not change analytic singularity and only involves at worst terminal singularities for flopping contractions (cf. \cite{kollar1989flops}*{\href{https://www.cambridge.org/core/services/aop-cambridge-core/content/view/S0027763000001240}{Theorem 2.4}}). 
        \end{enmthm}
        \begin{enmthm}[Bridgeland]\cite{Bridgeland_flop}*{Theorem 1}\label{birdgeland flop}
            If $X$ is a projective threefold with terminal singularities and 
            \begin{center}
                \DisableQuotes
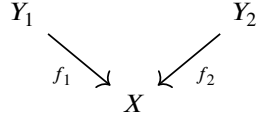
                \begin{tikzcd}
Y_1 \arrow[rd, "f_1"'] &   & Y_2 \arrow[ld, "f_2"] \\
                       & X &                      
\end{tikzcd}
            \end{center}
            are crepant birational morphisms with smooth projective varieties $Y_1,Y_2$, then there is an equivalence $\perf Y_1 \simeq \perf Y_2$. Moreover, such an equivalence can be taken to be birational.  
        \end{enmthm}
        In particular, if $X_1$ and $X_2$ are birationally equivalent smooth projective threefolds with nef canonical bundles (e.g. smooth projective Calabi-Yau threefolds), then there is a birational triangulated equivalence $\perf X_1 \simeq \perf X_2$. In other words, each irreducible component of $\spec^\fm \perf X_1$ containing a copy of $X_1$ contains all the copies of smooth projective threefolds with nef canonical bundles that are birationally equivalent to $X_1$. In particular, a smooth projective threefold with nef canonical bundle is not tt-separated in general.\qedhere
    \end{enumerate}
\end{example} 
\begin{remark}\label{complicated birational} \ 
\begin{enumerate}
    \item Given the first two examples, Orlov-Bondal conjectured that if smooth varieties $X$ and $Y$ are related by flops, then they are Fourier-Mukai partners (\cite{bondal1995semiorthogonal}). Note this is a special case of Conjecture \ref{dk hypothe}. On the other hand, a choice of a Fourier-Mukai kernel is a delicate problem. In the first two examples, given a flop $X_1 \rightarrow W \leftarrow X_2$, we can always take a kernel to be $\ecal O_Z$ for $Z = X_1\times_W X_2 \subset X_1 \times X_2$, but Namikawa showed that this is not the case for a stratified Mukai flop in \cite{namikawa2003mukai} although it is shown that there is a (birational) equivalence of derived categories corresponding to such a flop in \cite{cautis2013derived}.  
    \item Theorem \ref{kollar flop} is generalized to minimal models of general dimension by Kawamata in \cite{kawamata2008flops} and Theorem \ref{birdgeland flop} is generalized to more general singularities in \cite{Chen2002FlopsAE} and \cite{VdB_flop}. In particular, Van den Bergh uses a different method in his proof, which contributed to the development of noncommutative crepant resolutions. 
    \item In the last example, we have seen that any birationally equivalent Calabi-Yau threefolds are Fourier-Mukai partners, but the converse is not true. In \cite{ottem2018counterexample}*{\href{https://www.duo.uio.no/bitstream/handle/10852/58499/main.pdf?sequence=2&isAllowed=y}{Theorem 4.1}}, it is shown that there are Calabi-Yau threefolds that are Fourier-Mukai partners, but not birationally equivalent. In particular, $X$ is not tt-irreducible in general. More generally, Uehara also produces threefolds of Kodaira dimension $1$ that are Fourier-Mukai partners, but not birationally equivalent (\cite{Ueh3dcounter}).    \qedhere 
\end{enumerate}
\end{remark}
Now, by the symmetry of flops, note that Example \ref{flop} (and conjecturally any flop not changing class of singularity) produces autoequivalences by applying triangulated equivalences corresponding to flops back and forth. Those autoequivalences, called \textbf{flop-flop autoequivalences}, can indeed be non-standard autoequivalences and thus their actions on tt-spectra would be interesting.  
\begin{example} We will use the same notations as in Example \ref{flop} in each respective part.
\begin{enumerate}
    \item Consider the following diagram of the standard flop: 
    \begin{center}\DisableQuotes
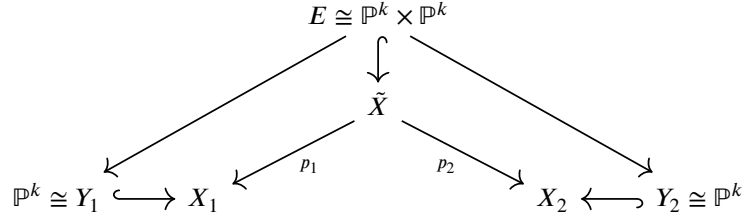
            \begin{tikzcd}
                                 &     & E \iso \bb P^k \times \bb P^k \arrow[d, hook] \arrow[lldd] \arrow[rrdd] &     &                                 \\
                                 &     & \tilde X \arrow[ld, "p_1"] \arrow[rd, "p_2"']                           &     &                                 \\
\bb P^k \iso Y_1 \arrow[r, hook] & X_1 &                                                                         & X_2 & Y_2 \iso\bb P^k \arrow[l, hook]
\end{tikzcd}
        \end{center}
        Then, we get an autoequivalence
        \[
        \Phi = {p_1}_*{p_2}^*\circ {p_2}_*{p_1}^*:\perf X_1 \overset{\sim}{\to} \perf X_2 \overset{\sim}{\to} \perf X_1. 
        \]
        Note that $\ecal O_{\bb P^k}(i)$ is a spherical object in $\perf X_1$ for any $k \in \bb Z$ (e.g. \cite{HuyBook}*{Example 8.10 (v)}). By \cite{addington2019mukai}*{Theorem A}, we have 
        \[
        \Phi = T\inv_{\ecal O_{\bb P^k}(-1)} \circ T\inv_{\ecal O_{\bb P^k}(-2)} \circ \cdots \circ T\inv_{\ecal O_{\bb P^k}(-k)}.  
        \]
        More generally, any flop-flop autoequivalence
        \[
        {p_1}_*(\ecal O_{\tilde X}(mE)\tens {p_2}^*(-))\circ {p_2}_*(\ecal O_{\tilde X}(nE)\tens {p_1}^*(-))
        \]
        can be written as compositions of (inverses of) spherical twists around $\ecal O_{\bb P^k}(l)$ for some $l$. In particular, any flop-flop autoequivalence associated to a standard flop glue copies of $X_1$ along an open neighborhood of $X_1\setminus Y_1$. Here, note that a priori the open neighborhood can be strictly larger than $X_1\setminus Y_1$ since Corollary \ref{birational spherical} can only tell behaviors on $X_1\setminus Y_1$ and it is a priori possible that a skyscraper sheaf $k(x)$ at some $x \in Y_1$ gets mapped to a skyscraper after several compositions of spherical twists. In particular, if we suppose $X$ is a threefold so that $k = 1$ and $Y_i$ are $(-1,-1)$-curves (in which case the standard flop is called the \textbf{Atiyah flop}), then we have $\Phi = T\inv _{\ecal O_{\bb P^1}(-1)}$, i.e., $\Phi$ is the inverse of the spherical twist and we can apply Corollary \ref{birational spherical}. 
        \item In the same paper, the authors show similar results for Mukai flops by using $\bb P$-twists instead of spherical twists (\cite{addington2019mukai}*{Theorem B}). In particular, let us consider the case when $k = 1$, i.e., when $X_1$ is a surface. Then, the Mukai flop indeed does not do anything, but the corresponding flop-flop autoequivalence ${q_1}_*{q_2}^*\circ {q_2}_*{q_1}^*$ is the inverse of the spherical twist $T_{\ecal O_{\bb P^1}(-1)}$ so we can apply Corollary \ref{birational spherical}. In higher dimensions, it will be interesting to investigate $\bb P$-twists more closely. 
        \item In addition to the Atiyah flops, which flop $(-1,-1)$-curves in threefolds, we can also think of flopping other types of rational curves in threefolds, namely, $(-2,0)$-curves and $(-3,1)$-curves, which are the only possibilities (cf. \cite{pinkham1983factorization}*{Proposition 2}). Note, in those cases, the corresponding sheaves are not spherical, so the straightforward generalization cannot be expected. On the other hand, Toda showed that the corresponding flop-flop autoequivalences can be described as generalized spherical twists \cite{Toda2007}. Moreover, the case of $(-3,1)$-curves is considered by Donovan-Wemyss \cite{DW16}, where they define noncommutative generalization of twist functors. It turns out that cases of $(-1,-1)$-curves and $(-2,0)$-curves end up staying in commutative world while we strictly need noncommutative methods for the case of $(-3,1)$-curves. It will also be interesting to investigate their actions on tt-spectra in the future.   \qedhere
\end{enumerate}
\end{example}
Now, we will see that birational Fourier-Mukai partners are closely related to a well-studied topic in birational geometry, namely $K$-equivalences. First, let us recall the definition.  
\begin{definition}
    Let $X$ and $Y$ be smooth projective varieties. We say $X$ and $Y$ are \textbf{$K$-equivalent} if there is a crepant birational map $X \ratmap Y$. 
\end{definition}
\begin{remark}
    Any birationally equivalent minimal models are $K$-equivalent by \cite{wang1998topology}*{\href{https://arxiv.org/pdf/math/9804050.pdf}{Corollary 1.10}}. 
\end{remark}
One of the main conjectures about $K$-equivalences is the following:
\begin{conjecture}[DK hypothesis]\label{dk hypothe}\cite{Kawamata2002DEquivalenceAK}
    Let $X$ and $Y$ be smooth projective varieties. If $X$ and $Y$ are K-equivalent, then they are Fourier-Mukai partners. Conversely, if $X$ and $Y$ are birationally equivalent Fourier-Mukai partners and have non-negative Kodaira dimension, then they are K-equivalent.
\end{conjecture}
\begin{remark}\label{hk remark}
    One way to approach  the first part of this conjecture is the following. First, it is conjectured by Kawamata (\cite{kawamata2017birational}*{Conjecture 3.6}) that any $K$-equivalence factors into compositions of flops, which is shown for $K$-equivalences between threefolds (\cite{Kawamata2002DEquivalenceAK}), minimal models (\cite{kawamata2008flops}) and toric varieties (\cite{Kawtoric}), respectively. Now, combined with the conjecture that varieties related by a flop are Fourier-Mukai partners by Bondal-Orlov \cite{bondal1995semiorthogonal}, we obtain the first part of the conjecture. Moreover, Example \ref{flop} suggests that such varieties are birational Fourier-Mukai partners. 
\end{remark}

Now $K$-equivalence comes into our contexts due to the following result:
\begin{theorem}\label{bfm is ke}\cite{toda_2006}*{\href{https://www.cambridge.org/core/services/aop-cambridge-core/content/view/31DA2A4009AE544E3E9A0B88A001E937/S0010437X06001977a.pdf/fourier-mukai-transforms-and-canonical-divisors.pdf}{Lemma 7.3.}}
    Let $X$ and $Y$ be smooth projective varieties. If $X$ and $Y$ are birational Fourier-Mukai partners, then they are $K$-equivalent under the birational map associated to a birational triangulated equivalence. 
\end{theorem}
\begin{remark}\label{Uehara K-equiv}
    For this reason, a birational triangulated equivalence up to shifts is called a triangulated equivalence of $K$-equivalent type by Uehara in \cite{Uehara2017ATF}. 
\end{remark}
We can rephrase this result as follows:
\begin{corollary}
    Let $\cal T$ be a triangulated category and suppose $X, Y \in \FM \cal T$. If there is a prime thick subcategory of $\cal T$ that is an ideal with respect to both tt-structures transported from $(\perf X,\tens_{\ecal O_X}^\bb L)$ and $(\perf Y,\tens_{\ecal O_Y}^\bb L)$, then $X$ and $Y$ are $K$-equivalent. 
\end{corollary}
Let us make several observations on the relation of Fourier-Mukai loci with the DK hypothesis. 
\begin{obs}
    We have the following natural question arising from Theorem \ref{bfm is ke}:
    \begin{indq}\label{Mukai flop}
    If $X$ and $Y$ are $K$-equivalent Fourier-Mukai partners, then are they birational Fourier-Mukai partners? 
\end{indq}
Note arguments in Remark \ref{hk remark} suggest the question has a positive answer.  Indeed, in case Question \ref{Mukai flop} has a positive answer, one direction of the DK hypothesis is equivalent to non-emptiness of intersections of tt-spectra of birationally equivalent Fourier-Mukai partners by Lemma \ref{main intersection} and Theorem \ref{bfm is ke}. In other words, we propose the following version of the DK hypothesis:
\begin{indc}[tt-DK hypothesis]\label{deep dk}
        Let $X$ be a smooth projective variety. Then, every connected component of $\fmspec \perf X$ containing a copy of $X$ contains all smooth projective varieties that are $K$-equivalent to $X$ as open subschemes.   
\end{indc}
Note in particular that this version of the DK hypothesis has a topological necessary condition to be true. 
\end{obs}

For surfaces, $K$-equivalence is indeed a quite strong condition. 
\begin{lemma}\cite{HuyBook}*{Corollary 12.8.}
    If $X$ and Y are $K$-equivalent smooth projective surfaces, then they are isomorphic. In particular, any birational Fourier-Mukai partners are isomorphic. 
\end{lemma}
As a corollary, we obtain generalization of Lemma \ref{K3 disjoint}:
\begin{corollary}\label{surface FM disjoint}
    Let $\cal T$ be a triangulated category with a surface $X \in \FM \cal T$. Then $\cal T$ is of disjoint Fourier-Mukai type and hence we have 
    \[
    \spec^\fm \cal T = \bigsqcup_{X \in \FM \cal T } \spec_{\tens,X} \cal T. \qedhere
    \]
\end{corollary}
Finally, let us summarize notions and implications we have seen so far.
\begin{obs}\label{bg obs}
     Let $X$ and $Y$ be smooth projective varieties. Write relations of $X$ and $Y$ as follows:
    \begin{enumerate}
        \item[ISO:] $X$ and $Y$ are isomorphic;
        \item[TTE:] There exists a triangulated category $\cal T$ such that $X,Y \in \FM \cal T$ and $\spec_{\tens,X} \cal T = \spec_{\tens,Y}\cal T \subset \spec_\vartriangle \cal T$;
        \item[BFM:] $X$ and $Y$ are birational Fourier-Mukai partners, or equivalently there exists a triangulated category $\cal T$ such that $X,Y \in \FM \cal T$ and $\spec_{\tens,X} \cal T \cap \spec_{\tens,Y}\cal T \neq \emp$;
        \item[KEFM:] $X$ and $Y$ are $K$-equivalent Fourier-Mukai partners;
        \item[BEFM:] $X$ and $Y$ are birationally equivalent Fourier-Mukai partners;
        \item[FM:] $X$ and $Y$ are Fourier-Mukai partners;
        \item[KE:] $X$ and $Y$ are $K$-equivalent;
        \item[BE:] $X$ and $Y$ are birationally equivalent.
    \end{enumerate}
    Now, we clearly have the following implications (black arrows): 
    \begin{center}
        \DisableQuotes
        \begin{tikzcd}
    &      &     &      & \text{BEFM} \arrow[r, Rightarrow] \arrow[rdd, Rightarrow] \arrow[dd, blue, "\textcolor{black}{\text{DK hyp.}}" description, Rightarrow] \arrow[ld, red, "\textcolor{black}{\text{(ii)}}" description, Rightarrow, bend right, shift right]           & \text{FM} \arrow[llld,green, "\textcolor{black}{X,Y:\text{with big (anti-)canonical bundle}}" description, Rightarrow, bend right, shift right=5] \arrow[llllld, green, "\textcolor{black}{X,Y:\text{{with ample (anti-)canonical bundle}}}" description, Rightarrow, bend right, shift right=7] \\
\text{ISO} \arrow[r, Rightarrow] & \text{TTE} \arrow[r, Rightarrow] \arrow[l, blue, "\textcolor{black}{\text{Q.\ref{tte}}}"', Rightarrow, bend right, shift right=2] & \text{BFM} \arrow[r, Rightarrow] \arrow[ll,red, "\textcolor{black}{\text{(i)}}" description, Rightarrow, bend left, shift left] & \text{KEFM} \arrow[ru, Rightarrow] \arrow[rd, Rightarrow] \arrow[l, blue, "\textcolor{black}{\text{Q.\ref{Mukai flop}}}"', Rightarrow, bend right, shift right=2] &  &   \\
&                                                                                                                 &                                                                                                          &                                                                                                                                                 & \text{KE} \arrow[r, Rightarrow] \arrow[lu, blue, "\textcolor{black}{\text{DK hyp.}}" description, Rightarrow, bend left, shift left] \arrow[llllu,green, "\textcolor{black}{X,Y:\text{surfaces}}" description, Rightarrow, bend left, shift left=2] & \text{BE} \arrow[lllu, green, "\textcolor{black}{X,Y:\text{CY threefolds}}"description, very near end, Rightarrow, bend left=49] \arrow[lllllu, green, "\textcolor{black}{X,Y:\text{minimal surfaces (e.g. $K3$ surfaces), abelian varieties}}" description, Rightarrow, bend left=49, shift left=2] 
\end{tikzcd}
    \end{center}
    where blue arrows are implications under question/conjecture, green ones are ones that hold under specified conditions, and red ones are ones with counter-examples listed below:
    \begin{enumerate}
        \item This is equivalent to saying that there is a triangulated category $\cal T$ that is not of disjoint Fourier-Mukai type, which we have already seen (e.g. standard/Mukai flops between non-isomorphic varieties or any birationally equivalent non-isomorphic Calabi-Yau threefolds). This shows at least one of BFM $\Rightarrow$ TTE and TTE $\Rightarrow$ ISO is false, where the former is more likely to be false since it is sufficient to find $K$-equivalent Fourier-Mukai partners that are not birationally equivalent around some points.
        \item In \cite{Ue04minimalelliptic}, Uehara constructed a family of rational elliptic surfaces that are Fourier-Mukai partners (and birationally equivalent), which are not $K$-equivalent. Note they have Kodaira dimension equal to $-\infty$ and hence they do not contradict with the DK hypothesis. \qedhere
    \end{enumerate}

\end{obs}

\section{Comparison with Serre invariant loci}\label{Comparison with Serre invariant loci}
What we have seen so far has been geometric in the sense that using a tt-structure that is actually coming from a tt-category $(\perf X, \tens_{\ecal O_X})$ is equivalent to specifying $X$ by Theorem \ref{balmer}. Indeed, as we have seen, many results in the last section were obtained by applications of birational geometry. In this section and the next section, we will try to give more categorical descriptions of Fourier-Mukai loci, which could potentially lead to reverse applications to birational geometry. In this section, we will focus on Serre invariant loci of triangular spectra introduced in \cite{HO22}. 
\begin{definition} Let $\cal T$ be a triangulated category with a Serre functor $\bb S$. A prime thick subcategory $\cal P$ of $\cal T$ is said to be \textbf{Serre invariant} if $\bb S(\cal P) = \cal P$. Let 
        \[\spc^\ser  \cal T \subset \spc_\vartriangle \cal T\] denote the subspace of Serre-invariant prime thick subcategories. Since the thick subcategories are replete, the Serre invariant locus does not depend on a choice of Serre functors. 
\end{definition}
\begin{notation}
    \textbf{In the rest of this paper}, we put the subspace topology on $\spc^\fm \cal T$ in $\spc_\vartriangle \cal T$, which is a priori coarser than the original topology, but it turns out these two topologies in fact agree by \cite{ito2024new}.
\end{notation}
We have the following quick observation.
\begin{corollary}\label{FM preservation}
    For any triangulated category $\cal T$, we have $\spc^\fm \cal T \subset \spc^\ser  \cal T \subset \spc_\vartriangle \cal T$. Moreover, any triangulated equivalence $\tau:\cal T \simeq \cal T'$ induces a homeomorphism $\spc_\vartriangle \cal T\iso \spc_\vartriangle \cal T'$ that restricts to homeomorphisms on the corresponding loci.   
\end{corollary}
\begin{proof}
The inclusion $\spc^\fm \cal T \subset \spc^\ser  \cal T$ follows by Lemma \ref{universal recovery}. The latter claims follow from commutativity of a Serre functor with triangulated equivalences and the construction of Fourier-Mukai loci, respectively. 
\end{proof}
The following lemma is a useful criteria for checking if a thick subcategory is Serre invariant. This should be known to experts, but let us include a proof for completeness. 
\begin{lemma}\label{admissible serre invariant}
    Let $\cal T$ be a triangulated category with a Serre functor $\bb S$. Moreover, suppose $\cal T$ is \textbf{connected}, i.e., cannot be written as a direct sum of nontrivial triangulated subcategories. Then, any nontrivial admissible subcategory $\cal I$ of $\cal T$ is not Serre invariant.    
\end{lemma}
\begin{proof}
    By the definition of a Serre functor, if an admissible subcategory $\cal I$ is Serre-invariant, then the right orthogonal and the left orthogonal agree, which implies $\cal T = \cal I \oplus \cal I^\perp$ since for any $X \in \cal T$, there is a split triangle $Y \to X \to Z \overset{0}{\to}Y[1]$ with $Y\in \cal I$ and $Z \in \cal I^\perp = {}^\perp\cal I$. Thus, since $\cal T$ is connected, such an $\cal I$ should be trivial. 
\end{proof}
In particular,  combined with Corollary \ref{FM preservation}, this provides another proof for \cite{HO22}*{\href{https://arxiv.org/pdf/2112.13486.pdf}{Lemma 5.7}}. 
\begin{example}\label{serre invariant example}  \ 
\begin{enumerate}
    \item  \cite{HO22}*{\href{https://arxiv.org/pdf/2112.13486.pdf}{Remark 5.2}}If $\cal T = \perf X$ for a smooth projective variety $X$ with trivial canonical bundle, then $\bb S = -\tens_{\ecal O_X}[\dim X]$ and hence \[\spc^\ser  \cal T = \spc_\vartriangle \cal T.\] 
    \item If $\cal T = \perf k  Q $ for a quiver $ Q $ of type $A_2$, then by Example \ref{a2}, we can see 
    \[
    \spc^\ser  \cal T = \emp.  
    \]
    As pointed out right after \cite{GrSt_2023}*{\href{https://arxiv.org/pdf/2205.13356v2.pdf}{Example 7.1.7}} by Gratz-Stevenson, the same results hold for any quiver of type $\sf A_n$ and $\sf D_n$. More generally, since any thick subcategory of $\perf k\sf Q$ for any Dynkin quiver $\sf Q$ is admissible (for example, by \cite{ingalls_thomas_2009}), the same results hold by Lemma \ref{admissible serre invariant}. \qedhere
\end{enumerate}
\end{example}
One important observation is the following although we are not using the result in this paper: 
\begin{theorem}[Hirano-Ouchi] \cite{HO22}*{\href{https://arxiv.org/pdf/2112.13486.pdf}{Proposition 5.6}}\label{FM ample}
    Let $\cal T$ be a triangulated category with a smooth projective variety with ample (anti-)canonical bundle $X \in \FM \cal T$. Then, the inclusion \[X \iso \spc_{\tens_{X,\eta}} \cal T =  \spc^\fm \cal T \inj \spc^\ser  \cal T.\qedhere\] 
\end{theorem}
\begin{remark}[\cite{HO22}*{\href{https://arxiv.org/pdf/2112.13486.pdf}{Remark 5.5}}]
    Use the same notation as in Theorem \ref{FM ample}. Now, since either if $\dim X \geq 2$ and $k$ is uncountable or if $\dim X \geq 3$ and $|k|>2$, then we know that $X$ is completely determined by its underlying topological space by \cite{kollar2023determines}*{Theorem 4.1.14 (i), (ii)}, Theorem \ref{FM ample} reproves the reconstruction theorem by Bondal-Orlov in those cases. 
\end{remark}
Now the following is the main result in this section:
\begin{theorem}\label{curves}
    Let $\cal T$ be a triangulated category with a smooth projective curve $X \in \FM\cal T $. Then, \[\spc^\fm \cal T = \spc^\ser  \cal T.\qedhere\] 
\end{theorem}
For a proof let us introduce some notions and results discussed in \cite{HO22}.
\begin{construction}\label{elliptic moduli}
    Let $X$ be an elliptic curve. Let $I := \{(r,d) \mid r>0,\gcd(r,d) = 1\}$ and $M(r,d)$ denote the fine moduli space of $\mu$-semistable sheaves having Chern character $(r,d)$ with universal family $\ecal U_{r,d} \in \coh(X\times M(r,d))$. Let
    \[
    \eta_{r,d}: \perf M(r,d) \simeq \perf X
    \]
    denote the Fourier-Mukai equivalence given by $\ecal U_{r,d}$. Then, we have the corresponding open subscheme 
    \[
    M_{r,d}:= \spec_{\tens_{M(r,d),\eta_{r,d}}} \perf X \subset \spec^\fm \perf X. 
    \]
    By abuse of notation, let $M_{r,d}$ also denote $\spc_{\tens_{M(r,d),\eta_{r,d}}}\perf X$. Note if $\ecal U'_{r,d}$ is another universal family, we have an invertible sheaf $\ecal L$ on $M(r,d)$ such that $\ecal U'_{r,d} \iso \ecal U_{r,d}\tens p^*\ecal L$, where $p:X\times M(r,d) \to M(r,d)$ is a canonical projection. Thus, writing the corresponding Fourier-Mukai equivalence by $\eta_{r,d}'$, we see that $\eta_{r,d}'=\eta_{r,d}\circ(-\tens_{M(r,d)} \ecal L)$ and thus $\eta_{r,d}$ and $\eta_{r,d}'$ define the same subspace of $\spc_\vartriangle \perf X$, so in particular $M_{r,d}$ does not depend on a choice of universal families.
\end{construction}
\begin{lemma}\label{elliptic key} \cite{HO22}*{\href{https://arxiv.org/pdf/2112.13486.pdf}{Lemma 4.13}}
    Let $X$ be an elliptic curve. Then, $\FM(\perf X)$ is a singleton and 
    \[
    \spc_\vartriangle \perf X = \spc^\ser  \perf X = \spc_{\tens_{\ecal O_X}^\bb L} \perf X \sqcup  \bigsqcup_{(r,d) \in I} M_{r,d}.  
    \]
    In particular, we have: 
    \begin{enumerate}
        \item $ M_{r,d} \cap \spc_{\tens_{\ecal O_X}^\bb L} \perf X = \emp$ for any $(r,d) \in I$;
        \item $M_{r,d} \cap M_{r',d'} = \emp$ for any $(r,d)\neq (r',d') \in I$. \qedhere
    \end{enumerate}
\end{lemma}
\begin{remark}\label{Matsui reconstruction of elliptic curve}
    Let us emphasize Matsui's reconstruction result \cite{matsui2023triangular}*{Corollary 4.10}. Namely, by Lemma \ref{noetherian hypothesis}, the equalities above can be made into isomorphisms of ringed spaces:
    \[
    \spec_\vartriangle \perf X \iso \spec_{\tens_{\ecal O_X}^\bb L} \perf X \sqcup  \bigsqcup_{(r,d) \in I} M_{r,d},
    \]
    and hence by looking at connected components, we have reconstructed an elliptic curve from its derived category. 
\end{remark}
Now, we can put everything together to show Theorem \ref{curves}:
\begin{proof}[Proof of Theorem \ref{curves}]
    By Corollary \ref{FM ample} we only need to consider the case when $X$ is an elliptic curve. By Corollary \ref{FM preservation} we may assume $\cal T = \perf X$ and by Construction \ref{elliptic moduli} we have
    \[
    \spc_{\tens_{\ecal O_X}^\bb L} \perf X \cup  \bigcup_{(r,d) \in I} M_{r,d} \subset \spc^\fm \perf X.
    \]
    Thus we conclude by Lemma \ref{elliptic key}.
\end{proof}
\begin{construction}\label{elliptic symmetric}
    Let $\cal T$ be a triangulated category with an elliptic curve $X \in \FM\cal T $. Theorem \ref{curves} is mostly a formal consequence of Lemma \ref{elliptic key}, but it suggests a more symmetric/canonical description of $\spc^\ser  \cal T$:
    \[
\spc^\ser  \cal T = \spc_{\tens,X} \cal T = \bigcup_{\tau \in \auteq\cal T } \spc_{\tens_{X,\tau \circ \eta}} \cal T.
\]
In particular, it clarifies the components $\spc_{\tens_{\ecal O_X}^\bb L} \perf X$ and $M_{r,d}$ are essentially symmetric, i.e., each component simply corresponds to a tt-spectrum of a tt-structure given by a different choice of triangulated equivalences $\perf X \simeq \cal T$, which is more apparent by the following observation. First, recall that the determinant map gives an isomorphism $\det: M(r,d) \to \pic^d(X)$. 
Hence, fixing a line bundle on $X$ of degree $1$ (or equivalently fixing a base point of $X$), we obtain a canonical isomorphism $z_{r,d}: X \iso M(r,d)$ for any $(r,d) \in I$ and let $\zeta_{r,d}: \perf X \to \perf M(r,d)$ denote the corresponding tt-equivalence. Therefore, writing $\tau_{r,d}:= \eta_{r,d} \circ \zeta_{r,d} \in \auteq \perf X$, we have 
    \[
    M_{r,d} = \spec_{\tens_{M(r,d),\eta_{r,d}}} \perf X = \spec_{\tens_{X,\tau_{r,d}}} \perf X  
    \]
    and thus we can write 
    \[
    X_\tens = X_\tens(\id_{\perf X}) \sqcup \bigsqcup_{(r,d) \in I}X_\tens(\tau_{r,d}). 
    \]
    Note this expression of $X_\tens$ shows that the expression in Lemma \ref{elliptic key} corresponds to a specific choice of a family of autoequivalences of $\perf X$ appearing in Theorem \ref{irreducible component}. In particular, this interpretation by Theorem \ref{irreducible component} simplifies the reconstruction in Remark \ref{Matsui reconstruction of elliptic curve} in that we only need Lemma \ref{elliptic key} part (i) for showing connected component consists of a single copy of an elliptic curve. 
\end{construction}
\begin{remark}\label{minimal generator for elliptic curves?} Use the same notations as in Construction \ref{elliptic symmetric} and let $G_{X} \subset \auteq\perf X$ denote the subgroup generated by $\tau_{r,d}$ (which depends on a choice of a base point of $X$). Then for any triangulated equivalence $\eta:\perf X \simeq \cal T$, we have 
    \[
    \spc^\ser  \cal T = \bigcup_{\tau \in G_{X,\eta}} \spc_{\tens_{X,\tau \circ \eta}} \cal T.
    \]
    In particular, $G_{X,\eta}$ is an $\eta$-generator. Let us observe $G_{X,\eta}$ more closely (cf. Remark \ref{minimal generator candidate}). For example, to the author, it is not so straightforward to see if $G_{X,\eta}$ is a minimal $\eta$-generator. We can also ask if different choices of base points of $X$ will indeed yield different subgroups and in that case if we can recover for example a group structure of the elliptic curve $X$ from the set of those subgroups. Note that since $\tau_{r,d}$ are not $\id_{\perf X}$-stabilizers, we see that they are represented by an element of $\sl(2,\bb Z)$, recalling the short exact sequence
    \[
    1 \to  \aut X \ltimes \pic^0 (X) \times \bb Z[2] \to \auteq \perf X \to \sl(2,\bb Z) \to 1 
    \]
    (cf. Example \ref{elliptic twist}). 
\end{remark}
Since the cases of elliptic curves and non-elliptic curves are two extremes, we propose the following conjecture:
\begin{conjecture}\label{main conjecture}    
Let $\cal T$ be a triangulated category with $X \in \FM\cal T$. Then we have
    \[
     \spc^\ser  \cal T = \spc^\fm \cal T. 
    \]
    In particular, $\dim \spc^\ser  \cal T = \dim \spc^\fm \cal T = \dim X$.
\end{conjecture}
\begin{remark}
    Note the conjecture also holds for $\cal T = \perf R$ when $R$ is the path algebra of a Dynkin quiver by Example \ref{serre invariant example} (ii). 
\end{remark}
So far, we have been assuming projectivity of varieties, but it is also natural to only assume properness. So, let us leave several comments in this direction. 
\begin{obs}
    Let $\cal T$ be a triangulated category. It is also natural to consider a \textbf{proper Fourier-Mukai locus} that should sit between the Fourier-Mukai locus and the Serre invariant locus. Naively, we can define it to be  
    \[
    \spc^\sf{pFM} \cal T := \bigcup_{X \in \pfm \cal T } \spc_{\tens,X} \cal T \subset \spc_\vartriangle \cal T, 
    \]
    where $\pfm \cal T $ denotes the set of smooth \textbf{proper} varieties that are Fourier-Mukai partners of $\cal T$, i.e., there exists a triangulated equivalence $\perf X \simeq \cal T$. Clearly, we have the inclusions:
    \[
    \spc^\fm \cal T \subset \spc^\sf{pFM} \cal T \subset \spc^\ser  \cal T.  
     \]
    Since a smooth proper variety of dimension $\leq 2$ is projective, we see that if we have a curve or surface $X \in \FM \cal T$, then $\spc^\fm \cal T = \spc^\sf{pFM} \cal T$. On the other hand, $\spc^\sf{pFM} \cal T$ can be a priori strictly larger than $\spc^\fm \cal T$ since flops of $3$-dimensional varieties may not preserve projectivity while inducing equivalences of derived categories. 
\end{obs}
Thus, we can establish an a priori weaker conjecture:
\begin{conjecture}
    Let $\cal T$ be a triangulated category with $X \in \FM \cal T$ (or $X \in \pfm \cal T$). Then, we have
    \[
    \spc^\sf{pFM} \cal T = \spc^\ser  \cal T. 
    \]
     In particular, $\dim \spc^\ser  \cal T = \dim \spc^\sf{pFM} \cal T = \dim X$, where the last equality is also a part of the conjecture.
\end{conjecture}    
\begin{remark}\label{proper dg}
    From a dg-categorical point of view, the weaker conjecture is indeed more natural since properness and smoothness are categorical properties of a dg-enhancement of a derived category while projectivity is not (e.g. cf. \cite{ORLOV_nc_scheme}). 
\end{remark}
Finally, note there is also a topological restriction for Conjecture \ref{main conjecture} to be true. 
\begin{definition}
    A topological space $X$ is said to be:
    \begin{enumerate}
        \item \textbf{spectral} if it is homeomorphic to the underlying topological space of the affine spectrum of some ring;
        \item \textbf{locally spectral} if it has an open cover by spectral spaces;
        \item \textbf{sober} if any irreducible closed subset has a unique generic point.\qedhere
    \end{enumerate}
\end{definition}
\begin{lemma}[Hochster]\cite{hochster1969prime}*{\href{https://www.ams.org/journals/tran/1969-142-00/S0002-9947-1969-0251026-X/S0002-9947-1969-0251026-X.pdf}{Theorem 9}}
    A topological space is locally spectral if and only if it is homeomorphic to the underlying topological space of a scheme. In particular, a locally spectral space is sober. 
\end{lemma}
Therefore, recalling Theorem \ref{irreducible component}, the following should also hold if Conjecture \ref{main conjecture} holds:
\begin{conjecture}
    Let $\cal T$ be a triangulated category with $X \in \FM \cal T$. Then, $\spc^\ser  \cal T$ is locally spectral (in particular sober) and locally noetherian of dimension $\dim X$. Moreover, the connected components and irreducible components coincide and all of them are homeomorphic to each other.  
\end{conjecture}

\section{Geometric tt-structures and more loci}\label{Geometric tt-structures and more loci} 
In this section, we study Fourier-Mukai loci by giving some attempts to categorically characterize tt-structures that are equivalent to $(\perf X,\tens_{\ecal O_X}^\bb L)$ for some smooth projective variety $X$. In other words, noting Balmer's reconstruction says we can classify/reconstruct Fourier-Mukai partners of $\cal T$ by classifying/specifying "geometric" tt-structures on $\cal T$, our goals of this section are to enlarge a notion of "geometric" tt-structures so that it has a more categorical description and to compare corresponding loci with Fourier-Mukai loci. This approach also gives a way to remedy an uninteresting fact that the Fourier-Mukai locus of a triangulated category $\cal T$ with $\FM \cal T = \emp$ is tautologically empty by enlarging Fourier-Mukai loci so that it could be nonempty even when $\FM \cal T = \emp$. For those purposes, let us first define "geometric" tt-structures geometrically. 
\begin{definition} Let $(\cal T,\tens)$ be a tt-category. A tt-structure $\tens$ is said to be \textbf{geometric} if there exist $X \in \FM\cal T $ and a tt-equivalence $(\eta,\epsilon_\eta): (\perf X,\tens_{\ecal O_X}^\bb L) \simeq (\cal T,\tens)$ such that $\tens \simeq \tens_{X,\eta}$ and a prime thick subcategory of $\cal T$ is said to be \textbf{geometric} if it is a prime ideal for a geometric tt-structure $\tens$. We say tt-structures $\tens$ and $\tens'$ are \textbf{geometrically equivalent} if $\spc_{\tens}\cal T = \spc_{\tens'} \cal T \subset \spc_\vartriangle \cal T$ and let $\gtt \cal T $ denote the set of geometric equivalence classes of geometric tt-structures.  
\end{definition}   

\begin{corollary}
    For a triangulated category $\cal T$, $\spc^\fm \cal T$ coincides with the subspace of geometric prime thick subcategories of $\cal T$. Equivalently, we have 
    \[
    \spc^\fm \cal T = \bigcup_{\tens\in \gtt\cal T } \spc_\tens \cal T. \qedhere
    \]
\end{corollary}
\begin{example} Let $\cal T$ be a triangulated category with $X \in \fm \cal T$.
        If $X$ is a smooth projective variety with ample (anti-)canonical bundle, we see that $\gtt \cal T = \{\tens_{X,\eta}\}$ for any $\eta :\perf X \simeq \cal T$. If $X$ is an elliptic curve, then $\gtt \cal T = \{\tens_{X,\eta}\}\sqcup \bigsqcup_{(r,d) \in I} \{\tens_{X,\tau_{r,d} \circ \eta}\}$ for any $\eta :\perf X \simeq \cal T$ (cf. Construction \ref{elliptic symmetric}). Note if $G\subset \auteq \cal T$ is an $\eta$-generator, then $\gtt \cal T = \bigcup_{\tau \in G}\{\tens_{X,\tau\circ \eta}\}$.  
\end{example}
I suggest three more less geometric notions of "geometric" tt-structures although there can be much more ways to generalize:
\begin{definition}
    Let $(\cal T,\tens)$ be a tt-category with $\spc_\tens \cal T\subset \spc_\vartriangle \cal T$ (cf. Lemma \ref{noetherian hypothesis}). 
    \begin{enumerate}
        \item A tt-structure $\tens$ is said to be \textbf{pseudo-geometric} if $X = \spec_\tens \cal T$ is a{ not necessarily irreducible} smooth projective variety (i.e., a smooth projective reduced scheme over $k$) and a prime thick subcategory of $\cal T$ is said to be \textbf{pseudo-geometric} if it is a prime ideal for an pseudo-geometric tt-structure $\tens$. Let $\pgtt \cal T $ denote the set of geometric equivalence classes of almost geometric tt-structures on $\cal T$. Define the \textbf{pseudo-geometric locus} of $\spc_\vartriangle \cal T$ to be
        \[
        \spc^\pg \cal T := \bigcup_{\tens \in \pgtt \cal T } \spc_\tens \cal T\subset \spc_\vartriangle \cal T.  
        \]
        \item A tt-structure $\tens$ is said to be \textbf{almost geometric} if $X = \spec_\tens \cal T$ is a smooth projective variety and a prime thick subcategory of $\cal T$ is said to be \textbf{almost geometric} if it is a prime ideal for an almost geometric tt-structure $\tens$. Let $\agtt \cal T $ denote the set of geometric equivalence classes of almost geometric tt-structures on $\cal T$. Define the \textbf{almost geometric locus} of $\spc_\vartriangle \cal T$ to be
        \[
        \spc^\ag \cal T := \bigcup_{\tens \in \agtt \cal T } \spc_\tens \cal T\subset \spc_\vartriangle \cal T.  
        \]
        \item A tt-structure $\tens$ is said to be \textbf{Serre-geometric} if a Serre functor can be written as $-\tens F$ for some $F \in \pic(\cal T,\tens)$ and if $\spc_{\tens} \cal T\subset\spc_\vartriangle \cal T$ (cf. Lemma \ref{noetherian hypothesis}). A prime thick subcategory of $\cal T$ is said to be \textbf{Serre-geometric} if it is a prime ideal for a Serre-geometric tt-structure $\tens$. Let $\sgtt \cal T $ denote the set of geometric equivalence classes of Serre-geometric tt-structures on $\cal T$. Define the \textbf{Serre-geometric locus} of $\spc_\vartriangle \cal T$ to be
        \[
        \spc^\sg \cal T := \bigcup_{\tens \in \sgtt \cal T } \spc_\tens \cal T\subset \spc_\vartriangle \cal T.  \qedhere
        \] 
    \end{enumerate}
\end{definition}
\begin{remark}
     Note if a Serre functor can be written as $- \tens K$, then $K$ is necessarily $\tens$-invertible (Lemma \ref{pic}).  
\end{remark}

First let us think about pseudo-geometric and almost geometric tt-structures. By definition, we have the following inclusions. 
\begin{corollary}
    Let $\cal T$ be a triangulated category. Then, we have 
    \[
    \spc^\fm \cal  T \subset \spc^\ag \cal T \subset \spc^\pg \cal T \subset \spc_\vartriangle \cal T. 
    \]
    In particular, if $\spc^\fm \cal T = \spc_\vartriangle \cal T$ (e.g. if an elliptic curve $X \in \FM \cal T$), then all of them agree. 
\end{corollary}
When $\FM \cal T$ may be empty, some of the inclusions may be strict.   
\begin{example}\label{p1} Let $\cal T$ be a triangulated category. 
        \begin{enumerate}
            \item We see the inclusion $\spc^\fm \cal T\subset \spc^\ag \cal T$ can be strict. Define a tt-structure on the category $\vect_k$ of \textbf{finite dimensional} $k$-vector spaces as follows. First, we define a triangulated category structure on $\vect_k$ by setting the shift functor to be the identity functor and by setting a triangle $X \to Y \to Z \to X$ to be distinguished if it is an exact sequence (or equivalently if it is a finite direct sum of (shifts of) the triangle $k\overset{\id}{\to} k \to 0 \to k$). Then, $(\vect_k,\tens_k)$ is a tt-category with the usual tensor product $\tens_k$. Since the only additive proper subcategory of $\vect_k$ is $0$, the underlying topological space of $\spec_{\tens_k} \vect_k$ is a singleton and since $\End_{\vect_k}(k)= Z(\vect_k)_\sf{lrad} =k$, it follows that
            \[
            \spec_\vartriangle \vect_k = \spec_{\tens_k} \vect_k = \spec k
            \]
            as ringed spaces. 
            \item  We see that there are almost geometric structures that are not geometric. We follow the proof of \cite{sosnatensor}*{\href{https://www.math.uni-hamburg.de/home/sosna/diplom-online.pdf}{Proposition 4.0.9}}, which is suggested here \cite{ATessinj}. Let $X$ be a smooth projective variety (or more generally a connected noetherian scheme). Define a tt-structure $\boxtimes$ on $\perf (X\sqcup X)  = \perf X \oplus \perf X$ by setting $$(A,B) \boxtimes (C,D) := (A \tens_{\ecal O_X}^\bb L C, A \tens_{\ecal O_X}^\bb L D \oplus C \tens_{\ecal O_X}^\bb L B),$$
            where the unit is given by $(\bb 1_{\perf X},0)$ with $\bb 1_{\perf X} = \ecal O_X$. Also note that the projection 
            \[
            \Phi: \perf X \oplus \perf X \to \perf X;\quad (A,B) \mapsto A
            \]
            is a tt-functor and hence we have a morphism 
            \[
            \phi:\spec_{\tens_{\ecal O_X}^\bb L} \perf X \to \spec_\boxtimes (\perf X \oplus \perf X); \quad \cal P \mapsto \cal P \oplus  \perf X.  
            \]
            of ringed spaces, which was shown to be a homeomorphism with inverse $\psi$ induced by the inclusion tt-functor
            \[
            \Psi: \perf X \to \perf X \oplus \perf X  ;\quad A \mapsto (A,0). 
            \]
            Now, by construction (cf. Construction \ref{str sheaf of tt}), for any open subset $U \subset \spec_{\tens_{\ecal O_X}^\bb L} \perf X$, $\Phi$ induces an equivalence $$(\perf X\oplus \perf X)(U\oplus \perf X) \simeq \perf X(U)$$ sending $\bb 1_{U\oplus \perf X}=(\bb 1_U,0)$ to $\bb 1_U$ and hence by definition of structure sheaves on tt-spectra, $\Phi$ induces an isomorphism
            $$X \iso \spec_{\tens_{\ecal O_X}^\bb L} \perf X \iso \spec_\boxtimes \perf (X\sqcup X)$$ 
            as ringed spaces and hence $\boxtimes$ is an almost geometric tt-structure on $\perf X \oplus \perf X$. On the other hand, it is not geometric since $\perf X \not \simeq \perf (X \sqcup X)$. Now, by construction we see that
            \[
            \spec_\boxtimes \perf (X\sqcup X) \subset \spec_{\tens_{\ecal O_{X\sqcup X}}^\bb L}  \perf (X\sqcup X)
            \]
            is the inclusion of $X$ into the first component of $X \sqcup X$. Hence, although this example shows $\gtt \perf (X \sqcup X) \subsetneq \agtt \perf(X \sqcup X)$, it does \textbf{not} imply 
            \[
            \spc^\fm \perf(X \sqcup X) \subsetneq \spc^{\ag} \perf(X \sqcup X). 
            \]  
            \item We see that the inclusion $\spc^\fm \cal  T \subset \spc^\pg \cal T$ can be strict even when $\FM \cal T \neq \emp$. Let $X = \bb P^1$. First, it is well-known that $\perf X = \perf k \sf{Kr}$, where $k\sf{Kr}$ is the path algebra of the Kronecker quiver $\sf{Kr}$. Thus, we can consider the tt-category $(\perf X, \tens_{\sf{rep}}^\bb L)$. By Theorem \ref{quiver tt} we have $\spec_{\tens_\sf{rep}^\bb L} \perf X = \spec k \sqcup \spec k$, where those two points correspond to prime thick subcategories $\bra{\ecal O_X}$ and $\bra{\ecal O_X(1)}$ and hence $\tens_\sf{rep}^\bb L$ is a pseudo-geometric tt-structure and $\spc^\ser  \perf X \cap \spc_{\tens_\sf{rep}^\bb L} \perf X = \emp$ by Example \ref{basic example}. Thus, we have 
            \[
            \spc^\ser  \perf X = \spc^\fm \perf X \subsetneq \spc^{\pg} \perf X. 
            \]  
            By transporting the tt-structure $\tens_\sf{rep}^\bb L$ under autoequivalences coming from $\pic(X) = \bb Z$, we see that $$\spc^\pg \perf X = \spc_\vartriangle \perf X.$$ In particular, $\spc^\pg \perf X\setminus \spc^\fm \perf X = \bb Z$. By similar procedures, we expect any variety $X$ with full exceptional collections has
            \[
            \spc^\fm \perf X \subsetneq \spc^{\pg} \perf X 
            \]
            (cf. \cite{SirLiu13}*{\href{https://www.pure.ed.ac.uk/ws/portalfiles/portal/16965040/Recovering_Quivers_from_Derived_Quiver_Representations.pdf}{(2.1.6)}}). However, we do not know if $\spc^\pg \perf X = \spc_\vartriangle \perf X$ in general. 
            \item Let $X$ be a smooth projective variety with ample (anti-)canonical bundle and suppose $X \in \FM \cal T$. In this case, we can see almost geometric tt-structures behave similarly as geometric tt-structures in the following sense. Suppose $\tens$ is an almost geometric structure on $\cal T$ with $\spec_\tens \cal T \iso X$ with unit $\eta(\ecal O_X)$ for some $\eta:\perf X \simeq \cal T$. Then, by \cite{castro2023spaces}*{\href{https://hal.science/tel-04018741/document}{Theorem 3.0.23}}, we have that $\eta$ is monoidal on objects, i.e., for any $\ecal F,\ecal G \in \perf X$, we have
            \[
            \eta(\ecal F\tens_{\ecal O_X}^\bb L \ecal G) = \eta(\ecal F) \tens \eta(\ecal G). 
            \]
            In particular, $\eta$ induces a bijection $\spc_{\tens}\cal T \iso \spc_{\tens_{\ecal O_X}^\bb L} \perf X$ of the underlying topological spaces. It is worth mentioning that in \cite{castro2023spaces}, Castro extensively considers deformation of tt-structures (up to tt-equivalences) using Davydov-Yetter cohomology, but note that in general geometric equivalence does not imply tt-equivalence and vice versa. \qedhere
        \end{enumerate}
\end{example}
\begin{remark}
    Let $\cal T$ be a triangulated category with $X \in \FM \cal T$. As far as the author knows, there is no explicit example of a non-geometric almost-geometric structure $\tens$ on $\cal T$ with $\spec_\tens \cal T \iso X$  although we should expect such an example since the definition of the structure sheaf on a tt-spectrum is too biased towards algebraic geometry so that tt-spectra of non-geometric tt-structures may be schemes as we have observed. The author would like to thank Paul Balmer for sharing some thoughts on this topic, in particular, Example \ref{p1} (i).  
\end{remark}
Finally, let us provide a motivation to define a Serre-geometric tt-structure:
\begin{lemma}
    Let $\cal T$ be a triangulated category with a Serre functor. We have the following inclusions:
    \[
    \spc^\fm \cal T \subset \spc^\sg \cal T \subset \spc^\ser  \cal T.
    \]
    In particular, if $\spc^\fm \cal T = \spc^\ser  \cal T$, then all of them agree. 
\end{lemma}
\begin{proof}
    By Example \ref{serre functor examples}, we get the first inclusion and by the definition of ideals, we get the second inclusion. 
\end{proof}
\begin{remark}
    Let $\cal T$ be a triangulated category with a Serre functor. Conjecture \ref{main conjecture} implies that all of the inclusions $\spc^\fm \cal T \subset \spc^\sg \cal T \subset \spc^\ser  \cal T$ are the equalities if $\FM \cal T$ is not empty. Here, note both $\spc^\fm \cal T = \spc^\sg \cal T$ and $\spc^\sg \cal T = \spc^\ser  \cal T$ are a priori non-trivial and those considerations may produce counter-examples to some conjectures/questions in this paper. 
\end{remark}
I will wrap up this paper with the following conjecture, which is weaker than the previous conjectures yet more plausible from a smooth and proper dg-categorical point of view:
\begin{conjecture}\label{weaker}
    Let $\cal T$ be a triangulated category with $\FM \cal T \neq \emp$. Then we have
    \[
    \spc^\sf{pFM} \cal T = \spc^\sg \cal T. \qedhere
    \]
\end{conjecture}
\bibliography{bib}

\end{document}